\documentclass[11pt]{amsart}

\usepackage[T1]{fontenc}
\usepackage[utf8]{inputenc}

\usepackage{lmodern}

\usepackage{mathtools,amsfonts,amsthm,mathrsfs,amssymb}
\usepackage{doi}
\usepackage{textcomp}
\mathtoolsset{showonlyrefs}
\usepackage[shortlabels]{enumitem}
\usepackage{pgf,interval}
\intervalconfig{
  soft open fences,
  oo/.style={open},
  oc/.style={open left},
  co/.style={open right},
  cc/.style={},
}
\usepackage{tikz}
\usepackage{bookmark}
\usepackage{hypcap}
\usepackage{microtype}

\newtheorem{theorem}{Theorem}
\newtheorem{lemma}[theorem]{Lemma}
\newtheorem{proposition}[theorem]{Proposition}
\newtheorem{corollary}[theorem]{Corollary}
\newtheorem{conjecture}[theorem]{Conjecture}

\theoremstyle{definition}
\newtheorem*{remark}{Remark}

\newtheorem*{spatialmarkov}{Spatial Markov property}
\newtheorem*{localoccupancy}{Local occupancy}

\newtheorem*{llll}{Lopsided Lov\'{a}sz local lemma}
\newtheorem*{chernoffnegcor}{Chernoff bound for negatively correlated variables}

\theoremstyle{definition}
\newtheorem*{standing}{Standing assumptions}
\newtheorem*{furtherstanding}{Further standing assumptions}

\newcommand{\EE}{\mathbb{E}}
\newcommand{\ZZ}{\mathbb{Z}}
\newcommand{\NN}{\mathbb{N}}

\newcommand{\cE}{\mathcal{E}}
\newcommand{\bF}{\mathbf{F}}
\newcommand{\bI}{\mathbf{I}}
\newcommand{\cI}{\mathcal{I}}

\newcommand{\bA}{\mathbf{A}}
\newcommand{\bB}{\mathbf{B}}

\newcommand{\bU}{\mathbf{U}}

\newcommand{\bS}{\mathbf{S}}
\newcommand{\bX}{\mathbf{X}}
\newcommand{\bY}{\mathbf{Y}}
\newcommand{\bZ}{\mathbf{Z}}
\newcommand{\sH}{\mathscr{H}}
\newcommand{\ba}{\mathbf{a}}
\newcommand{\bb}{\mathbf{b}}

\newcommand{\badeve}{B'}

\newcommand{\Lam}{\Lambda}
\newcommand{\lam}{\lambda}
\newcommand{\Gam}{\Gamma}
\newcommand{\gam}{\gamma}
\renewcommand{\epsilon}{\varepsilon}
\newcommand{\eps}{\varepsilon}
\renewcommand{\Pr}{\mathbb{P}}
\renewcommand{\subset}{\subseteq}
\newcommand{\RR}{\mathbb{R}}
\newcommand{\indicator}[1]{\mathbf{1}_{#1}}
\newcommand{\fn}{K}
\DeclareMathOperator{\dom}{dom}

\DeclareMathOperator{\mad}{mad}

\newlist{enumerata}{enumerate}{1}
\setlist[enumerata]{label=\upshape{(\alph*)}}
\setlist[enumerate]{label=\upshape{(\roman*)}}

\title{Graph structure via local occupancy}
\date{\today}

\author[E.\ Davies]{Ewan Davies}
\address{Department of Computer Science, University of Colorado Boulder, USA}
\email{research@ewandavies.org}
\thanks{(E.\ Davies) The research leading to these results has received funding from the European Research Council under the European Union's Seventh Framework Programme (FP7/2007-2013) / ERC grant agreement \textnumero{} 339109. Part of this work was done while the author was visiting the Simons Institute for the Theory of Computing. Revision of this work was supported in part by NSF grant CCF-2309707.}
\author[R.\ J.\ Kang]{Ross J.\ Kang}
\address{Korteweg--de Vries Institute for Mathematics, University of Amsterdam, Netherlands.}
\email{r.kang@uva.nl}
\thanks{(R.\ J.\ Kang) Supported by a Vidi grant (639.032.614) of the Netherlands Organisation for Scientific Research (NWO). Revision of this work was supported in part by grant OCENW.M20.009 of the Dutch Research Council (NWO) and the Gravitation Programme NETWORKS (024.002.003) of the Dutch Ministry of Education, Culture and Science (OCW).}

\author[F.\ Pirot]{Fran\c{c}ois Pirot}
\address{Equipe GALaC (LRI), Universit\'e Paris--Saclay, France.}
\email{francois.pirot@loria.fr}

\author[J.-S.\ Sereni]{Jean-S\'{e}bastien Sereni}
\address{Centre national de la recherche scientifique, Institut de recherche mathématique avancée, Strasbourg, France.}
\email{jean-sebastien.sereni@cnrs.fr}

\hypersetup{pdfauthor={E. Davies, R. J. Kang, F. Pirot, J.-S. Sereni},pdftitle={Graph structure via local occupancy}}



\begin{document}
\begin{abstract}
The first author together with Jenssen, Perkins and Roberts (2017) showed how local properties of the hard-core model on triangle-free graphs guarantee the existence of large independent sets, of size matching the best-known asymptotics due to Shearer (1983).
The present work strengthens this in two ways: first, by guaranteeing stronger graph structure in terms of colourings through applications of the Lov\'{a}sz local lemma; and second, by extending beyond triangle-free graphs in terms of local sparsity, treating for example graphs of bounded local edge density, of bounded local Hall ratio, and of bounded clique number.
This generalises and improves upon much other earlier work, including that of Shearer (1995), Alon (1996) and Alon, Krivelevich and Sudakov (1999), and more recent results of Molloy (2019), Bernshteyn (2019) and Achlioptas, Iliopoulos and Sinclair (2019).
Our results derive from a common framework built around the hard-core model. It pivots on a property we call local occupancy, giving a clean separation between the methods for deriving graph structure with probabilistic information and verifying the requisite probabilistic information itself.
\end{abstract}

\maketitle

\newpage
\tableofcontents


\section{Introduction}\label{sec:intro}

\subsection{Background and motivation}

The asymptotic estimation of Ramsey numbers is of fundamental importance to combinatorial mathematics~\cite{Ram29,ErSz35,Erd47}. The special case of off-diagonal Ramsey numbers has in itself been critical in the development of probabilistic and extremal combinatorics~\cite{AKS80,AKS81,She83,Kim95}. 
The best known asymptotic upper bound on the off-diagonal Ramsey numbers, of the form~$R(3,k) \le (1+o(1))k^2/\log k$ as~$k\to\infty$, has long remained that of Shearer~\cite{She83}. Let us state a weaker form of his result in terms of the size of a largest independent set in triangle-free graphs of given maximum degree; it has the above Ramsey number bound as a corollary.
\begin{theorem}[Shearer~\cite{She83}]\label{thm:shearer}
For any triangle-free graph $G$ of maximum degree~$\Delta$, the independence number of~$G$ satisfies, as~$\Delta\to\infty$, 
\[ \alpha(G) \ge (1-o(1)) |V(G)|\frac{\log \Delta}{\Delta}. \] 
\end{theorem}
\noindent
Note that Shearer actually showed that the asymptotic bound above holds with average degree in the place of $\Delta$.
This bound is sharp up to an asymptotic factor of~$2$ due to random regular
graphs~\cite{KoMa77,Bol81}, while the Ramsey number corollary is sharp up to an
asymptotic factor of~$4$ by the ultimate outcome of the triangle-free graph
process~\cite{BoKe21,FGM20}.
(See `Note added' for updates in the lower bound.)

With somewhat distinct but also classic origins, cf.~e.g.~\cite{Zyk49,UnDe54}, the chromatic number of triangle-free graphs has been intensively studied from various perspectives for decades. In a breakthrough that improved upon and dramatically simplified a seminal work of Johansson~\cite{Joh96}, Molloy analysed a randomised colouring procedure with entropy compression to show the following; see also~\cite{Ber19,DJKP20}.
\begin{theorem}[Molloy~\cite{Mol19}]\label{thm:molloy}
For any triangle-free graph~$G$ of maximum degree~$\Delta$, the chromatic number of $G$ satisfies $\chi(G) \le (1+o(1)) \Delta/\log \Delta$ as~$\Delta\to\infty$.
\end{theorem}
\noindent
As every colour in a proper colouring induces an independent set, note Molloy's result matches Shearer's bound in the sense that it directly implies the statement in Theorem~\ref{thm:shearer}.

The context of both Theorems~\ref{thm:shearer} and~\ref{thm:molloy}---in particular, their heuristic similarity to the corresponding problems in random graphs---hints that some suitable probabilistic insights could yield a deeper structural understanding of the class of triangle-free graphs or closely related classes.
The present work is devoted to making this intuition concrete through the use, in concert with the Lov\'{a}sz local lemma, of the hard-core model, an elegant family of probability distributions over the independent sets originating in statistical physics.

A precursor to this is the discovery by the first author together with Jenssen, Perkins and Roberts~\cite{DJPR18} that Theorem~\ref{thm:shearer} may, via a local understanding of the hard-core model, be derived through the probabilistic method's most elementary form, i.e.~with a bound on the average size of an independent set; see Subsection~\ref{sub:ideas}. 
Previous work of a subset of the authors with de Joannis de Verclos~\cite{DJKP20} found that nearly the same approach yields a result intermediate to Theorems~\ref{thm:shearer} and~\ref{thm:molloy}, in terms of fractional chromatic number.
Building further upon this, here we have developed a unified framework with which we can handle significantly more general settings, in two different senses. 
First, going beyond triangle-free, we treat graphs with certain neighbourhood sparsity conditions. 
Second, going beyond the independence and chromatic numbers (for which we have also obtained new bounds), we are able to make conclusions not only for occupancy fraction and local fractional colourings (as in~\cite{DJPR18,DJKP20,DJKP21}) but also, importantly, for local list colourings and correspondence colourings (to be discussed more fully later).
Our main result encompasses many results from the earlier literature and improves several of them~\cite{AIS19,AEKS81,AKS81,Alo96,AKS99,BGG18,Ber19,BKNP22,DJKP20,DJKP21,DJPR18,Joh96,Kim95b,Mol19,PS15,She83,She95,Vu02}. 
If one is interested in constructing colourings in polynomial time then our framework can give such algorithms through the use of an algorithmic version of the Lov\'{a}sz local lemma, and we give these details in a companion paper dedicated to this purpose~\cite{DKPS20}.
Here we focus more on the existence of colourings and are not concerned with algorithms. 

In the following subsection we give an informal version of our main result and the applications to various graph classes. 
For clarity, we give these statements purely in terms of chromatic number (so along the lines of Theorem~\ref{thm:molloy}), and present a stronger, more general, and significantly more technical version of our main result in Section~\ref{sec:main}, and describe  subsequent applications in their own sections. 
In Subsection~\ref{sub:intro,local} we introduce the motivation for the technical strengthening of our results, namely local colouring. 

\subsection{Framework and applications in terms of chromatic number}\label{sub:intro,chi}

In order to describe our main result, we introduce the hard-core model and its partition function.
It is also helpful to keep in mind the following basic sequence of inequalities, valid for every graph~$G$:
\begin{align}\label{eq:params,basic}
\omega(G)\le \rho(G) \le \chi(G) \le \Delta(G)+1,
\end{align}
where~$\omega(G)$ is the clique number of~$G$, $\rho(G) \coloneqq
\max\{|V(H)|/\alpha(H)\mid H\subseteq G\}$ is the \emph{Hall ratio} of~$G$ (in which $\alpha(H)$ is the independence number of the graph $H$, the size of a largest independent set in $H$),
and~$\Delta(G)$ is the maximum degree of~$G$.

Given a graph $G$, we write $\cI(G)$ for the collection of independent sets of~$G$ (including the empty set).
Given~$\lam>0$, the \emph{hard-core model on~$G$ at fugacity~$\lam$} is a
probability distribution we denote $\mu_{G,\lam}$ on~$\cI(G)$, where each~$I\in\cI(G)$ occurs with
probability proportional to~$\lam^{|I|}$. Writing~$\bI$ for the random
independent set,
\[
    \forall I\in\cI(G),\quad\Pr_{G,\lam}(\bI=I) = \frac{\lam^{|I|}}{Z_G(\lam)},
\]
where the normalising term in the denominator, $Z_G(\lam)$, is the \emph{partition function}
(or \emph{independence polynomial}), defined to be~$\sum_{I\in\cI(G)} \lam^{|I|}$.
Note that with~$\lam=1$ this becomes the uniform distribution over~$\cI(G)$. 
The \emph{occupancy fraction} of the distribution is~$\EE_{G,\lam}|\bI|/|V(G)|$, which one can easily show is equal to 
\[ \frac{1}{|V(G)|}\frac{\lambda Z'_G(\lambda)}{Z_G(\lambda)}. \] 
Observe that $\alpha(G) \ge \EE_{G,\lam}|\bI|$ by the probabilistic method.
Where no confusion arises, we omit the subscripts $G,\lam$ from the notation $\Pr_{G,\lam}$ and $\EE_{G,\lam}$ for probabilities and expectations, respectively, with respect to the hard-core measure $\mu_{G,\lam}$.

Our results rely critically on being able to guarantee a certain property of the hard-core model we call \emph{local occupancy}.
Given a graph~$G$ and $\lam,\beta,\gam>0$, we say that the hard-core model on~$G$ at fugacity $\lam$ has \emph{local $(\beta,\gam)$-occupancy} if, for each~$u\in V(G)$ and each induced subgraph~$F$ of the subgraph~$G[N(u)]$ induced by the neighbourhood of~$u$, it holds that
\begin{align}\label{eq:main,nonlocal}
\beta\frac{\lam}{1+\lam} \frac{1}{Z_F(\lam)}+ \gam\frac{\lam Z_F'(\lam)}{Z_F(\lam)} \ge 1.
\end{align}
We demonstrate later that, using a few standard properties of the hard-core model, condition~\eqref{eq:main,nonlocal} implies that when $\bI\sim\mu_{G,\lam}$ is distributed according to the hard-core model on $G$ we have 
\[ \beta\Pr(u\in \bI) + \gamma \EE|N(u)\cap \bI| \ge 1. \]
As such, one may consider local occupancy a quantitative statement that, on the level of expectation, either a vertex is likely to be occupied, or many of its neighbours are.

Our results for the chromatic number require a stronger form of~\eqref{eq:main,nonlocal} in which we replace
the expectation term~$\lambda Z_F'(\lam)/Z_F(\lam)$ by the lower-quantile quantity~$p Q_{F,p}(\lam)$ and
strengthen the `balancing' in~\eqref{eq:main,nonlocal} to the condition that at least one of the terms is at
least~$1$.
For a graph~$F$ and~$p\in\interval[oo]{0}{1}$, let
\[
    Q_{F,p}(\lambda)
    =\max\left\{x\in\mathbb Z_{\ge0}:\Pr_{F,\lam}(|\bI|\ge x)\ge p\right\},
\]
where the probability is over~$\bI\sim\mu_{F,\lam}$ distributed
according to the hard-core model on~$F$ at fugacity~$\lam$.  We retain
the shorthand
\[
    Q_F(\lambda)=Q_{F,1/(1+\lambda)}(\lambda)
\]
for the particular quantile used in the original formulation.
Note that~$Q_F(1)$ is indeed the median size of a uniformly random independent set in~$F$. 
The probability~$1/(1+\lam)$ is convenient in several applications,
but allowing it to vary independently of~$\lam$ will be important below.

We say that the hard-core model on~$G$ at fugacity~$\lam$ has
\emph{$p$-chromatic local $(\beta,\gamma)$-occupancy} if, for every
vertex~$u\in V(G)$ and subgraph~$F\subset G[N(u)]$, we have
\begin{align}\label{eq:main,nonlocal,chromatic}
    \beta\frac{\lam}{1+\lam} \frac{1}{Z_F(\lam)}\ge 1
    \qquad\text{ or }\qquad
    \gam p Q_{F,p}(\lam) \ge 1.
\end{align}
The unqualified term \emph{chromatic local $(\beta,\gamma)$-occupancy}
means the special case~$p=1/(1+\lam)$, so this definition extends the
one used above without changing its notation.
Equivalently, \eqref{eq:main,nonlocal,chromatic} demands that, if
$Z_F(\lam)> \beta\frac{\lam}{1+\lam}$, then
$Q_{F,p}(\lam)\ge 1/(p\gam)$.
We note that (for~$\bI\sim\mu_{F,\lam}$)
\[
    \EE_{F,\lam}|\bI|
    \ge Q_{F,p}(\lambda)
        \Pr_{F,\lam}(|\bI|\ge Q_{F,p}(\lambda))
    \ge p Q_{F,p}(\lambda),
\]
which yields
\[
    \frac{\lam Z'_F(\lam)}{Z_F(\lam)}
    =\EE_{F,\lam}|\bI|\ge p Q_{F,p}(\lam),
\]
and hence $p$-chromatic local $(\beta,\gam)$-occupancy
implies  local $(\beta,\gam)$-occupancy  as defined in~\eqref{eq:main,nonlocal}.

Our main framework is a suite of results of increasing technical difficulty showing that local occupancy guarantees graph structure in terms of occupancy fraction, fractional colouring, and list/correspondence chromatic number. 
Here we give only an imprecise version of what we obtain for the chromatic number, deferring the full statements of our results to Section~\ref{sec:main}.

\begin{theorem}\label{thm:main,chi}
Suppose that $G$ is a graph of maximum degree $\Delta$ such that the hard-core model on~$G$ at fugacity~$\lam$ has chromatic local $(\beta,\gam)$-occupancy for some positive reals $\lam,\beta,\gam$.
If, as $\Delta\to\infty$,
\[
    \lam=o(1), \qquad
    \lam\beta=o\left(\frac{\gam\Delta}{\log\Delta}\right), \qquad
    \gam=\Omega\left(\frac{1}{\log\Delta}\right),
\]
and~$G$ satisfies one further technical condition, then the chromatic number of~$G$ satisfies $\chi(G) \le (1+o(1))\gam\Delta$.
\end{theorem}

\noindent
Clearly, the strength of the bound on $\chi(G)$ improves as $\gam$ shrinks, though it is not always easy to grasp the theoretical limit of chromatic local occupancy. 
Through the connection with `standard' local $(\beta,\gam)$-occupancy in~\eqref{eq:main,nonlocal} we find that one can focus on establishing that condition with parameters that minimise $\beta+\gam\Delta$ (for $\lam$ sufficiently small). 
In all our applications, the difference between doing this and establishing chromatic local occupancy is small and the extra technical conditions we require are easily verified.
Thus, Theorem~\ref{thm:main,chi} essentially reduces the problem of bounding the chromatic number to a local analysis of the hard-core model.

As a key example~\cite{DJPR18,DJKP20}, if~$G$ is triangle-free then a relatively simple calculation shows that choosing
\begin{align}\label{eq:betagam,triangle}
    \beta=\frac{\gam{(1+\lam)}^{\frac{1+\lam}{\gam\lam}}}{e\log(1+\lam)} \ \text{ and } \ \gam=\frac{1+\lam}{\lam}\frac{\log(1+\lam)}{1+W(\Delta\log(1+\lam))},
\end{align}
where $W(\cdot)$ is the Lambert $W$-function (see Subsection~\ref{sub:Lambert} for more on this function), suffices for local $(\beta,\gam)$-occupancy in the sense of~\eqref{eq:main,nonlocal}; see Subsection~\ref{subsec:tfreewarump}. 
Moreover, these choices satisfy
\begin{align}
\beta+\gam\Delta=\frac{1+\lam}{\lam} \exp(W(\Delta\log(1+\lam))).
\end{align}
Ignoring for a moment the difference between this and chromatic local occupancy, by the asymptotic properties of the function~$W$ and with $\lam=1/\log \Delta$, we should expect that Theorem~\ref{thm:main,chi} in the case of triangle-free graphs gives the same bound as Theorem~\ref{thm:molloy}. 
In Section~\ref{sec:conc} we discuss the fact that this choice of $(\beta,\gam)$ is essentially optimal for the class of triangle-free graphs~\cite{DJPR18}.
For triangle-free graphs, the stronger chromatic local occupancy condition is not needed to obtain the desired colouring bound; see Subsection~\ref{sub:tfree}. Nevertheless, in Subsection~\ref{subsec:tfreewarump} we verify this stronger condition, with comparable parameters, as a warm-up for more challenging settings.
  
We are interested in generalising the condition of having no triangles---a condition on local density---and applying Theorem~\ref{thm:main,chi} to obtain results similar to Theorem~\ref{thm:molloy}. 
In the first four cases we actually give strict generalisations, in the sense that Theorem~\ref{thm:molloy} with the same leading constant `$1$' is a special case of our results. 
The final setting is the classic case of bounded clique number~$\omega$ which is also a generalisation of triangle-free; however, for this setting our result does not `smoothly' extend the triangle-free case in the sense that Theorem~\ref{thm:molloy} outperforms the special case of $\omega=2$. 
(The difficulty of aligning the cases $\omega=2$ and $\omega>2$ is well recognised and subject to an important conjecture of Ajtai, Erd\H{o}s, Koml\'{o}s and Szemer\'{e}di~\cite{AEKS81}. See `Note added' for an update on this conjecture.)

\smallskip
For our first setting we view a triangle as a cycle and consider the generalisation to arbitrary lengths. 
The following result is a common generalisation of Theorem~\ref{thm:molloy} and Kim's~\cite{Kim95b} colouring bound for graphs of girth~$5$, in the sense that the conclusion is the same in (asymptotic) strength, but under the requirement of a single excluded cycle length that is not too large as a function of the maximum degree. 
We also note that this is a stronger form of one consequence from~\cite[Cor.~2.4]{AKS99}, the same result but with neither a specific leading constant nor any dependence of the cycle length upon maximum degree.

\begin{theorem}\label{thm:Cfree,chi}
  For any graph $G$ of maximum degree~$\Delta$ that contains no subgraph isomorphic to the cycle~$C_k$ of length $k$, where $k= k(\Delta)\ge 3$ satisfies $k =\Delta^{o(1)}$ as $\Delta\to\infty$, the chromatic number of~$G$ satisfies~$\chi(G) \le (1+o(1)) \Delta/\log \Delta$.
\end{theorem}

\noindent
This result follows from Theorem~\ref{thm:main,chi} and a local occupancy analysis given in Subsections~\ref{sub:sparse,hcm} and~\ref{sub:Cfree}.

\smallskip
The second setting is to view a triangle-free graph as one that has no edges in any neighbourhood, and to relax this condition to having `few' edges in any neighbourhood. 
This is the same as having a `bounded local triangle count'.
More precisely, suppose that $G$ is a graph of maximum degree $\Delta$, and for some $f \le \Delta^2+1$ the neighbourhood of every vertex in $G$ spans at most $\Delta^2/f$ edges. 
Note that $f=\Delta^2+1$ corresponds to the triangle-free case.
This problem setting was first considered, in terms of independence number, by Ajtai, Koml\'{o}s and Szemer\'{e}di~\cite{AKS81} and by Shearer~\cite{She83}, and, in terms of chromatic number, by Alon, Krivelevich and Sudakov~\cite{AKS99}, Vu~\cite{Vu02}, and more recently by Achlioptas, Iliopoulos and Sinclair~\cite{AIS19}.
Via a local occupancy analysis and an iterative splitting procedure given in Subsections~\ref{sub:sparse,hcm},~\ref{sub:sparsenbhds} and~\ref{sub:sparsenbhds,proof2}, we show how Theorem~\ref{thm:main,chi} implies the following.

\begin{theorem}\label{thm:sparsenbhds,chi}
For any graph $G$ of maximum degree $\Delta$ in which the neighbourhood of
every vertex in~$G$ spans at most~$\Delta^2/f$ edges, where $f=f(\Delta) \le
\Delta^2+1$, the chromatic number of~$G$ satisfies
$\chi(G)\le(1+o(1))\Delta/\log \sqrt{f}$ as~$f\to\infty$.
\end{theorem}

\noindent
This asymptotically matches the known fractional bound~\cite{DJKP21}, affirms~\cite[Conj.~3]{DJKP21}, and improves the chromatic number bounds in~\cite{AIS19} by about a factor~$2$ in nearly the entire range of rates for~$f$ as a function of~$\Delta$. The bound is sharp up to an asymptotic factor of between~$2$ and~$4$ due to a random regular construction or a suitable blow-up, cf.~\cite{She83,DJKP21}. 
Though the analysis of the hard-core model and selection of suitable $(\beta,\gamma)$ for local occupancy in this setting has essentially already been done~\cite{DJKP21}, here we have a simpler and more general argument from which we deduce Theorems~\ref{thm:Cfree,chi} and~\ref{thm:sparsenbhds,chi}, cf.~Section~\ref{sec:mad}.

\smallskip
In fact, we have been able to naturally combine the settings of Theorems~\ref{thm:Cfree,chi} and~\ref{thm:sparsenbhds,chi} and give a result for graphs with the property that each vertex is contained in few $k$-cycles. 
We actually only require the weaker condition that for each~$u\in V(G)$, the induced neighbourhood subgraph~$G[N(u)]$ contains at most~$t$ copies of the path~$P_{k-1}$ on~$k-1$ vertices. 

\begin{theorem}\label{thm:boundedC,chi}
For any graph~$G$ of maximum degree~$\Delta$ in which the subgraph induced by each neighbourhood contains at most~$t$ copies of~$P_{k-1}$, where $t=t(\Delta)\ge 1/2$ and $k=k(\Delta)\ge 3$ satisfy $k+\sqrt{t}=\Delta^{o(1)}$ as~$\Delta\to\infty$, the chromatic number of~$G$ satisfies
\[
\chi(G)\le (1+o(1))\frac{\Delta}{\log(\Delta /(k+\sqrt{t}))}.
\]
\end{theorem}

\noindent
An earlier version stated this theorem under the weaker hypothesis
$\Delta/(k+\sqrt t)\to\infty$. The stronger hypothesis above is used
to verify the large-partition-function condition~\eqref{eq:Zlarge}
in the corrected correspondence-colouring argument.

\noindent
The case $k=3$ is covered more generally by Theorem~\ref{thm:sparsenbhds,chi}.  We outline the relevant local occupancy analysis in Subsection~\ref{sub:boundedC}. 
The algorithmic question of whether the colourings guaranteed by Theorems~\ref{thm:Cfree,chi}--\ref{thm:boundedC,chi} can be constructed in polynomial time is tackled in our companion paper~\cite{DKPS20}. 
In short (and in the setting of list colouring), for constant $k$ and small enough $t$ such constructions follow from our framework given some significant additional work.

\smallskip
For our fourth setting, consider the situation where~$G$ has some prescribed local independent set structure. More concretely, suppose that there is some~$\rho\ge 1$ such that every neighbourhood induces a subgraph of Hall ratio at most~$\rho$. Note that by~\eqref{eq:params,basic} taking~$\rho=1$ corresponds to the triangle-free case. 
Also by~\eqref{eq:params,basic}, we see that this condition is satisfied if every neighbourhood in $G$ induces a subgraph of chromatic number at most~$\rho$. 
Motivated by the corresponding problem for bounded clique number (which we will discuss further shortly), Alon~\cite{Alo96} implicitly considered this problem setting and gave an upper bound on the independence number of such graphs~$G$. 
It was also considered by Johansson~\cite{Joh96a} and Molloy~\cite{Mol19} in the context of chromatic number, and more recently by Bonamy et al.~\cite{BKNP22}. 
By optimising a local analysis of the hard-core model (given in Subsection~\ref{sub:entropy,hcm} and Section~\ref{sec:localhall}) and then applying Theorem~\ref{thm:main,chi}, we obtain the following (which we note also implies Theorem~\ref{thm:molloy}).

\begin{theorem}\label{thm:localhall,chi}
    There is a monotone increasing function $\fn\colon\interval[co]{1}{\infty}\to\interval[co]{1}{\infty}$ satisfying $\fn(1)=1$ and $\fn(\rho)=(1+o(1))\log \rho$ as $\rho\to\infty$ such that the following holds.
For any~$\rho\ge 1$ and graph~$G$ of maximum degree~$\Delta$ in which the neighbourhood of
every vertex $u\in V(G)$ induces a subgraph of Hall ratio at most~$\rho$, the chromatic number of~$G$ satisfies
$\chi(G) \le (\fn(\rho)+o(1))\Delta/\log \Delta$ as~$\Delta\to\infty$.
\end{theorem}

\noindent
Again by the random regular graph, this bound is of the correct asymptotic order.
For details about the parameter~$\fn$, which improves upon all earlier leading
asymptotic constants, see Subsection~\ref{sub:Lambert}.  It is possible that
the bound in Theorem~\ref{thm:localhall,chi} is correct up to a multiplicative
constant independent of~$\rho$, but it is unclear to us how to devise a
construction certifying this.

\smallskip
Finally, suppose that~$G$ has bounded clique number~$\omega$ so that $\omega=2$ corresponds to the triangle-free case. 
Equivalently, assume that all neighbourhoods induce subgraphs which are $K_{\omega}$-free. 
This classic setting (cf.~\cite{AEKS81}) is considered a central Ramsey-type problem.
Our contribution in this setting (via Theorem~\ref{thm:main,chi} and the local hard-core analysis given in Subsection~\ref{sub:clique,hcm}) is the following result. 

\begin{theorem}\label{thm:clique,chi}
For any graph~$G$ of clique number $\omega\ge3$ and maximum degree~$\Delta$, the chromatic number of~$G$ satisfies as~$\Delta\to\infty$
\begin{align}
\chi(G) \le (1+o(1))\min\left\{(\omega-2)\frac{\Delta\log \log \Delta}{\log\Delta},5\Delta\sqrt{\frac{\log(\omega-1)}{\log\Delta}}\right\}.
\end{align}
\end{theorem}

\noindent
Ignoring the leading constants for a moment, the first bound is a bound of
Johansson~\cite{Joh96a} (cf.~also~\cite{PS15,BGG18,Mol19,Ber19,BKNP22}), inspired by
the corresponding bound of Shearer~\cite{She95}
for independence number; and the second is a more recent bound of Bonamy, Kelly, Nelson and Postle~\cite{BKNP22}, foreshadowed by an independence number bound of Bansal, Gupta and Guruganesh~\cite{BGG18}.
Theorem~\ref{thm:clique,chi} improves on all of these earlier bounds by some constant factor.
Keep in mind that for the first bound in terms of independence number, Ajtai, Erd\H{o}s, Koml\'{o}s and Szemer\'{e}di~\cite{AEKS81} have famously conjectured a better asymptotic \emph{order}: that the~$\log\log\Delta$ factor is unnecessary. (See `Note added' for an update on this conjecture.)

\smallskip
Just as for Theorem~\ref{thm:molloy} and its predecessor in Johansson's theorem~\cite{Joh96a}, the proof of Theorem~\ref{thm:main,chi} uses a randomised (list) colouring procedure and suitable applications of the Lov\'{a}sz local lemma to show that sufficiently distant events in the graph are close to independent.
Our approach in fact builds upon the works of Molloy~\cite{Mol19}, Bernshteyn~\cite{Ber19}, and some subsequent studies~\cite{BKNP22,DJKP20}. In a nutshell, we have incorporated the hard-core model into the earlier proof method, where previous work had focused on the special case~$\lam=1$ of uniformly chosen independent sets (or partial proper colourings). We present and prove the full, general version of Theorem~\ref{thm:main,chi} in Section~\ref{sec:main}. An important feature of our framework is that the easier parts of it---essentially, for independence number---allow one to have a preview for the possible stronger results from the harder parts---essentially, for chromatic number---as they share a similar dependence on the local occupancy properties of the hard-core model.

\subsection{Framework in terms of local colouring}\label{sub:intro,local}

As mentioned, the bounds in Theorems~\ref{thm:shearer},~\ref{thm:molloy},~\ref{thm:Cfree,chi}--\ref{thm:localhall,chi} are each tight up to some constant factor (independent of~$|V(G)|$ and~$\Delta$), and this is due to some probabilistic constructions that have all vertex degrees equal. 
If the (hypothetical) extremal examples are indeed regular or having all degrees asymptotically equal, it would intuitively suggest that vertices of maximum degree are of primary importance for such problems. 
One might naturally wonder if vertices of lower degree are `easier' to colour in a quantifiable sense. This is the idea behind local colouring. 
It has its roots in degree-choosability as considered by Erd\H{o}s, Rubin and Taylor~\cite{ERT80}, and a more
systematic study was carried out in~\cite{BKNP22}, cf.~also~\cite{Kin09thesis,DJKP20}.  It turns out that
local analysis of the hard-core model lends itself well to producing local colourings, in part because our
framework easily incorporates a `more local' version of the conditions~\eqref{eq:main,nonlocal}
and~\eqref{eq:main,nonlocal,chromatic}.

\begin{localoccupancy}
Given a graph~$G$, a positive real~$\lam$, and sequences $\beta,\gam\in(\RR^+)^{V(G)}$ of positive reals indexed by the vertices of~$G$, we say that the hard-core model on $G$ at fugacity $\lam$ has \emph{local $(\beta,\gam)$-occupancy} if, for each~$u\in V(G)$ and each subgraph~$F$ of the subgraph~$G[N(u)]$ induced by the neighbourhood of~$u$, it holds that
\begin{align}\label{eq:main}
\beta_u\frac{\lam}{1+\lam} \frac{1}{Z_F(\lam)}+ \gam_u\frac{\lam Z_F'(\lam)}{Z_F(\lam)} \ge 1.
\end{align}
Since this condition is required for all subgraphs~$F$, not only induced
ones, we also refer to it as \emph{strong local $(\beta,\gam)$-occupancy}.
For $p\in\interval{0}{1}$, we say that the hard-core model on~$G$ at fugacity
$\lam$ has \emph{$p$-chromatic local $(\beta,\gam)$-occupancy} if, for
each~$u\in V(G)$ and each subgraph~$F$ of the subgraph~$G[N(u)]$
induced by the neighbourhood of~$u$, it holds that
\begin{align}\label{eq:main,chromatic}
    \beta_u\frac{\lam}{1+\lam}\frac{1}{Z_F(\lam)} \ge 1
    \qquad\text{ or }\qquad
    \gam_u p Q_{F,p}(\lam) \ge 1.
\end{align}
As above, \emph{chromatic local $(\beta,\gam)$-occupancy} without a
prefix means the special case~$p=1/(1+\lam)$.
\end{localoccupancy}

Our main theorem, Theorem~\ref{thm:main,chic} stated in Section~\ref{sec:main}, is an explicit, strong form of
Theorem~\ref{thm:main,chi} above which says that under~\eqref{eq:main,chromatic} and some mild additional
conditions the graph can be properly coloured so that every vertex $u$ of not too small a degree is coloured
from a list of size not much larger than~$\beta_u + \gam_ud_u$. 
It turns out that it is little extra work to expand local occupancy in the sense of~\eqref{eq:main,nonlocal}, with~$\beta$ and~$\gam$ depending on the maximum degree~$\Delta$, to local occupancy in the sense of~\eqref{eq:main}, instead with~$\beta_u$ and~$\gam_u$ having a dependence on~$s_u$, for \emph{any} sequence~$s\in(\RR^+)^{V(G)}$ of positive reals.
As a result, for example, a refinement of the local analysis that led to the choices in~\eqref{eq:betagam,triangle} (see Section~\ref{sec:mad} or~\ref{sec:localhall}) leads to the following local version of Theorem~\ref{thm:molloy}, cf.~\cite{DJKP20}.

\begin{theorem}\label{thm:molloy,local}
For any~$\eps>0$ there exist~$\delta_0$ and~$\Delta_0$ such that the following holds for
all~$\Delta\ge\Delta_0$. Any triangle-free graph of maximum degree~$\Delta$,
writing~$d_u^+=\max\{d_u,\delta_0\log\Delta\}$, admits a proper colouring in which each vertex~$u$ is
coloured from 
\[ \left\{ 1,\dots,\left\lceil (1+\eps)\frac{d_u^+}{\log(d_u^+/\log \Delta)} \right\rceil \right\}.\]
\end{theorem}

\noindent
Analogous refinements of Theorems~\ref{thm:Cfree,chi}--\ref{thm:clique,chi} also hold; we defer their precise statements to later sections as we wish to also consider the settings of list and correspondence colouring which require further definitions.

One should take notice of the minimum list size condition encoded by~$d_u^+$ in Theorem~\ref{thm:molloy,local} as well as in later statements.
We remark that if~$G$ is of minimum degree at least~$\delta_0\log\Delta$ then the list sizes are truly local, and one can interchangeably think of a minimum list size condition as we state, both here and later, or a condition on the minimum degree of~$G$.

Let us also point out that the minimum list size condition in Theorem~\ref{thm:molloy,local} has a modest
dependence on $\Delta$. 
As already shown~\cite{DJKP20} and as discussed later here, the corresponding condition for
a local \emph{fractional} colouring result needs no such dependence. Theorem~\ref{thm:molloy,local}  follows from a substantially
stronger \emph{list} colouring version. This version does demand, and indeed
needs, a minimum list size depending on~$\Delta$, which helps in the proof mainly for concentration
considerations. 
We discuss this issue further in Subsection~\ref{sub:local}.

\subsection{Organisation}

In Section~\ref{sec:prelim}, we give a quick overview of some terminology as well as some basic and guiding principles/tools from graph theory and probability.
In Section~\ref{sec:main}, we fully present and prove our main framework for producing global graph structure from local occupancy properties of the hard-core model.
In Section~\ref{sec:furtherprelim}, we give some further preliminary and general results needed to carry out the specific hard-core analyses for our applications.
In Sections~\ref{sec:mad}--\ref{sec:clique}, we perform several local occupancy analyses to prove Theorems~\ref{thm:Cfree,chi}--\ref{thm:clique,chi} within our framework.
In Section~\ref{sec:conc}, we give some concluding remarks as well as some hints for further study.

\section{Notation and preliminaries}\label{sec:prelim}

\subsection{Graph structure}\label{sub:structure}

Let $G$ be a graph. We write $V(G)$ and $E(G)$ (or just $V$ and $E$) for the vertex and edge sets of $G$, respectively.
Given~$u\in V(G)$, we write~$N_G(u)$ for the (open) neighbourhood of~$u$,
$d_u=|N_G(u)|$ for the degree of~$u$, and $N_G[u] = \{u\}\cup N_G(u)$ for
the closed neighbourhood of~$u$. 
In all these, we often drop the subscript if it is unambiguous.

We write~$\cI(G)$ for the collection of \emph{independent sets}, i.e.~vertex subsets inducing edgeless subgraphs, of~$G$, and the \emph{independence number~$\alpha(G)$}
of~$G$ is the size of a largest set from~$\cI(G)$. A \emph{proper $k$-colouring}
of~$G$ is a partition of~$V(G)$ into~$k$ sets from~$\cI(G)$, or equivalently a
mapping $c\colon V(G)\to [k]$ such that $c^{-1}(\{i\})\in \cI(G)$ for every~$i\in[k]$,
and the \emph{chromatic number~$\chi(G)$} of~$G$ is the least~$k$ for which~$G$
admits a proper $k$-colouring.

Given a probability distribution over~$\cI(G)$, writing~$\bI$ for the random
independent set, the \emph{occupancy fraction} of the distribution
is~$\EE|\bI|/|V(G)|$.  As noted earlier, when the distribution for~$\bI$ is the
hard-core model at fugacity $\lam>0$, the occupancy fraction may be rewritten
\[
\frac{\EE |\bI|}{|V(G)|} = \frac{\lam Z'_G(\lam)}{Z_G(\lam)|V(G)|}.
\]
Note again that $\alpha(G) \ge \EE|\bI|$.

We (first) define fractional colouring in probabilistic terms, in terms of
uniform occupancy. A \emph{fractional $k$-colouring} of~$G$ is a probability
distribution~$\bI$ over~$\cI(G)$ such that, writing $\bI$ for the random
independent set, it holds that $\Pr(v\in \bI) \ge 1/k$ for every vertex $v\in
V(G)$. The \emph{fractional chromatic number}~$\chi_f(G)$ is the least~$k$ for
which $G$ admits a fractional $k$-colouring. Note that a fractional
$k$-colouring of~$G$ has occupancy fraction at least~$1/k$, and so $\alpha(G)
\ge |V(G)|/\chi_f(G)$. One can also see fractional colouring as a relaxation of usual
colouring in that $\chi_f(G)\le \chi(G)$, which follows by defining the
fractional $k$-colouring arising from selecting uniformly at random
one of the independent sets associated to the colours of
a proper $k$-colouring of~$G$.  

One equivalent, but maybe more concrete, definition of a fractional $k$-colouring
of~$G$ is an assignment~$w$ of pairwise disjoint intervals
contained in $\interval[co]{0}{k}$ to independent sets of~$G$ such that $\sum_{I\in\cI(G), I \ni
v}|w(I)|= 1$  for all~$v\in V(G)$.  Such a colouring naturally induces an
assignment of subsets (of measure~$1$) to the vertices of~$G$, namely~$c(v)
=\bigcup_{I\in\cI(G)\,:\,I \ni v} w(I)$ for each~$v\in V(G)$, such that~$c(u)$
and~$c(v)$ are disjoint whenever $uv\in E(G)$. 
The equivalence of notions of fractional colouring that may superficially appear quite different is studied in e.g.~\cite{DSV14}.

In another direction, we also consider structural parameters that are significantly stronger than the chromatic number.
A mapping $L\colon V(G)\to \binom{\ZZ^+}{k}$ is called a
\emph{$k$-list-assignment} of~$G$; a colouring $c\colon V(G)\to \ZZ^+$ is
called an \emph{$L$-colouring} if $c(v)\in L(v)$ for any $v\in V(G)$. We say that $G$
is \emph{$k$-choosable} if for any $k$-list-assignment $L$ of $G$ there is a
proper $L$-colouring of~$G$. The \emph{list chromatic number} (or \emph{choosability}) $\chi_\ell(G)$ of~$G$ is the least~$k$ such that~$G$ is
$k$-choosable.  By a constant $k$-list-assignment,~$G$ admits a proper
$k$-colouring if it is $k$-choosable.

The list chromatic number is a classic colouring parameter,
cf.~e.g.~\cite{Alo93}; however, we are also able to extend our framework to the
newer, but even stronger notion of correspondence chromatic number introduced
by Dvořák and Postle~\cite{DP18}. We mainly adopt the notation of
Bernshteyn~\cite{Ber19}.
Given a graph~$G$, a \emph{cover} of~$G$ is a pair~$\sH = (L, H)$, consisting of a graph~$H$ and a mapping~$L \colon V(G) \to 2^{V(H)}$, satisfying the following requirements:
\begin{enumerate}
    \item the sets $\{L(u) \,:\,u \in V(G)\}$ form a partition of $V(H)$;
    \item for every $u \in V(G)$, the graph $H[L(u)]$ is complete;
    \item if $E_H(L(u), L(v)) \neq \varnothing$, then either $u = v$ or $uv \in E(G)$;
    \item\label{item:matching} if $uv \in E(G)$, then $E_H(L(u), L(v))$ is a matching (possibly empty).
\end{enumerate}
An \emph{$\sH$-colouring} of~$G$ is an independent set in~$H$ of size~$|V(G)|$.
For a sequence~$k\in (\RR^+)^{V(G)}$ we say that $G$ is \emph{$k$-correspondence colourable} if~$G$ admits an $\sH$-colouring for every cover~$\sH = (L, H)$ with~$|L(u)| \ge k_u$ for all~$u \in V(G)$.
We call such covers \emph{$k$-fold}.
The \emph{correspondence chromatic number} (or \emph{DP-chromatic
number})~$\chi_c(G)$ is the least integer~$k$ for which~$G$ admits an
$\sH$-colouring for every $k$-fold cover $\sH = (L, H)$ of~$G$.

Note that every $k$-list-assignment~$L$ translates into a $k$-fold cover (hence the common lettering).
Given any list assignment~$L$ of~$G$, one forms an equivalent cover~$\sH = (\hat L, H)$ of~$G$ by setting
$\hat L(v) = \{(v, c) : c\in L(v)\}$ for each~$v\in V(G)$ and forming the graph~$H$ on vertex set
$\bigcup_{v\in V(G)}\hat L(v)$ by putting a clique on each~$\hat L(v)$ and adding the edges connecting~$(u,
c)$ and~$(v, c)$ for each edge~$uv\in E(G)$ and colour~$c\in L(u)\cap L(v)$.
Given this, an independent set in~$H$ of size~$|V(G)|$ corresponds to a proper $L$-colouring of~$G$. This
implies that~$\chi_\ell(G) \le \chi_c(G)$.
We say that a cover of~$G$ that can be constructed from a list assignment in this fashion is a cover that
\emph{arises from a list assignment}. 
What holds for such covers but does not hold in general is a transitivity property that follows from the
``global'' identity for each colour. 
For example, in a cover~$\sH=(L, H)$ of~$G$ arising from a list assignment, if $\{u,v,w\}$ induces a triangle
in~$G$ such that there is some~$x\in L(u)$ with neighbors~$y\in L(v)$ and~$z\in L(w)$ in~$H$, then~$yz\in
E(H)$. 
This transitivity property implies that~$N_H(x)\setminus L(u)$ induces a subgraph of~$H$ isomorphic to an
\emph{induced} subgraph of~$G[N(u)]$, whereas in a general cover~$N_H(x)\setminus L(u)$ can induce an
arbitrary subgraph of~$G[N(u)]$.

Note that, as for (list) colouring, this stronger parameter remains amenable to inductive approaches. For instance, it can be related to the following density
parameter through a greedy colouring procedure.  Writing~$\delta(H)$ for the
minimum degree of a graph~$H$, the \emph{degeneracy}~$\delta^*(G)$ of~$G$ is~$\max_{H\subseteq G} \delta(H)$. 
The degeneracy of~$G$ is clearly bounded by the maximum degree~$\Delta(G)$ of~$G$.
It is worth noting that we will also make use of the closely related density notion of \emph{maximum average
degree}~$\mad(G)$ of~$G$, defined to be
\[
\mad(G) = \max_{\substack{H\subseteq G\\|V(H)|\ge 1}} \frac{2|E(H)|}{|V(H)|},
\] 
and which satisfies $\delta^*(G)\le \mad(G)<2\delta^*(G)$.

Synthesising the above discussion, here is an enlarged version of~\eqref{eq:params,basic}:
\begin{align}\label{eq:params}
    \omega(G)\le \rho(G) &\le \chi_f(G) \le \chi(G)  \le \chi_\ell(G)\\
   &\le \chi_c(G) \le \delta^*(G)+1\le \mad(G)+1\le \Delta(G)+1.
\end{align}

In this work, we focus on upper bounds for the second to the sixth parameters
in~\eqref{eq:params}, especially in graph classes defined according to local
structural conditions (which in turn are often defined with parameters lying along the above sequence of
inequalities).

\smallskip
For perspective, we next make a few general remarks regarding strictness of
the inequalities in~\eqref{eq:params}. 
For the first, we have, in relation to the sharpness of
Theorem~\ref{thm:shearer}, already mentioned the existence of a family of
graphs with clique number~$2$ over which~$\rho$ is unbounded.
For the second, Dvořák, Ossona de
Mendez, and Wu~\cite{DOW20} (cf.~also~\cite{BLM+22}) showed the existence of a
family of graphs with Hall ratios at most~$18$ and unbounded fractional
chromatic numbers.  It remains an interesting open problem of
Harris~\cite{Har19} to determine whether~$\rho(G)$ is always at least some
constant fraction of~$\chi_f(G)$ for triangle-free graphs~$G$.
(See `Note added' for an update on this problem.)
For the third inequality in~\eqref{eq:params}, it is well known that the
Kneser graphs~$KG_{n,k}$ (which are triangle-free if $n<3k$) satisfy
$\chi(KG_{n,k})=n-2k+2$ and $\chi_f(KG_{n,k})=n/k$, cf.~Lov\'{a}sz~\cite{Lov78}.
For the fourth to sixth inequalities, the complete $d$-regular bipartite
graph~$K_{d,d}$ satisfies $\chi(K_{d,d})=2$, $\chi_\ell(K_{d,d})=(1+o(1))\log_2
d$~\cite{ERT80}, $\chi_c(K_{d,d})=(1/2+o(1))d/\log d$~\cite{Ber16},
and~$\delta^*(K_{d,d})=d=\mad(G)$ (as~$d\to\infty$). 

\smallskip
As described in Subsection~\ref{sub:intro,local}, we in fact prove our results in
the context of more refined versions of $\chi_f$, $\chi$, $\chi_\ell$,
and $\chi_c$, by individually restricting the `amount' of colour available per
vertex. For $\chi_\ell$ and $\chi_c$, we do so by allowing the lower bound
condition on~$|L(v)|$ to vary as a function of parameters of~$N(v)$, such as
the degree~$\deg(v)$ of the vertex~$v$. On the other hand, for~$\chi_f$ and~$\chi$ we do so by moreover demanding (as in Subsection~\ref{sub:intro,local})
that the colour or set of colours assigned to~$v$ be chosen only from an
interval (of length depending on~$\deg(v)$) whose left endpoint is at the
origin.  For the local colouring results we prove in the former situation, we
invariably assume some mild but uniform
minimum list size that is defined in terms of the maximum degree of the graph.
We discuss this further in Subsection~\ref{sub:local}.

\subsection{Probabilistic tools}\label{sub:prob}

Given a probability space, the $\{0,1\}$-valued random variables~$\bX_1,\dotsc,\bX_n$ are
\emph{negatively correlated} if for each subset~$S$ of~$\{1,\dotsc,n\}$,
\[
    \Pr\big(\bX_i=1, \forall i\in S\big)\le\prod_{i\in S}\Pr(\bX_i=1).
\]
Note that independent random variables are negatively correlated.

\begin{chernoffnegcor}[\cite{PS97}]
    Given a probability space, let~$\bX_1, \dotsc, \bX_n$ be $\{0,1\}$-valued random variables.
    Set~$\bX = \sum_{i=1}^{n}\bX_i$ and~$\bY_i = 1-\bX_i$ for each~$i\in\{1,\dotsc,n\}$.
If the variables~$\bY_1,\dotsc,\bY_n$ are negatively correlated, then
\[
  \forall \eta\in\interval[oo]{0}{1},\quad \Pr\big(\bX\le(1-\eta)\EE\bX\big) \le e^{-\eta^2\EE\bX/2}.
\]
\end{chernoffnegcor}

\begin{llll}[\cite{ErLo75}, cf.~\cite{AS16book}]
Given a probability space, let~$\cE=\{A_1,\dots,A_n\}$ be a set of (bad) events.
Suppose for each~$i\in[n]$ there is some~$\Gam(i)\subseteq[n]$ such that~$|\Gam(i)|\le d$ and for all~$S\subseteq [n]\setminus\Gam(i)$ it holds that
\[
\Pr\left(A_i \mathrel{\Bigg|} \bigcap_{j\in S} \overline{A_j}\right) \le p.
\]
If~$4pd\le 1$, then the probability that none of the events in~$\cE$ occur is positive.
\end{llll}


\section{The framework}\label{sec:main}

In this section we describe our framework in better detail now that we are equipped with the requisite notation from Subsection~\ref{sub:structure}. This continues the discussion we began in Subsections~\ref{sub:intro,chi} and~\ref{sub:intro,local}.

\subsection{The general theorems}\label{sub:theorems}

The condition~\eqref{eq:main,nonlocal} of local occupancy is very close to a lower bound on the occupancy fraction of the hard-core model in $G$, and the following result has an elementary proof following from basic properties of the hard-core model, see Subsection~\ref{sub:ideas}. 

\begin{theorem}[cf.~\cite{DJKP20,DJKP21,DJPR17,DJPR18}]\label{thm:main,occfrac}
Suppose $G$ is a graph of maximum degree~$\Delta$ such that the hard-core model on~$G$ at fugacity~$\lam$ has local $(\beta,\gam)$-occupancy for some $\lam,\beta,\gam>0$. Then its occupancy fraction satisfies
\[
\frac{1}{|V(G)|}\frac{\lam Z'_G(\lam)}{Z_G(\lam)} \ge \frac{1}{\beta+\gam\Delta}.
\]
In particular, the independence number of $G$ satisfies
\[
\alpha(G) \ge \frac{|V(G)|}{\beta+\gam\Delta}.
\]
\end{theorem}

A crucial realisation, made essentially in~\cite{DJKP20}, is that the same argument in conjunction with a greedy colouring procedure~\cite[Lem.~3]{DJKP20}, cf.~also~\cite[Ch.~21]{MoRe02}, leads to the fractional relaxation of Theorem~\ref{thm:molloy} (and any other setting in which we can prove local occupancy). 
Moreover, under the condition in~\eqref{eq:main} instead of~\eqref{eq:main,nonlocal}, it permits a clean local formulation. We restate~\cite[Lem.~3]{DJKP20} for completeness.

\begin{lemma}\label{lem:chifalg}
Fix a positive integer $r$.
Let $G$ be a graph and suppose that for every vertex $v\in V(G)$ we have a list $(\alpha_j(v))_{j=0}^r$ of $r+1$ real numbers.
Suppose that for all induced subgraphs $H$ of $G$, there is a probability distribution on $\cI(H)$ such that, writing $\bI_H$ for the random independent set from this distribution, for each $v\in V(H)$ we have the bound
\[\sum_{j=0}^r \alpha_j(v)\EE\big|N_H^j(v)\cap \bI_H\big| \geq 1.\]
Then there exists a fractional colouring of $G$ such that every $v\in V(G)$ is coloured with a subset of the interval
$\big[0,\sum_{j=0}^r \alpha_j(v)\big|N_G^j(v)\big|\big)$.
\end{lemma}
This is not fractional colouring with local demands in the sense of Kelly
and Postle~\cite{KePo18+}: here the available interval, rather than the
demanded measure, varies with the vertex.

The arguments given below in Subsection~\ref{sub:ideas} show that when $G$ satisfies a local occupancy condition, one can apply the case $r=1$ of the above result with the hard-core model giving each of the necessary probability distributions. 
This gives the fractional colouring part of our framework.

\begin{theorem}\label{thm:main,chif}
    Suppose $G$ is a graph such that the hard-core model on~$G$ at fugacity~$\lam$ has local $(\beta,\gam)$-occupancy for some~$\lam>0$ and some sequences $\beta,\gam\in(\RR^+)^{V(G)}$ of positive reals.
Then $G$ admits a fractional colouring in which each vertex $u$ is coloured with a subset of the interval
\[
\interval[co]{0}{\beta_u+\gam_ud_u}.
\]
In particular, the fractional chromatic number of $G$ satisfies
\[
\chi_f(G)\le \beta^*+\gam^*\Delta,
\] 
where $\beta^*=\max_u\beta_u$, $\gam^*=\max_u\gam_u$, and~$\Delta$ is the maximum degree of~$G$.
\end{theorem}

\noindent
The importance of this realisation is that it hints at a generalisation via the hard-core model and, simultaneously, an extension along~\eqref{eq:params} of Theorem~\ref{thm:shearer}.
In this work we confirm both of these subject to some mild additional conditions.
The following is our main result and the list and correspondence colouring generalisation of Theorem~\ref{thm:main,chi}. 

\begin{theorem}\label{thm:main,chic}
    Suppose that~$G$ is a graph of maximum degree~$\Delta\ge 2$ such
    that for~$\lam>0$, $p\in\interval[oo]{0}{1}$, and sequences~$\beta,\gam\in(\RR^+)^{V(G)}$ of positive
    reals, the hard-core model on~$G$ at fugacity~$\lam$ has $p$-chromatic local $(\beta,\gam)$-occupancy.
    Suppose also that for some~$\iota\in\interval[oo]{0}{1}$ and
    \begin{align}
        \ell&\ge \frac{2}{\iota^2}\log(16\Delta^3)\label{eq:elllarge}
    \end{align} 
    we have for every $u\in V(G)$ and every subgraph $F$ of~$G[N(u)]$ that
    \begin{align}
        d_u\gam_u &\ge \frac{2}{\iota^2p}\log(16\Delta^3),\label{eq:ysmall}\\
        \shortintertext{and}
        |V(F)|&\ge \ell/8 \Longrightarrow Z_F(\lam) \ge 8\Delta^4\label{eq:Zlarge}.
    \end{align}
    Then $G$ is $k$-correspondence colourable for any sequence $k\in(\RR^+)^{V(G)}$ such that
    \begin{align}
        k_u \ge \frac{1}{1-\iota}\left(\frac{\lam\ell}{1+\lam}\beta_u + \gam_ud_u\right).\label{eq:klarge}
    \end{align} 
\end{theorem}
\noindent
\begin{remark}
We stated an earlier version of this theorem assuming ordinary local $(\beta,\gam)$-occupancy; however, the original proof did not establish that statement.
Although that original statement might still follow from an alternative route, we use the stronger $p$-chromatic local occupancy assumption in the proof below of the revised statement.  The concentration
argument uses sequential resampling of the layers of a
cover, rather than negative correlation between the layer events, generalising prior concentration arguments~\cite{Mol19,Ber19,BKNP22} that specialised to the choice $\lambda=1$.
\end{remark}
\noindent
In the principal asymptotic regimes of our applications we are able to show $p$-chromatic local $(\beta,\gam)$-occupancy with parameters such that the bounds on the correspondence chromatic number obtained from Theorem~\ref{thm:main,chic} have the same asymptotic order as the bounds on fractional chromatic number from Theorem~\ref{thm:main,chif}, and in several cases they agree to leading order.

Roughly, our framework shows how one obtains a good fractional chromatic number result via an understanding of
the hard-core model on the level of expectation, and how verifying a few additional properties related to the
hard-core model one also obtains a comparable result for (list/correspondence) chromatic number.  Given a
suitable cover $\sH=L(H)$ of $G$, these additional properties imply the concentration of certain random
variables representing the sizes of the sets $L(u)\setminus N_H[I]$ when $I$ is sampled from the hard-core
model on $H$, via the fact that neighbourhoods in $H$ can be related to those in $G$. 

A reader may wish to consider Theorem~\ref{thm:main,chic} and~\cite[Thm.~1.13]{BKNP22}, another result with related conclusions, in juxtaposition. 
Our framework has at its heart the hard-core model with free choice of the fugacity $\lam$, which has conceptual and technical advantages:
\begin{itemize}
\item most importantly, it recovers or improves several of the strongest asymptotic results available when this work was first circulated, smoothly extending the maximum-degree form of the seminal bound of Shearer~\cite{She83} to various locally sparse graph classes (the stronger average-degree form of Shearer's result lies outside our framework);
\item in terms of occupancy fraction, the bounds we obtain within our framework are asymptotically tight in some cases (e.g.~triangle-free graphs, cf.~\cite{DJPR18}), presenting a natural limit to these methods;
\item it naturally threads along~\eqref{eq:params}, from Hall ratio through (importantly!) fractional chromatic number to list/correspondence colouring, via local occupancy; and
\item
it lends itself to conveniently producing good local results as per Subsection~\ref{sub:intro,local}, indeed matching or bettering earlier results in this direction~\cite{BKNP22,DJKP20} (see Subsection~\ref{sub:local}).
\end{itemize}

\subsection{Main ideas and proof of Theorem~\ref{thm:main,occfrac}}\label{sub:ideas}

To clarify the motivation for our approach, we discuss in detail the properties of the hard-core model that we exploit. 
This section serves as a warm-up for the proof of Theorem~\ref{thm:main,chic} and comprises a proof of Theorem~\ref{thm:main,occfrac}.
Moreover, the product of these arguments, fed to~\cite[Lem.~3]{DJKP20}, yields Theorem~\ref{thm:main,chif}.
This is an aggregation and distillation of ideas earlier substantiated~\cite{DJKP20,DJKP21,DJPR17,DJPR18}. 

Given a graph~$G$ of maximum degree~$\Delta$, let~$\bI$ be drawn from the hard-core model on~$G$ at fugacity~$\lam$. 
We are interested in a lower bound on~$\EE |\bI|$, as this implies a lower bound on the independence number~$\alpha(G)$. 
We may rewrite~$\EE |\bI|$, the expected number of vertices of~$G$ occupied by~$\bI$, in terms of the partition function:
\[
\EE |\bI| = \sum_{I\in\cI(G)} |I| \frac{\lam^{|I|}}{Z_G(\lam)}= \frac{\lam Z'_G(\lam)}{Z_G(\lam)}.
\]
We shall rely on a special local property of the hard-core model, which essentially states that $\bI$ restricted to certain induced subgraphs is itself distributed as the hard-core model (at the same fugacity).
More precisely, given~$X\subseteq V(G)$, write~$\bF_X$ for the random subgraph
of~$G[X]$ induced by the set of vertices obtained from the following random
experiment. Reveal~$\bI\setminus X$ and let the set~$\bU_X$ of \emph{externally
uncovered} vertices of~$X$ consist of those vertices in~$X$ with no neighbour
in~$\bI\setminus X$. Then let~$\bF_X$ be the subgraph of~$G$ induced
by~$\bU_X$. Formally,
\[ \bF_X = G[X\setminus N(\bI\setminus X)].  \]
All of our results depend on the following fundamental fact about the behaviour of~$\bI\cap X$ in terms
of~$\bF_X$. As an aside, we remark that this fact was key to Bernshteyn's application of the lopsided
Lov\'{a}sz local lemma~\cite[Eqn.~(\#)]{Ber19}, and, as we will see later, is just as important
for us.

\begin{spatialmarkov}
    For any graph~$G$ and any~$X\subseteq V(G)$,
    if~$\bI$ is a random independent set drawn from the hard-core model on~$G$
    at fugacity~$\lam>0$, then $\bI\cap X$ is distributed according to the
    hard-core model on $\bF_X\coloneqq G[X\setminus N(\bI\setminus X)]$ at fugacity~$\lam$.
\end{spatialmarkov}
\begin{proof}
Let $I_0$ be an arbitrary independent set of $G-X$ and
let us condition on the fact that $\bI\setminus X=I_0$. It follows that $\bI\cap X$ is contained
    in $X\setminus N(I_0)$. For any independent set $I_1$ contained
    in $X\setminus N(I_0)$, we have
\begin{align*}
    \Pr(\bI\cap X= I_1\mid \bI\setminus X=I_0)&=\frac{\Pr(\bI=I_0\cup I_1 \wedge \bI\setminus X=I_0)}{\Pr(\bI\setminus X=I_0)}
    =\frac{\Pr(\bI=I_0\cup I_1)}{\Pr(\bI\setminus X=I_0)}\\
    &=\frac{\lam^{|I_0\cup I_1|}}{Z_G(\lam)}\cdot\frac{Z_G(\lam)}{\sum_{I\in\cI(\bF_X)}\lam^{|I\cup I_0|}}
    =\frac{\lam^{|I_1|}}{\sum_{I\in\cI(\bF_X)}\lam^{|I|}},
\end{align*}
which completes the proof as $Z_{\bF_X}(\lam) = \sum_{I\in\cI(\bF_X)}\lam^{|I|}$.
\end{proof}

Armed with this fact, and given $u\in V(G)$, let us now consider the
terms~$\Pr(u\in\bI)$, the probability that~$u$ is occupied by $\bI$,
and $\EE|N(u)\cap\bI|$, the expected number of neighbours of~$u$ occupied
by~$\bI$, in turn.  For convenience, we write~$\bI_{N(u)}$ for an independent
set drawn from the hard-core model on~$\bF_{N(u)}$ at fugacity~$\lam$.

We derive the conditional probability $\Pr(u\in\bI \mid \bI\cap N(u)=\varnothing)$ by revealing~$\bI\setminus \{u\}$. The spatial Markov property implies that it is $\lam/(1+\lam)$. Thus
\begin{align*}
\Pr(u\in\bI) & = \Pr(u\in\bI \wedge \bI\cap N(u)=\varnothing) = \frac{\lam}{1+\lam} \Pr(\bI\cap N(u) = \varnothing) \\
& = \frac{\lam}{1+\lam} \Pr(\bI_{N(u)} = \varnothing) = \frac{\lam}{1+\lam} \EE\frac{1}{Z_{\bF_{N(u)}}(\lam)},
\end{align*}
where we used the spatial Markov property again in the second line. Similarly, 
\begin{align*}
\EE|N(u)\cap\bI| = \EE|\bI_{N(u)}|= \EE\frac{\lam Z'_{\bF_{N(u)}}(\lam)}{Z_{\bF_{N(u)}}(\lam)}.
\end{align*}

Now we see where conditions~\eqref{eq:main,nonlocal} and~\eqref{eq:main} enter: since $\bF_{N(u)}$ is an induced subgraph of $G[N(u)]$, either condition gives
\begin{align}\label{eq:ideas,closed}
\beta_u \Pr(u\in \bI) + \gam_u\EE|\bI\cap N(u)|
& = \EE\left[\beta_u\frac{\lam}{1+\lam} \frac{1}{Z_{\bF_{N(u)}}(\lam)}+ \gam_u\frac{\lam Z'_{\bF_{N(u)}}(\lam)}{Z_{\bF_{N(u)}}(\lam)}\right] \\
 = \sum_F \Pr(\bF_{N(u)}&=F)\left(\beta_u\frac{\lam}{1+\lam} \frac{1}{Z_F(\lam)}+ \gam_u\frac{\lam Z'_F(\lam)}{Z_F(\lam)} \right)
 \ge 1,
\end{align}
where the summation runs over all induced subgraphs~$F$ of~$G[N(u)]$.
Although~\eqref{eq:main} is more general, the above motivates the label `local occupancy'.

Writing $\beta=\max_u \beta_u$ and $\gam=\max_u \gam_u$, we then have
\[
    \forall u\in V(G),\quad\beta\Pr(u\in \bI) + \gam\EE|\bI\cap N(u)|\ge 1,
\]
and summed over $u\in V(G)$ this gives
\[
    \beta\EE|\bI| + \gam\Delta\EE|\bI| \ge |V(G)|,
\]
because each vertex $v$ appears $\deg(v)\le\Delta$ times in
\[
\sum_{u\in V(G)}\sum_{v\in N(u)}\Pr(v\in \bI) = \sum_{u\in V(G)}\EE|\bI\cap N(u)|.
\]
So we have
\[
    \alpha(G) \ge \EE|\bI| \ge \frac{|V(G)|}{\beta + \gam\Delta},
\]
which by uniformly taking $(\beta_u,\gam_u)=(\beta,\gam)$ is (a slightly stronger version of) the independence number conclusion required for Theorem~\ref{thm:main,occfrac}.

\subsection{Proof of Theorem~\ref{thm:main,chic}}\label{sub:genproofs}

We use a two-phase method for proving the existence of correspondence colourings via the hard-core model and the Lov\'{a}sz local lemma. 
This builds upon several previous proofs beginning with the breakthrough of Molloy~\cite{Mol19}, cf.~\cite{Ber19,BKNP22,DJKP20}. Note that all previous works regarded only the uniform case, whereas here it is crucial that we extend to the general hard-core model.

For the proof, we have chosen to adopt the terminology of Bernshteyn~\cite{Ber19}.
Suppose we are given a graph~$G$ and a $k$-fold cover $\sH=(L,H)$ of~$G$. 
For~$U\subset V(G)$, define $L(U) =\bigcup_{u\in U}L(u)$.
Define~$H^*$ to be the subgraph of~$H$ obtained by removing the edges inside~$L(u)$ for all~$u\in V(G)$; as Bernshteyn does, we write $\deg^*_{\sH}(x)$ instead of~$\deg_{H^*}(x)$. 
For $I\in \cI(H)$, the \emph{domain} of~$I$ is $\dom(I) = \{u\in V(G)\, :\, L(u)\cap I\ne\varnothing\}$.
Any set~$I\in \cI(H)$ corresponds to a partial $\sH$-colouring of~$G$, where the set of coloured vertices of~$G$ is~$\dom(I)$.
We have convenient subscript notation to refer to the uncoloured graph that remains and its induced cover. 
Let~$G_I = G-\dom(I)$ and $\sH_I=(L_I,H_I)$ be the cover of~$G_I$ given by~$H_I=H-N_H[I]$ and $L_I(u)=L(u)\setminus N_H[I]$ for all~$u\in V(G_I)$. 
Now by definition if $I'$ is an $\sH_I$-colouring of~$G_I$ then $I\cup I'$ is an $\sH$-colouring of~$G$. 

The concluding, second phase is standard in probabilistic graph colouring (cf.~\cite{MoRe02}) and is often referred to as the `finishing blow'. The version we employ is a local form, and it follows easily from the Lov\'{a}sz local lemma, albeit without any attempt to optimise the multiplicative constant~$1/8$.

\begin{lemma}[\cite{DJKP20}]\label{lem:phase2}
Let $\sH = (L,H)$ be a cover of a graph $G$. Suppose there is a function $\ell \colon V(G) \to \NN_{\ge 3}$, such that $|L(u)| \ge \ell(u)$ for all $u\in V(G)$ and $\deg^*_\sH(x) \le \min_{v\in N(u)}\ell(v)/8$ for all $x \in L(u)$. Then $G$ is $\sH$-colourable. 
\end{lemma}

\noindent
There are two interrelated conditions in the above lemma which guarantee an $\sH$-colouring: first that there are large enough lists, and second that these lists do not create too much local competition for colours. 

\begin{standing}
    From here until the end of the section, we will always assume that~$\sH=(L,H)$ is a
    $k$-fold cover of a graph~$G$ that satisfies the conditions of Theorem~\ref{thm:main,chic} with
    parameters~$\lam,\ell>0$, $p\in\interval[oo]{0}{1}$, and~${(\beta_u,\gam_u,k_u)}_{u\in V(G)}$.
\end{standing}

We may first delete colours so that $|L(u)|=\lceil k_u\rceil$ for every~$u$.  If $\lceil k_u\rceil\ge\Delta+1$, we delete~$u$ and extend a colouring of the remaining graph greedily.  Repeat this operation, retaining for each remaining vertex~$u$ its original degree~$d_u$ as an upper bound on its current degree, as well as the original parameters and~$\Delta$.  The local occupancy and partition-function hypotheses are inherited by the remaining induced graph, while~\eqref{eq:ysmall} and~\eqref{eq:klarge} retain their original values; every use of~$d_u$ below only requires this upper bound.  Thus, from now on, $|L(u)|\le\Delta$ for every~$u$.

In the main, first phase of the method, it will suffice to find some $I\in \cI(H)$, i.e.~a partial $\sH$-colouring of $G$, such that, for the uncoloured graph~$G_I$ that remains, the induced cover~$\sH_I$ satisfies the two conditions of Lemma~\ref{lem:phase2}.
That is, the conclusion of Theorem~\ref{thm:main,chic} follows from the following lemma: apply Lemma~\ref{lem:phase2} to~$\sH_I$ with the constant function~$\lceil\ell\rceil$.

\begin{lemma}\label{lem:phase1}
    There exists $I\in \cI(H)$ such that $|L_I(u)|\ge\lceil\ell\rceil$ and $\deg^*_{\sH_I}(x) \le \ell/8$ for all~$x\in L_I(u)$ and all~$u\in V(G_I)$.
\end{lemma}

\noindent
Lemma~\ref{lem:phase1}, in turn may be derived from the following pair of bounds on the probability of certain undesirable events in a random partial $\sH$-colouring of~$G$. Naturally, these events correspond closely to the conditions of Lemma~\ref{lem:phase2}.

\begin{lemma}\label{lem:hcmH}
Fix $u\in V(G)$ and $J\in \cI( H[L(V(G)\setminus{N[u]})] )$. If $\bI'$ is a random independent set drawn from the hard-core model on $H[L_J(N(u))]$ at fugacity~$\lam$, then writing $\bI=J\cup\bI'$ the following bounds hold. 
\begin{enumerata}
\item\label{itm:LIu} The probability of a small residual list satisfies
\[
\Pr(|L_{\bI}(u)|\le\ell)
\le \exp\left(-\frac{\iota^2\ell}{2(1-\iota)}\right)
   + \exp\left(-\frac{\iota^2p d_u\gam_u}{2(1-\iota)}\right)
\le \frac{1}{8\Delta^3}.
\]
\item\label{itm:deg} The probability of a large residual colour degree satisfies
\[
\Pr\left(\exists x\in L_{\bI}(u):
    \deg^*_{\sH_{\bI}}(x)>\ell/8\right)
\le \frac{1}{8\Delta^3}.
\]
\end{enumerata}
\end{lemma}

\noindent
The derivation of Lemma~\ref{lem:phase1} from Lemma~\ref{lem:hcmH} is analogous to a key derivation made by Molloy~\cite{Mol19} using the entropy compression method.  Bernshteyn~\cite{Ber19} soon after showed that this could be done instead with the lopsided Lov\'{a}sz local lemma. By now this derivation is standard, cf.~e.g.~\cite{BKNP22},
but since it provides a clear connection, via the spatial Markov property, between the hard-core model and the local lemma, we have decided to include it for completeness (and nearly verbatim from~\cite{Ber19}).
\begin{proof}[Proof of Lemma~\ref{lem:phase1}]
Let $\bI$ be a random independent set from the hard-core model on $H$ at fugacity $\lam$.
For each $u\in V(G)$, define $B_u$ to be the event
\begin{align}
\big\{ u\notin\dom(\bI)\text{ and either $|L_{\bI}(u)|\le\ell$ or $\exists x\in L_{\bI}(u)$ with $\deg^*_{\sH_{\bI}}(x)>\ell/8$} \big\}.
\end{align}
The probabilistic method ensures the desired conclusion if the probability that none of the events~$B_u$ occurs is positive.

For each $u\in V(G)$, take $\Gam(u)= N^3[u]$ (that is, the set of all vertices within distance~$3$ of~$u$).
Since $|\Gam(u)| \le \Delta^3$, it will be sufficient, by the lopsided Lov\'{a}sz local lemma, to prove that for all $S\subseteq V(G)\setminus \Gam(u)$,
\begin{align}
\Pr\left(B_u \mathrel{\Bigg|} \bigcap_{j\in S} \overline{B_j}\right) \le \frac{1}{4\Delta^3}.
\end{align}
By definition, the outcome of any $B_v$ is determined by the set $\bI \cap L(N^2[v])$. If $v\in V(G)\setminus \Gam(u)$, then the distance between $u$ and $v$ is at least $4$, and so $N^2[v] \subseteq V(G)\setminus N(u)$.
Thus it suffices to show that
\begin{align}
\Pr(B_u \mid \bI \cap L(V(G)\setminus N(u)) = J ) \le \frac{1}{4\Delta^3}
\end{align}
for all $J \in \cI(H[L(V(G)\setminus N(u))])$. To that end, fix such an
independent set~$J$. We may assume that $u\notin \dom(J)$, i.e.~$J\in
\cI(H[L(V(G)\setminus N[u])])$, or else the probability we want to bound is
automatically zero.  Let $\bI'=\bI\cap L(N(u))$. By the spatial Markov
property,~$\bI'$, under the conditioning event, is a random independent set
from the hard-core model on~$H[L_J(N(u))]$ at fugacity~$\lam$. Therefore, it
follows from Lemma~\ref{lem:hcmH} that
\[
\Pr(B_u \mid \bI \cap L(V(G)\setminus N(u)) = J ) \le \frac{1}{8\Delta^3}+\frac{1}{8\Delta^3}=\frac{1}{4\Delta^3}.\qedhere
\]
\end{proof}
For the remainder of the section we focus on Lemma~\ref{lem:hcmH}, which is all that is left to complete the proof of Theorem~\ref{thm:main,chic}.

\begin{furtherstanding}
From here until the end of the section, we will always assume that $u$ is a fixed vertex, $J$ is a fixed independent set, and $\bI$ and $\bI'$ are random independent sets, all as in Lemma~\ref{lem:hcmH}.
\end{furtherstanding}

We require some additional notation.
For $x\in L(u)$, let $\Lam(x)$ be the \emph{layer} of~$x$ given by $\Lam(x)=N_{H^*_J}(x)$. 
This consists of the colours in $L_J(N(u))$ that conflict with~$x$, and so for distinct $x,x'\in L(u)$ the layers~$\Lam(x)$ and~$\Lam(x')$ are necessarily disjoint.

For brevity, we write~$k=|L(u)|$ and we arbitrarily index the list as~$L(u)=\{x_1,\dotsc,x_k\}$. 
First let~$\bI'_0$ be drawn from the hard-core model on~$H[L_J(N(u))]$ at fugacity~$\lam$.
By the spatial Markov property, we can construct~$\bI'$ from~$\bI'_0$ through a sequential resampling process, as in previous works~\cite{Mol19,Ber19,BKNP22}, that continuously maintains that it has the hard-core distribution at fugacity~$\lam$. 
For~$j=1,\dotsc,k$, let 
$\bI'_j$ be obtained from $\bI'_{j-1}$ by resampling the independent set induced in the layer $\Lam(x_j)$ as follows. 

After the~$(j-1)$th step,
let~$\bU_j = \Lambda(x_j)\setminus N(\bI'_{j-1}\setminus \Lambda(x_j))$ be the set of externally uncovered neighbours in the layer~$\Lam(x_j)$ and write~$\bF_j = H[\bU_j]$ and~$\bZ_j = Z_{\bF_j}(\lam)$. 
Then sample~$\bS_j\sim\mu_{\bF_j,\lam}$ and let $\bI'_j = (\bI'_{j-1}\setminus \Lambda(x_j))\cup \bS_j$.
In words, the~$j$th step of this sequence resamples which vertices in~$N(u)$ are coloured with a colour that corresponds to~$x_j$. By the spatial Markov property and an induction,~$\bI'_j$ has the hard-core distribution on~$H[L_J(N(u))]$ at fugacity~$\lam$. 
After the~$k$th step, we let~$\bI'$ be the end result, i.e.~ we set~$\bI'$ to be the independent set~$\bI'_k$.

Since  $\bI=J\cup\bI'$,
 the event~$\{x_j\in L_\bI(u)\}$ is determined purely by the resampling choice of~$\bS_j$ in the~$j$th step above, in the sense that~$x_j\in L_\bI(u)$ if and only if~$\bS_j=\emptyset$.
Conditioned on the history before the~$j$th resampling, this event occurs with probability~$1/\bZ_j$. 
For notational convenience, let~$t = \lam\beta_u/(1+\lam)$.
Taking~$F$ to be the empty graph in the $p$-chromatic local occupancy
condition shows that~$t\ge1$, since~$Z_F(\lam)=1$ and~$Q_{F,p}(\lam)=0$.
Using $p$-chromatic local $(\beta,\gam)$-occupancy and the fact that~$\bF_j$ is isomorphic to a subgraph of
$G[N(u)]$, we see that for each~$j$ we have either
\begin{enumerata}
    \item\label{itm:bigprob} $1/\bZ_j\ge 1/t$, or
    \item\label{itm:bigmedian} $1/\bZ_j < 1/t$ and, conditionally on the history, with probability at least $p$ we have $|\bI\cap \Lambda(x_j)| = |\bS_j| \ge 1/(p\gam_u)$.
\end{enumerata}
We exploit this in a coupling argument analogous to the ones in previous works~\cite{Mol19,Ber19,BKNP22}.

Let~$\ba$ be an infinite sequence of independent Bernoulli-$1/t$ random variables and let~$\bb$ be an infinite sequence of Bernoulli-$p$ random variables (each indexed by $\{1,2,\dotsc\})$.
Write~$\bA = \{j : 1/\bZ_j \ge 1/t\}$ for the indices of colours such that~\ref{itm:bigprob} holds and write~$\bB=[k]\setminus \bA=\{j : 1/\bZ_j < 1/t\}$ for the indices of colours such that~\ref{itm:bigmedian} holds.
Note that $|\bA|+|\bB|=k$.

Given $\bI'_{j-1}$, if $j\in \bA$ then
$\Pr(\bS_j=\emptyset) \ge 1/t$ and if~$j\in \bB$ then
$\Pr(|\bS_j|\ge 1/(p\gam_u))\ge p$.
Then we can couple the sampling of~$\bI'$ with the sampling of~$\ba$ and~$\bb$.
This coupling may be constructed online, using at each step the next unused variable of the appropriate sequence and fresh auxiliary randomness, so the variables used retain the stated independent Bernoulli laws despite the adaptive choice between~$\bA$ and~$\bB$.
More precisely:
\begin{enumerate}
  \item for~$j\in \bA$, if~$i = |\{1,\dotsc,j\}\cap A|$ then~$\ba_{i}=1$ implies~$\bS_j=\emptyset$; and
  \item for~$j\in \bB$, if~$i' = |\{1,\dotsc,j\}\cap B|$ then~$\bb_{i'}=1$ implies~$|\bS_j|\ge Q_{\bF_j,p}(\lambda)$.
\end{enumerate}
Under this coupling, we claim that 
\begin{align} 
\sum_{i=1}^{|\bA|} \ba_i &\le \sum_{j\in \bA} \mathbf{1}\{\bS_j = \emptyset\}\le |L_\bI(u)|,\label{eq:coupleprob}
\shortintertext{and}
\sum_{i=1}^{|\bB|} \bb_i &\le \sum_{j\in \bB} \mathbf{1}\{|\bS_j|\ge Q_{\bF_j,p}(\lambda)\} \le p d_u\gam_u.\label{eq:couplemedian}
\end{align}
To see this, note that the first inequality in each of~\eqref{eq:coupleprob} and~\eqref{eq:couplemedian} holds by the definition of the coupling. 
The second inequality in~\eqref{eq:coupleprob} holds because the sum over~$\bA$ counts the number of
colours~$x_j\in L(u)$ such that~$j\in \bA$ and~$x_j\in L_\bI(u)$, which is clearly at most~$|L_\bI(u)|$.
The second inequality in \eqref{eq:couplemedian} holds because for any~$j\in \bB$
we have~$Q_{\bF_j,p}(\lam)\ge1/(p\gam_u)$ as noted earlier.
Then 
\begin{align}
\frac{1}{p\gam_u} \sum_{j\in \bB} \mathbf{1}\{|\bS_j| \ge Q_{\bF_j,p}(\lam)\} \le \sum_{j\in \bB} Q_{\bF_j,p}(\lam)\mathbf{1}\{|\bS_j|\ge Q_{\bF_j,p}(\lam)\}\le d_u,
\end{align}
because the second sum over~$\bB$ counts the number of neighbours of~$u$ that under~$\bI'$ get coloured with a colour that corresponds to some~$x_j\in L(u)$ for~$j\in \bB$. 
This is clearly at most the degree~$d_u$ of~$u$.

We now show that with probability at least~$1-\Delta^{-3}/8$ we have~$|L_\bI(u)|> \ell$. 
Set
\begin{align}
    A = k - \frac{d_u\gam_u}{1-\iota} \quad\text{ and }\quad
    B= \frac{d_u\gam_u}{1-\iota}.
\end{align}
Since $k=|L(u)|\ge k_u$, condition~\eqref{eq:klarge} gives
\[
    A\ge \frac{1}{1-\iota}\frac{\lam\ell}{1+\lam}\beta_u.
\]
Let $a_0=\lceil A\rceil$ and $b_0=\lfloor B\rfloor+1$.  The sums of the first $a_0$ terms from $\ba$ and of the first $b_0$ terms from~$\bb$ have means $\mu_a$ and $\mu_b$ satisfying
\begin{align*}
\mu_a=\frac{a_0}{t}\ge\frac{\ell}{1-\iota} \quad\text{ and }\quad \mu_b = pb_0 > p\frac{d_u\gam_u}{1-\iota},
\end{align*}
respectively.
By a Chernoff bound, we have
\[
\Pr\left(\sum_{i=1}^{a_0}\ba_i \le \ell \right)
\le \exp\left(-\frac{\iota^2\ell}{2(1-\iota)}\right)
\le \frac{1}{16\Delta^3},
\]
and
\[
\Pr\left(\sum_{i=1}^{b_0}\bb_i \le p d_u\gam_u \right)
\le \exp\left(-\frac{\iota^2p d_u\gam_u}{2(1-\iota)}\right)
\le \frac{1}{16\Delta^3},
\]
where the final upper bounds on the probabilities hold by the conditions~\eqref{eq:elllarge} and \eqref{eq:ysmall} from Theorem~\ref{thm:main,chic}. 
By a union bound, it suffices to show that conditioned on the negation of both these `bad' events we have~$|L_\bI(u)|>\ell$.
If~$|\bB|\ge b_0$ then the second event gives
\[ \sum_{i=1}^{|\bB|} \bb_i \ge \sum_{i=1}^{b_0} \bb_i > p d_u\gam_u, \]
a contradiction of~\eqref{eq:couplemedian}, and hence~$|\bB|\le\lfloor B\rfloor$.
Then $|\bA|= k-|\bB|\ge k-\lfloor B\rfloor=\lceil A\rceil=a_0$, and the first event together with~\eqref{eq:coupleprob} gives
\begin{align}
    |L_\bI(u)| &\ge \sum_{i=1}^{|\bA|} \ba_i \ge \sum_{i=1}^{a_0} \ba_i > \ell.
\end{align}
This completes the proof of Lemma~\ref{lem:hcmH}\ref{itm:LIu}.

Lemma~\ref{lem:hcmH}\ref{itm:deg} follows from another analysis of the hard-core model for which the sequential sampling of~$\bI'$ is not a helpful tool. 
We instead analyze~$\bI'$ as if it were sampled in one shot, leaving the `history' of the resampling opaque.
For an arbitrary~$x\in L(u)$, write~$\bU_x = \Lam(x)\setminus N_{H_J}(\bI'\setminus \Lam(x))$ and let~$\bF_x = H[\bU_x]$.
We caution that~$\bU_x$ is not necessarily the same as the set~$\bU_j$ from the resampling process, where~$j$
is the index of~$x$ process above, as future resampling steps can uncover more vertices. 
Writing
\[
\badeve_x = \big\{x\in L_{\bI}(u) \text{ and } \deg^*_{\sH_{\bI}}(x) > \ell/8\big\},
\]
When~$x\in L_{\bI}(u)$ we must have~$|\bI'\cap \Lam(x)|=0$ and hence~$\deg^*_{\sH_{\bI}}(x) = |\bU_x|$, as
some~$x'\in N_{H^*_J}(x)$ remains a member of~$N_{H^*_\bI}(x)$ only if~$x'$ is an uncovered vertex in~$\Lam(x)$ conditioned on~$\bI'\setminus \Lam(x)$. 
In the case~$|\bI'\cap \Lam(x)|=0$, every vertex of~$\bU_x$ contributes to~$\deg^*_{\sH_{\bI}}(x)$.
Then~$\badeve_x$ occurs if and only if both~$|\bU_x|>\ell/8$ and~$|\bI' \cap \Lam(x)|=0$. 
Conditioned on~$\bI'\setminus\Lam(x)$, the spatial Markov property gives
\[
\Pr\bigl(\bI'\cap\Lam(x)=\varnothing\mid\bI'\setminus\Lam(x)\bigr)=\frac{1}{Z_{\bF_x}(\lam)}.
\]
The structure of a cover means that~$\bF_x$ is isomorphic to a subgraph of~$G[N(u)]$.  Consequently, condition~\eqref{eq:Zlarge} gives
\[
\Pr(\badeve_x)=\EE\left[\mathbf 1_{\{|\bU_x|>\ell/8\}}\frac{1}{Z_{\bF_x}(\lam)}\right]\le\frac{1}{8\Delta^4}.
\]
Since $|L(u)|=k\le\Delta$, a union bound proves Lemma~\ref{lem:hcmH}\ref{itm:deg}.
This completes the proof of Lemma~\ref{lem:hcmH}, and hence the proof of Theorem~\ref{thm:main,chic}.

Theorem~\ref{thm:main,chic} allows the quantile probability~$p$ to be
chosen independently of the fugacity.  It is crucial that our method
allows for choosing both~$\lam$ small and~$p$ close to~$1$ in order to
obtain the strongest bounds.
Prior work~\cite{Mol19,Ber19,BKNP22} used uniformly chosen independent sets or partial proper colourings,
corresponding to~$\lam=1$ in the hard-core formulation; in several of the non-triangle-free applications
below, the freedom to take~$\lam$ small and~$p$ close to~$1$ improves the resulting leading constants. The
framing in terms of local occupancy also permits insights on whether the quantitative aspects of the method
are close to tight.

\subsection{Simplification in the case of triangle-free graphs}\label{sub:tfree}

The heart of the proof of Theorem~\ref{thm:main,chic} is the probabilistic analysis in the proof of Lemma~\ref{lem:hcmH}. 
In the case that the relevant vertex $u$ from Lemma~\ref{lem:hcmH} is contained in no triangles, a significantly simpler argument suffices. 
The argument only requires local occupancy in the sense of~\eqref{eq:main} rather than the stronger condition of chromatic local occupancy and offers a few quantitative improvements.

When $u$ is contained in no triangles, we claim that the indicator random variables of the events $\{x\notin L_\bI(u)\}$ for $x\in L(u)$ are negatively correlated. 
This was proved by Bernshteyn~\cite{Ber19}, and we give a more detailed version of the proof for completeness. 

It is enough to show that for all~$x\in L(u)$ and~$Y\subset L(u)\setminus\{x\}$ we have 
\[
\Pr(x\notin L_{\bI}(u) \mid Y \cap L_{\bI}(u) = \varnothing ) \le \Pr(x\notin L_{\bI}(u)).
\]
A simple application of the definition of conditional probability gives that this is equivalent to 
\[ 
\Pr(Y\cap L_{\bI}(u)=\varnothing \mid x\in L_{\bI}(u)) \ge \Pr(Y\cap L_{\bI}(u)=\varnothing).
\] 
That is, we have negative correlation if conditioning on $x\in L_\bI(u)$ can only increase the probability of the event $\{Y\cap L_\bI(u)=\varnothing\}$.
When $u$ is contained in no triangles, note that each $H[L_J(N(u))]$ that we consider is a disjoint union of complete graphs indexed by the neighbours of $u$, and conditioning on $x\in L_\bI(u)$ is, by the spatial Markov property, equivalent to deleting $\Lam(x)$ from $H[L_J(N(u))]$. 
We justify this with a componentwise coupling.
For each~$v\in N(u)$, write $H_v=H[L_J(v)]$.
Since~$u$ is contained in no triangles, the nonempty graphs~$H_v$ are
the components of~$H[L_J(N(u))]$, and each is a clique.
Moreover, the matching condition for a cover implies that
$\Lam(x)$ contains at most one vertex~$a_v$ of each~$H_v$.

Let~$\bS_v$ be an unconditional hard-core sample from~$H_v$.
If~$a_v$ does not exist, set~$\widehat{\bS}_v=\bS_v$.
Otherwise, obtain a sample~$\widehat{\bS}_v$ conditioned not to
contain~$a_v$ as follows. If~$\bS_v\ne\{a_v\}$, set
$\widehat{\bS}_v=\bS_v$; if~$\bS_v=\{a_v\}$, resample
$\widehat{\bS}_v$ from the hard-core model on~$H_v-a_v$.
This has the required conditional distribution: if
$T\ne\{a_v\}$ has hard-core weight~$w(T)$ and
$Z_{H_v}(\lam)=Z_{H_v-a_v}(\lam)+\lam$, then
\[
\Pr(\widehat{\bS}_v=T)
=\frac{w(T)}{Z_{H_v}(\lam)}
+\frac{\lam}{Z_{H_v}(\lam)}
  \frac{w(T)}{Z_{H_v-a_v}(\lam)}
=\frac{w(T)}{Z_{H_v-a_v}(\lam)}.
\]
Apply these couplings independently over~$v\in N(u)$.
The resulting independent set has the distribution of~$\bI'$
conditioned on~$x\in L_\bI(u)$.
Since~$\Lam(x)$ is disjoint from~$\Lam(y)$ for every~$y\in Y$, the
coupling never removes a selected vertex from any~$\Lam(y)$.
Consequently it preserves the event that every~$y\in Y$ is
unavailable, proving
\[
\Pr(Y\cap L_\bI(u)=\varnothing\mid x\in L_\bI(u))
\ge \Pr(Y\cap L_\bI(u)=\varnothing).
\]

The Chernoff bound for negatively correlated variables now gives
\[
\Pr\left(|L_\bI(u)|\le\frac{1}{2}\EE|L_\bI(u)|\right)
\le \exp\left(-\frac{1}{8}\EE|L_\bI(u)|\right),
\]
obviating the need for an analysis of the sequential sampling procedure.
We proceed to give a lower bound on $\EE|L_\bI(u)|$ using the weaker notion of local occupancy~\eqref{eq:main}.
For $x\in L(u)$, let $\bF_x$ be the subgraph of $H$ induced by the
random set $\bU_x = \Lam(x)\setminus N_H(\bI'\setminus \Lam(x))$.
Conditionally on $\bI'\setminus\Lam(x)$, the spatial Markov property gives
\begin{align*}
\Pr\left(x\in L_\bI(u)\mid \bI'\setminus\Lam(x)\right)
&=\frac{1}{Z_{\bF_x}(\lam)},\\
\EE\left(|\bI'\cap\Lam(x)|\mid \bI'\setminus\Lam(x)\right)
&=\frac{\lam Z'_{\bF_x}(\lam)}{Z_{\bF_x}(\lam)}.
\end{align*}
The structure of a cover means that $\bF_x$ is isomorphic to a subgraph of $G[N(u)]$. Hence, taking expectations in condition~\eqref{eq:main} gives
\[ \beta_u\frac{\lam}{1+\lam}\Pr(x\in L_\bI(u)) + \gam_u\EE|\bI\cap \Lam(x)| \ge 1. \]
Summing this over all $x\in L(u)$ we obtain
\begin{align*}
k_u \le |L(u)|
&\le \beta_u \frac{\lam}{1+\lam}\EE|L_{\bI}(u)| + \gam_u\sum_{x\in L(u)}\EE|\bI'\cap \Lam(x)| 
\\&\le \beta_u \frac{\lam}{1+\lam}\EE|L_{\bI}(u)| + \gam_u d_u.
\end{align*}
The last inequality holds because $\sum_{x\in L(u)}\EE|\bI' \cap \Lam(x)|$ is the expected number of neighbours of~$u$ which are $\sH$-coloured by~$\bI'$ with a colour that conflicts with some colour in~$L(u)$. 
This is at most~$d_u$ because $H[L_J(N(u))]$ admits a partition into at most $d_u$ vertex-disjoint cliques, namely $\{L_J(v) : v\in N(u)\}$. 
Rearranging yields the lower bound
\begin{align}
    \EE|L_\bI(u)| &\ge \frac{1+\lam}{\beta_u\lam}\left(k_u-\gam_ud_u\right).
\end{align}
Assuming that
\begin{align}
    \frac{1+\lam}{\beta_u\lam}\left(k_u-\gam_ud_u\right) \ge 2\ell \ge 8\log(8\Delta^3),
\end{align}
a Chernoff bound with $\eta=1/2$ now gives the statement of Lemma~\ref{lem:hcmH}\ref{itm:LIu}.
The proof of Lemma~\ref{lem:hcmH}\ref{itm:deg} goes through unchanged and requires only the condition~\eqref{eq:Zlarge}.

That is, in the triangle-free case we may replace the chromatic local
$(\beta,\gam)$-occupancy and concentration hypotheses used above by local
$(\beta,\gam)$-occupancy together with
\begin{align}
    \ell &\ge 4\log(8\Delta^3),\\
    k_u &\ge \frac{2\lam\ell}{1+\lam}\beta_u + \gam_ud_u.
\end{align}
Condition~\eqref{eq:Zlarge} remains unchanged.
The main advantage is in the weakening of the local occupancy condition, as the weaker version is easier to understand precisely in triangle-free graphs. 
We optimize the resulting bounds for triangle-free graphs in Subsection~\ref{subsec:tfreewarump}.

The negative-correlation argument itself is specific to triangle-free graphs.
The following construction shows why care is needed when relaxing the triangle-free condition.

Suppose that~$G[N(u)]$ is a path on three vertices~$v_1,v_2,v_3$ (where~$v_2$ has degree~$2$). Suppose that
for some suitable~$J$, we have~$L_J(v_1) = \{x_1,x_2\}$, $L_J(v_2) = \{y_1,y_3\}$, and~$L_J(v_3) = \{z_1,z_4\}$,
where colour correspondences hold when the subscripts agree. 
Suppose also that~$L(u)=\{w_2,w_4\}$. 
Then the following graph is isomorphic to~$H[L_J(N[u])]$. 

\begin{figure}[ht]
\begin{tikzpicture}[scale=1.25]
\tikzset{every node/.style={circle,draw,font=\small,minimum size=20,inner sep=0}}
    \node (u2) at (3.5,0.5) {$w_2$};
    \node (u4) at (3.5,-0.5) {$w_4$};
    
    \node (v11) at (0,1) {$x_1$};
    \node (v12) at (1,1) {$x_2$};

    \node (v21) at (0,0) {$y_1$};
    \node (v23) at (1,0) {$y_3$};

    \node (v31) at (0,-1) {$z_1$};
    \node (v34) at (1,-1) {$z_4$};

    \draw (v12) -- (u2) -- (u4) -- (v34);
    \draw (v11)--(v12);
    \draw (v21)--(v23);
    \draw (v31)--(v34);
    \draw (v11)--(v21)--(v31);
\end{tikzpicture}
\caption{A counterexample to negative correlation in the presence of edges in~$N(u)$.}
\end{figure}

Let~$\bI'\sim\mu_{H[L_J(N(u))],\lam}$. A direct computation gives the partition function
\[ Z(\lam)\coloneqq Z_{H[L_J(N(u))]}(\lam) = 1 + 6\lam+10\lam^2+5\lam^3. \]
Then by symmetry,   
\begin{align}
    \Pr(w_2\notin L_\bI(u)) = \Pr(w_4\notin L_\bI(u)) &= \frac{\lam+4\lam^2+3\lam^3}{Z(\lam)}\\
    \Pr(w_2\notin L_\bI(u) \land w_4\notin L_\bI(u)) &= \frac{\lam^2 + 2\lam^3}{Z(\lam)},
\end{align}
but it is straightforward to check by multiplying out and comparing coefficients that for any~$\lam>0$ we have positive correlation of the events~$\{w_2\notin L_\bI(u)\}$ and~$\{w_4\notin L_\bI(u)\}$ because 
\[ \frac{(\lam+4\lam^2+3\lam^3)^2}{Z(\lam)^2} <  \frac{\lam^2 + 2\lam^3}{Z(\lam)}.\]


\section{Further prerequisites}\label{sec:furtherprelim}

In Sections~\ref{sec:mad}--\ref{sec:clique}, we perform several local hard-core analyses that are needed to justify the consequences of our framework, those stated in Subsection~\ref{sub:intro,chi}. This section provides a few more preliminaries for such analyses. Some of these analyses give rise to terms best expressed in terms of the Lambert $W$-function, several properties of which we describe in Subsection~\ref{sub:Lambert}. We repeatedly (locally) apply two general bounds on the expected size of a random independent set, which are given in Subsections~\ref{sub:entropy,hcm} and~\ref{sub:sparse,hcm}.

\subsection{The Lambert \texorpdfstring{$W$}{W}-function}\label{sub:Lambert}

We will be interested in the solutions to equations such as $y = x e^x$ and $y=e^x/\log x$ which cannot be expressed with elementary functions, and we collect the necessary material here.

The equation~$y= x e^x$ is well studied and the solution gives rise to the \emph{Lambert~$W$-function}. 
This has two real branches, and we write~$W$ for the `upper' or principal branch and~$W_{-1}$ for the `lower' or negative real branch, see Figure~\ref{fig:Lambert}. 
That is, both~$W$ and~$W_{-1}$ are the inverse of~$x\mapsto x e^x$ but~$W:\interval[co]{-1/e}{\infty}\to\interval[co]{-1}{\infty}$, and~$W_{-1}:\interval[co]{-1/e}{0}\to\interval[oc]{-\infty}{-1}$, see~\cite{CGHJK96} for more details, and proofs of the properties we discuss below. We note that~$W(-1/e) = -1 = W_{-1}(1/e)$.

\begin{figure}[ht]
\centering
\includegraphics{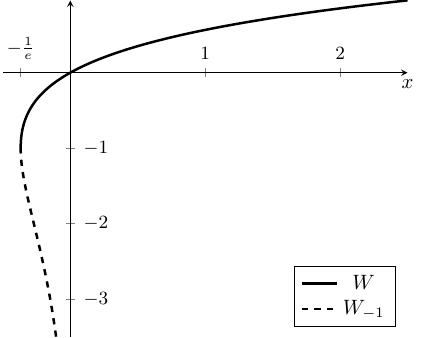}
\caption{The functions~$W(x)$ (solid line) and~$W_{-1}(x)$ (dashed line)}\label{fig:Lambert}
\end{figure}

For the upper branch, we will use that as~$x\to\infty$ we have 
\[
W(x)=\log x-\log\log x+o(1),
\]
and note that~$W(x)e^{W(x)}=x$ implies~$e^{W(x)} = x/W(x)$ and~$e^{W(x)}+x = \frac{x}{e^{W(x)}}\cdot (W(x)+1)$
if~$x\neq0$. These latter properties clearly holds for any branch of~$W$.

We encounter the lower branch as a solution to the equation $x\log(1/x)=1/(ey)$. 
A little rearranging gives~$(\log x) e^{\log x} = -1/(ey)$, and hence~$\log x$ is given by some branch of the Lambert $W$-function applied to~$-1/(ey)$.
If we know that~$x\in\interval[oc]{0}{1/e}$ and~$y\ge 1$ then the solution is
\[
x = e^{W_{-1}(-1/(ey))} = \frac{1}{ey}\cdot \frac{1}{-W_{-1}(-1/(ey))}.
\]
For brevity we introduce~$\fn \colon \interval[co]{1}{\infty}\to\interval[co]{1}{\infty}$ for the function~$\fn(y) = -W_{-1}(-1/(ey))$ appearing above. 
Then for~$x\in\interval[oc]{0}{1/e}$ and~$y\ge 1$ we have
\[
x\log\frac{1}{x} = \frac{1}{ey} \quad\Longrightarrow\quad x = \frac{1}{ey \fn(y)}.
\]
Standard properties of~$W_{-1}$ give that~$\fn$ is continuous, increasing, and as~$x\to\infty$ satisfies~$\fn(x)\sim\log x$, see~\cite{CGHJK96}. 

For yet another equation we only need an asymptotic solution, but up to an additive~$o(1)$ error term. 
Suppose that~$x$ and~$y$ are reals at least~$2$ such that
\[
y = \frac{e^x}{\log x}.
\]
Arguments similar to those establishing analogous properties for Lambert functions yields
that, as $y\to\infty$,
\begin{align}
x &= \log y + \log\log\log y + o(1).
\end{align}

Similarly, if~$x$ and~$y$ are reals at least~$2$ such that
\[
    y = \frac{xe^x}{\log x},
\]
then as~$y\to\infty$ we have
\[
    x = \log y - \log\log y + \log\log\log y + o(1).
\]

\subsection{The hard-core model and entropy}\label{sub:entropy,hcm}

In this subsection, we develop refinements of arguments of Shearer~\cite{She95} and Alon~\cite{Alo96} which offer lower bounds on the expectation and median-like quantile which appear in the conditions of local occupancy and chromatic local occupancy.
In both results we use entropy, see e.g.~the book by Alon and Spencer~\cite{AS16book} for an introduction to the necessary entropy material.
We write~$\log_2$ for the base-2 logarithm and~$h(x)$ for the standard binary entropy function. That
is,~$h(0)=h(1)=0$ and for~$x\in\interval[oo]{0}{1}$ we have 
\[
h(x) = -x\log_2x-(1-x)\log_2(1-x).
\]

We start with a lower bound on the expected size of an independent set from the hard-core model in a graph~$F$
in terms of the number of vertices~$y$ and the partition function~$Z_F(\lam)$. Some minor changes to Shearer's
proof are necessary to handle a general positive fugacity~$\lam$, but the case~$\lam=1$ of the proof below
simply corresponds to a rather precise description of Shearer's original method.  The result is most useful
when one has control of the term~$\fn\big(y\lam/\log Z_F(\lam)\big)$. This in turn boils down to bounding the
number of independent sets in~$F$ in terms of~$y$, and in later relevant sections we establish such results
subject to various structural assumptions on~$F$. 

\begin{lemma}\label{lem:generalizedShearer}
For any graph~$F$ on~$y\ge1$ vertices and any positive real~$\lam$,
\[
\frac{\lam Z_F'(\lam)}{Z_F(\lam)} \ge \frac{\log Z_F(\lam)}{\fn\big(y\lam/\log Z_F(\lam)\big)}.
\]
\end{lemma}

\begin{proof}
In the proof we write~$z$ instead of~$Z_F(\lam)$ for brevity, and we suppose that~$\bX\sim\mu_{F,\lam}$. 
Let~$\phi = \EE|\bX|/y$ be the occupancy fraction, and note that we desire a lower bound on~$\EE|\bX|$.
We first prove that 
\begin{align}\label{eq:logzub}
\log z \le y\phi\log(e\lam/\phi).
\end{align}
To this end, we compute the entropy~$H(\bX)$,
\begin{align*}
H(\bX) 
  &= -\sum_{I\in\cI(F)}\frac{\lam^{|I|}}{z}\log_2\frac{\lam^{|I|}}{z} = \log_2 z-\sum_{I\in\cI(F)} |I|\frac{\lam^{|I|}}{z}\log_2\lam 
\\&= \log_2 z -\EE|\bX|\log_2\lam,
\end{align*}
so that by the subadditivity of the entropy and the concavity of~$h$ (the standard binary entropy function) we have
\begin{align*}
\log_2 z
  &= \EE|\bX|\log_2\lam + H(\bX) 
\\&\le \EE|\bX|\log_2\lam + \sum_{u\in V(F)}H(\indicator{u\in\bX}) 
\\&= y\phi\log_2\lam + \sum_{u\in V(F)}h\big(\Pr(u\in\bX)\big) \le y\big(\phi\log_2\lam + h(\phi)\big).
\end{align*}
Inequality~\eqref{eq:logzub} follows since for all~$x\in\interval[oc]{0}{1}$ we have~$h(x)\le x\log_2(e/x)$, and~$h(0)=0=\lim_{x\to0}x\log_2(e/x)$.
Note that we only use logarithms to base~$2$ in the above calculations, and work with the natural logarithm elsewhere, including in~\eqref{eq:logzub}.

The right-hand side of~\eqref{eq:logzub} is an increasing function of~$\phi$ because the partial derivative with respect to~$\phi$ is~$y\log(\lam/\phi)$ and in the hard-core model
\[
\phi=\frac{1}{y}\sum_{u\in V(F)}\Pr(u\in\bX) \le \frac{1}{y}\sum_{u\in V(F)}\Pr(u\in\bX\mid N(u)\cap\bX=\varnothing) = \frac{\lam}{1+\lam} < \lam.
\]
Hence we bound $\phi$ from below by solving~\eqref{eq:logzub} with equality, leading to
\[
\frac{\phi}{e\lam}\log\frac{e\lam}{\phi} = \frac{\log z}{ey\lam}.
\]
The right-hand side lies in~$[0,1/e]$ since~$z\le {(1+\lam)}^y$.
The exact solution is naturally expressed in terms of the function~$\fn$ (see Subsection~\ref{sub:Lambert}), giving
\[
\EE|\bX| = y\phi \ge \frac{\log Z_F(\lam)}{\fn\big(y\lam/\log Z_F(\lam)\big)}.\qedhere
\]
\end{proof}

A superficially similar argument, given by Alon~\cite{Alo96} in the case~$\lam=1$, provides a lower bound on
the quantile~$Q_F(\lam)$, which we recall is defined to be the greatest~$x$ such that 
\[ \Pr_{F,\lam}(|\bX|\ge x)\ge 1/(1+\lam), \]
where~$\bX\sim\mu_{F,\lam}$.
For simplicity, we restrict to~$0<\lam\le 1$.

\begin{lemma}\label{lem:generalizedAlon}
  Let~$F$ be a $y$-vertex graph with~$y\ge 1$ and for~$\lam\in\interval[oo]{0}{1}$ let~$z=Z_F(\lam)$. 
If~$z\ge y$, then
\[
Q_F(\lam) \ge
\left\lfloor
\frac{\log(z/y)}
     {K\left(\frac{y\lam}{\log(z/y)}\right)}
\right\rfloor,
\]
where the expression inside the floor is interpreted as~$0$ when~$z=y$ and, otherwise, is an increasing
function of~$z$ and a decreasing function of~$y$.
If in addition~$y=z^{o(1)}$ then as~$z\to\infty$ we have 
\begin{align}
\label{eq:median-lowerbound}
    Q_F(\lam) &\ge (1-o(1)) \frac{\log z}{K\left(\frac{y\lam}{\log z}\right)}.
\end{align} 
\end{lemma}

\begin{remark}
    The minor difference in the conclusions of Lemma~\ref{lem:generalizedShearer} and Lemma~\ref{lem:generalizedAlon} explains why in practice one might be able to establish the stronger chromatic local occupancy condition~\eqref{eq:main,nonlocal,chromatic} without too much extra work compared to the weaker condition~\eqref{eq:main}. 
    The difference is that given a $y$-vertex graph $F$, $\bX\sim\mu_{F,\lam}$ and $z=Z_F(\lam)$, Lemma~\ref{lem:generalizedShearer} gives the lower bound 
    \[ \EE|\bX| \ge \frac{\log z}{K(y\lam/\log z)}, \]
    whereas Lemma~\ref{lem:generalizedAlon} bounds~$Q_F(\lam)$, up to integer rounding, using the smaller quantity obtained by replacing $z$ with $z/y$ in this expression.

    Practically, there is more work involved to apply Lemma~\ref{lem:generalizedAlon} as one has to verify the condition $z\ge y$, which may be genuinely restrictive. 
    The analogous conditions for Lemma~\ref{lem:generalizedShearer} are $z\ge 1$ and $y\lam\ge \log z$ which hold true for any nonempty $F$ and $\lam>0$. 
    The final assertion of Lemma~\ref{lem:generalizedAlon}, however, states that if $z\ge y$ and $y=z^{o(1)}$, then as $z\to\infty$ the two lower bounds are asymptotically equal. 
    This is true under suitable assumptions on $\lam$ and $F$, and conditions that guarantee $\lam$ is not too small and $F$ is not too dense should suffice.
    Lemma~\ref{lem:uniformAlon} below removes these restrictions by retaining
    the top level of the elementary lower-tail estimate more carefully.
\end{remark}

\begin{proof}[Proof of Lemma~\ref{lem:generalizedAlon}]
Let $p=1/(1+\lam)$ and let $\bX\sim\mu_{F,\lam}$. 
If $z=y$, the asserted lower bound is~$0$ and there is nothing to
prove, so assume that $z>y$.
Note that $\log(z/y)\le\log z\le y\log(1+\lam)\le y\lam$.
Define
\[
q=\frac{\log(z/y)}
        {K\left(y\lam/\log(z/y)\right)}
\qquad\text{and}\qquad
k=\lfloor q\rfloor.
\]
If $k=0$, there is nothing to prove, so suppose that $k\ge1$.
The definition of $K$ gives
\[
q\log\left(\frac{ey\lam}{q}\right)=\log(z/y).
\]
The function
\[
s\longmapsto \frac{s}{K(y\lam/s)}
\]
is increasing for $0<s\le y\lam$: its value is the unique
$x\in(0,y\lam]$ satisfying
$x\log(ey\lam/x)=s$, and the left-hand side is increasing on this
interval.  Since $\log(z/y)\le\log z$,
Lemma~\ref{lem:generalizedShearer} therefore gives
\[
q\le
\frac{\log z}{K(y\lam/\log z)}
\le\EE|\bX|
\le\frac{y\lam}{1+\lam}.
\]
Thus the sequence $\binom{y}{j}\lam^j$ is increasing for
$0\le j\le k$.  Since
$x\mapsto x\log(ey\lam/x)$ is increasing for $x\le y\lam$, we have
\begin{align}
\sum_{0\le j < k}\binom{y}{j}\lambda^j 
&\le k \binom{y}{k}\lambda^k
\le k \left(\frac{ey\lam}{k}\right)^{k}
\le \frac{\lam}{1+\lam}y
     \left(\frac{ey\lam}{q}\right)^{q}
=\frac{\lam}{1+\lam}z.
\end{align}
Here the third inequality uses~$k\le q\le y\lam/(1+\lam)$ together
with the preceding monotonicity.
Since~$1-p=\lam/(1+\lam)$ and~$i_j(F)\le\binom{y}{j}$, it follows that
\[
\Pr_{F,\lam}(|\bX|<k)
\le\frac{1}{z}\sum_{0\le j<k}\binom{y}{j}\lam^j
\le1-p.
\]
Therefore~$Q_F(\lam)\ge k$, as required.

That $q$ is an increasing function of $z$ and a decreasing function
of $y$ is a matter of straightforward, if tedious, calculus.  If
$y=z^{o(1)}$, then $\log(z/y)\sim\log z$ and the properties of $K$
give $q\to\infty$ and
\[
q=(1-o(1))
\frac{\log z}{K(y\lam/\log z)}.
\]
Consequently $\lfloor q\rfloor=(1-o(1))q$, which
proves~\eqref{eq:median-lowerbound}.
\end{proof}

\begin{lemma}\label{lem:uniformAlon}
Let~$F$ be a graph on~$y\ge1$ vertices, let~$\lam>0$ and~$p\in\interval[oo]{0}{1}$, and write
\[
    z=\log Z_F(\lam)
    \qquad\text{and}\qquad
    a_F=\frac{z}{\fn(y\lam/z)}.
\]
For every~$c\in\interval[oo]{0}{1}$ satisfying
\begin{equation}\label{eq:uniformAlon,condition}
    (1-c+c\log c)z
    \ge \log(a_F+1)+\log\frac{1}{1-p},
\end{equation}
we have
\[
    Q_{F,p}(\lam)>c a_F.
\]
Consequently, uniformly over all graphs~$F$ and choices of~$\lam$ and~$p$ satisfying
\[
    \log\frac{1}{1-p}=o\bigl(\log Z_F(\lam)\bigr),
\]
\begin{equation}\label{eq:uniformAlon,asymptotic}
    Q_{F,p}(\lam)\ge(1-o(1))
    \frac{\log Z_F(\lam)}
         {\fn\big(y\lam/\log Z_F(\lam)\big)}
    \qquad\text{as }\log Z_F(\lam)\to\infty.
\end{equation}
\end{lemma}

\begin{proof}
Let~$\bX\sim\mu_{F,\lam}$ and set~$m=\lfloor c a_F\rfloor+1$ and~$r=m-1$.
As in the proof of Lemma~\ref{lem:generalizedAlon}, it is enough to show
\[
    \sum_{0\le j<m}\binom{y}{j}\lam^j
    \le (1-p)Z_F(\lam).
\]
Lemma~\ref{lem:generalizedShearer}, together with
$\EE|\bX|\le y\lam/(1+\lam)$, gives $a_F\le y\lam/(1+\lam)$.
Thus the sequence $\binom{y}{j}\lam^j$ is increasing for $0\le j\le r$.
If $r\ge1$, it follows that
\begin{equation}\label{eq:uniformAlon,tail}
    \sum_{0\le j<m}\binom{y}{j}\lam^j
    \le m\binom{y}{r}\lam^r
    \le m\left(\frac{ey\lam}{r}\right)^r.
\end{equation}
The definition of $\fn$ gives
\[
    \log\frac{ey\lam}{a_F}=\fn(y\lam/z)
    \qquad\text{and hence}\qquad
    a_F\log\frac{ey\lam}{a_F}=z.
\]
The function $x\mapsto x\log(ey\lam/x)$ is increasing for
$x\le y\lam/(1+\lam)$.  Since $r\le c a_F$, the logarithm of the
right-hand side of~\eqref{eq:uniformAlon,tail} is therefore at most
\[
    \log(a_F+1)+c a_F\log\frac{ey\lam}{c a_F}
    =\log(a_F+1)+cz-c a_F\log c.
\]
As $\fn\ge1$, we have $a_F\le z$; since $\log c<0$, this last expression
is at most
\[
    \log(a_F+1)+\bigl(c-c\log c\bigr)z.
\]
Condition~\eqref{eq:uniformAlon,condition} now proves the desired
lower-tail bound.  The case $r=0$ follows directly from the same
condition, since the sum in question is then $1$.

For~\eqref{eq:uniformAlon,asymptotic}, take $c=c(z)\to1$ sufficiently
slowly that
\[
    (1-c+c\log c)z-\log(z+1)-\log\frac{1}{1-p}\to\infty,
\]
and use $a_F\le z$.
\end{proof}

\begin{lemma}
\label{lem:chromatic,threshold}
Let~$G$ be a graph, let~$\lam>0$ and~$p\in\interval{0}{1}$, and let~${(T_u)}_{u\in V(G)}$ and~${(a_u)}_{u\in V(G)}$
be positive reals with~$T_u\ge1$ for every~$u$.  Suppose that for every~$u\in V(G)$ and every
subgraph~$F\subseteq G[N(u)]$,
\begin{equation}\label{eq:chromatic,threshold}
    Z_F(\lam)>T_u
    \quad\Longrightarrow\quad
    Q_{F,p}(\lam)\ge a_u.
\end{equation}
Then the hard-core model on~$G$ at fugacity~$\lam$ has
$p$-chromatic local $(\beta,\gam)$-occupancy for
\[
    \beta_u=\frac{1+\lam}{\lam}T_u
    \qquad\text{and}\qquad
    \gam_u=\frac{1}{p a_u}.
\]
\end{lemma}

\begin{proof}
Fix $u\in V(G)$ and $F\subseteq G[N(u)]$.  If
$Z_F(\lam)\le T_u$, then
\[
    \beta_u\frac{\lam}{1+\lam}\frac{1}{Z_F(\lam)}
    =\frac{T_u}{Z_F(\lam)}\ge1.
\]
Otherwise~\eqref{eq:chromatic,threshold} gives
$Q_{F,p}(\lam)\ge a_u$, and hence
\[
    \gam_u p Q_{F,p}(\lam)
    =\frac{Q_{F,p}(\lam)}{a_u}\ge1.
\]
Thus one of the alternatives in~\eqref{eq:main,chromatic} always holds.
\end{proof}

\subsection{The hard-core model with bounded average degree}\label{sub:sparse,hcm}

The following lemma has essentially already appeared~\cite{DJKP21,SSSZ19}.

\begin{lemma}\label{lem:sparse,hcm}
For any graph~$F$ on~$y\ge1$ vertices with average degree~$d$ and any positive real $\lam$,
\begin{align}
\frac{\lam Z_F'(\lam)}{Z_F(\lam)} 
    & \ge \frac{\lam}{1+\lam}y{(1+\lam)}^{-d},\quad \text{and} \\
y\log(1+\lam) &\ge \log Z_F(\lam)
\ge \frac{y}{d+1}\log(1 + (d+1)\lambda).
\end{align}
\end{lemma}
\begin{proof}
For the first bound we apply the analysis from Subsection~\ref{sub:ideas} to~$F$. 
Let~$\bI$ be a random independent set from the hard-core model at fugacity~$\lam$ on~$F$. 
First, we have for any~$u\in V(F)$,
\[
    \Pr(u\in\bI) = \frac{\lam}{1+\lam}\Pr(\bI\cap N(u) = \varnothing)\ge \frac{\lam}{1+\lam}{(1+\lam)}^{-d_u},
\]
	because the spatial Markov property gives that $\bI\cap N(u)$ is a random independent set drawn from the
    hard-core model on the subgraph~$\bF_{N(u)}$ of~$F[N(u)]$ induced by the externally uncovered
    neighbours of~$u$.  The final inequality comes from the fact that any realisation of $\bF_{N(u)}$ has
    $Z_{\bF_{N(u)}}(\lam)\le {(1+\lam)}^{d_u}$.  The lemma follows by convexity:
\begin{align}
\EE|\bI| = \sum_{u\in V(F)}\Pr(u\in\bI) 
    &\ge \frac{\lam}{1+\lam}\sum_{u\in V(F)}{(1+\lam)}^{-d_u} \\
    &\ge \frac{\lam}{1+\lam}y{(1+\lam)}^{-d},
\end{align}

The upper bound on~$\log Z_F(\lam)$ follows trivially from the facts that~$Z_F(\lam)$ is monotone increasing
under the removal of edges from the graph~$F$, and an edgeless graph on~$y$ vertices has partition
function~$(1+\lam)^y$. 

The lower bound on~$\log Z_F(\lam)$ follows from a tight lower bound on the partition function due to Sah, Sawhney, Stoner and Zhao~\cite[Thm.~1.7]{SSSZ19}. That result is stronger and gives a bound in terms of the degree sequence of $F$, but the bound we need in terms of average degree follows from their result and the convexity of the function
\[ g(d) = \frac{1}{d+1}\log(1+(d+1)\lam) \]
over nonnegative real numbers.
To see this convexity, which holds for all~$\lam\ge 0$,
note that 
\[ g(d) = \int_0^\lam \frac{1}{1+(d+1)t}\mathrm dt \]
and for each~$t\in[0,\lam]$, the function~$d\mapsto\frac{1}{1+(d+1)t}$ is convex on~$[0,\infty]$.
\end{proof}

\begin{remark}
    The bounds in Lemma~\ref{lem:sparse,hcm} are of very different quality. 
    The first bound (on the expected size of an independent set from the hard-core model) is only tight in degenerate limiting regimes such as when~$\lam\to0$. 
    The bound is potentially useful for small~$\lam$ nonetheless. 

    The second (on the log-partition function) is tight for all integers~$d\ge 0$ and~$y$ divisible by~$d+1$ by a disjoint union of complete graphs~$K_{d+1}$. 
    It would be interesting to improve the first bound to match this level of tightness.
\end{remark}

\subsection{Local occupancy and chromatic local occupancy in triangle-free graphs}\label{subsec:tfreewarump}

To establish the basic local occupancy property~\eqref{eq:main,nonlocal} in triangle-free graphs is a simple exercise in optimisation~\cite{DJPR18,DJKP20}.
Each~$F$ under consideration is an edgeless graph on some number $y\in[0,\Delta]$ of vertices. 
Then 
\begin{align}
    Z_F(\lam) &= (1+\lam)^y&\text{and}&&\frac{\lam Z_F'(\lam)}{Z_F(\lam)} &= \frac{y\lam}{1+\lam}. 
\end{align}
Using the properties of the Lambert-$W$ function it is now not hard to show that taking~$\beta$ and~$\gam$ as in~\eqref{eq:betagam,triangle} we have viable choices of parameters that minimise~$\beta+\gam\Delta$.
When~$\lambda$ satisfies~$\lam=o(1)$ and~$\log(1/\lam) = o(\log\Delta)$ as~$\Delta\to\infty$, e.g.\ for~$\lam=1/\log\Delta$, we indeed have~$\beta+\gam\Delta \sim \Delta/\log\Delta$ as promised. 

More locally, since every subgraph of a neighbourhood is edgeless, the
same calculation also verifies the strong, vertex-dependent local occupancy
condition~\eqref{eq:main}. Combining this with the triangle-free simplification
from Subsection~\ref{sub:tfree} gives an optimised list-size bound. For any
$\gam_u>0$, setting
\begin{align}
    \beta_u=\frac{\gam_u{(1+\lam)}^{\frac{1+\lam}{\gam_u\lam}}}{e\log(1+\lam)}
\end{align}
ensures local $(\beta,\gam)$-occupancy. Moreover,
condition~\eqref{eq:Zlarge} holds provided
\[
    \ell\ge\frac{8\log(8\Delta^4)}{\log(1+\lam)}.
\]
We therefore fix
\[
    \ell=\max\left\{4\log(8\Delta^3),
        \frac{8\log(8\Delta^4)}{\log(1+\lam)}\right\}
\]
and choose~$\gam_u$ to minimise the list-size expression from
Subsection~\ref{sub:tfree},
\[ 
    \frac{2\lam\ell}{1+\lam}\beta_u + \gam_ud_u. 
\]
After substituting for~$\beta_u$, the objective is clearly convex in~$\gam_u$,
recalling that, over the positive real numbers, the function~$x\to \cdot f(x)$
is convex if and only if the function~$x\to f(1/x)$ is convex,
and its minimum occurs at
\begin{align}
    \gam_u &=\frac{1+\lam}{\lam}
        \frac{\log(1+\lam)}
        {1+W\left(\frac{(1+\lam)d_u}{2\lam\ell}
            \log(1+\lam)\right)},
\end{align}
giving
\[
\frac{2\lam\ell}{1+\lam}\beta_u + \gam_ud_u
=\frac{1+\lam}{\lam}\log(1+\lam)
 \,\frac{d_u}{W\left(\frac{(1+\lam)d_u}{2\lam\ell}
     \log(1+\lam)\right)}.
\]
Here the right-hand side is interpreted by continuity as~$2\ell$
when~$d_u=0$.
For fixed~$\lam>0$, we have $\ell=\Theta_\lam(\log\Delta)$, so as
$d_u/\log\Delta\to\infty$ the preceding expression is
\[
\left(\frac{1+\lam}{\lam}\log(1+\lam)+o(1)\right)
\frac{d_u}{\log(d_u/\log\Delta)}.
\]
The factor before the fraction tends to~$1$ as~$\lam\to0$; a complete
local formulation, including vertices of smaller degree, is given in
Subsection~\ref{sub:local}.

Though the strength of chromatic local occupancy is not required for triangle-free graphs, we give a derivation of similar parameters as a warm-up to the more technical settings considered in detail below. 
The important fact is that when~$|F|=y$ the size of the relevant independent set is binomially distributed with~$y$ trials and success probability~$q=\lam/(1+\lam)$. A standard Chernoff approximation gives for any~$\lam\le 1$ that 
\[ Q_F(\lam) \ge yq - \sqrt{2yq(1-q)\log(1/q)}. \]
For any fixed~$\eta\in\interval{0}{1}$, there is~$y_\eta$ such that if~$y\ge y_\eta$ then the displayed lower bound above is at least~$(1-\eta)yq$.
Then chromatic local $(\beta,\gam)$-occupancy follows by setting
\begin{align}
    \beta &= \frac{1+\lam}{\lam}(1+\lam)^{y^*}&\text{and}&&
    \gamma &= (1+2\eta)\frac{(1+\lam)^2}{y^*\lam},
\end{align}
where~$y^*$ is any real number at least~$y_\eta$.
Indeed, the first alternative in chromatic local occupancy holds for~$y\le y^*$, and the second holds for~$y>y^*$.
If for~$\eps\in\interval{0}{1/2}$ we have~$k=(1+\eps)\Delta/\log\Delta$, then choosing fixed~$\iota,\lam$
small enough in terms of~$\eps$, and~$\Delta$ large enough, we can take a suitable~$\ell=\Theta(\log\Delta)$
and 
\[ y^* = \frac{(1+\lam)^3}{\lam}\frac{\log\Delta}{1+\eps^2} \le \frac{1}{\log(1+\lam)}\frac{\log\Delta}{1+\eps^3} \]
such that $\beta\ell \le \Delta^{1-\eps^4}$ and $\gam\Delta\le (1+\eps^2)\Delta/\log\Delta$ and~$y^* \ge y_\eta$.
Before applying Theorem~\ref{thm:main,chic}, repeatedly delete vertices whose current degree is less than
\[
    D=\frac{2(1+\lam)}{\iota^2\gam p}\log(16\Delta^3).
\]
The remaining graph satisfies~\eqref{eq:ysmall} and the other conditions of Theorem~\ref{thm:main,chic}, while
the deleted vertices can be restored greedily since~$D=O_\eps((\log\Delta)^2)=o(k)$. Thus, we deduce that~$G$
is $k$-colorable.

This crude calculation suggests that in sufficiently locally sparse settings one can establish the stronger chromatic local occupancy property with only slight, asymptotically negligible, degradation of the parameters from the perspective of bounds on the chromatic number. 
While settings other than triangle-free graphs lack the convenient binomial distribution property we exploited for this quick calculation, in applications we can use Lemma~\ref{lem:uniformAlon} to transfer the corresponding expectation estimates to $Q_{F,p}(\lam)$.

\section{Bounded local maximum average degree}\label{sec:mad}

In this section we prove generalisations of Theorems~\ref{thm:Cfree,chi}--\ref{thm:boundedC,chi}. 
The key idea is to weaken the condition of triangle-freeness to having bounded average degree in any subgraph of a neighbourhood. 

\subsection{Local occupancy with bounded local mad}\label{sub:mad,localocc}
We begin with the relevant local occupancy result.

\begin{lemma}\label{lem:mad}
Let~$a=(a_u)_{u\in V(G)}$ be a sequence of nonnegative real numbers indexed by the vertices of a graph~$G$ such
    that $\mad(G[N(u)]) \le a_u$ for each~$u\in V(G)$. Then the following statements hold for
    any~$\lam>0$.
\begin{enumerate}
\item\label{itm:mad,lo}
    For any sequence $D\in (\RR^+)^{V(G)}$ of positive reals, setting
\begin{align*}
    \gam_u &\coloneqq \frac{1+\lam}{\lam}\frac{{(1+\lam)}^{a_u}\log(1+\lam)}{1+W(D_u{(1+\lam)}^{a_u}\log(1+\lam))}\\
    \beta_u &\coloneqq \frac{\gam_u{(1+\lam)}^{\frac{{(1+\lam)}^{1+a_u}}{\gam_u\lam}-a_u}}{e\log(1+\lam)},
\end{align*}
we have local ${(\beta,\gam)}$-occupancy in the hard-core model on~$G$ at fugacity~$\lam$ such that for all $u\in V(G)$,
\begin{align*}
    \beta_u+\gam_u D_u = \frac{1+\lam}{\lam} \frac{D_u{(1+\lam)}^{a_u}\log(1+\lam)}{W(D_u{(1+\lam)}^{a_u}\log(1+\lam))}.
\end{align*}
\item\label{itm:mad,chromatic}
For this item, consider an asymptotic family of graphs and parameters satisfying the hypotheses below, and interpret all asymptotic relations as holding uniformly over~$u$.
For~$\lam\in\interval[oc]{0}{1}$, $p\in\interval[oo]{0}{1}$, and any collection~$\beta\in(\RR^+)^{V(G)}$ of positive reals, write
\[
c_u=\frac{\log(1+(a_u+1)\lam)}{a_u+1}
\quad\text{and}\quad
b_u=\beta_u\frac{\lam}{1+\lam}.
\]
	  Suppose that $b_u\to\infty$ and
	  $\log(1/(1-p))=o(\log b_u)$.  Then the hard-core model on $G$ at fugacity $\lam$ has $p$-chromatic local $(\beta,\gam)$-occupancy for a choice of~$\gam$ satisfying
\begin{align*}
	    \gam_u = (1+o(1))\frac{1}{p\log b_u}\fn\left(\frac{\lam}{c_u}\right),
\end{align*}
	  and every componentwise larger choice of~$\gam$ also works.
	  Moreover, if $\ell\ge1$ and $d\in(\RR^+)^{V(G)}$ is a collection of positive reals such that $d_u/\ell\to\infty$, $\lam\to0$,
	  $p\to1$, $\log(1/(1-p))=o(\log(d_u/\ell))$, and
	  $a_u\lam\to0$, then, writing $K_u=\fn(\lam/c_u)$ and
\[
w_u=2W\left(\frac12\sqrt{\frac{d_uK_u}{p\ell}}\right),\quad
\beta_u=\frac{1+\lam}{\lam}e^{w_u},\quad
\gam_u=(1+o(1))\frac{K_u}{pw_u},
\]
the preceding conditions hold and
\[
    \frac{\lam\ell}{1+\lam}\beta_u+d_u\gam_u
    \sim \frac{d_u}{\log(d_u/\ell)}.
\]
\item\label{itm:mad,largeZ}
For any~$u\in V(G)$ and any subgraph~$F$ of~$G[N(u)]$ on~$y$ vertices, 
\begin{align*}
\log Z_F(\lam) \ge \frac{y}{a_u+1}\log(1+(a_u+1)\lam).
\end{align*}
\end{enumerate}
\end{lemma}

\begin{proof}
Let~$u$ be an arbitrary vertex of~$G$, and suppose that~$F\subset G[N(u)]$ has~$y$ vertices. 
If~$y\ge1$, then~$F$ has average degree at most~$a_u$, and
Lemma~\ref{lem:sparse,hcm} directly yields~\ref{itm:mad,largeZ};
if $y=0$, that assertion is immediate.
For~\ref{itm:mad,lo}, when~$y\ge1$ we have~$Z_F(\lam)\le {(1+\lam)}^y$, and Lemma~\ref{lem:sparse,hcm} gives
\[
\beta_u\frac{\lam}{1+\lam}\frac{1}{Z_F(\lam)} + \gam_u \frac{\lam Z'_F(\lam)}{Z_F(\lam)}\ge
    \frac{\lam}{1+\lam}\Big(\beta_u{(1+\lam)}^{-y} + \gam_u y{(1+\lam)}^{-a_u}\Big),
\]
while for~$y=0$ the same inequality holds with equality.  Define its
right-hand side to be~$g(y)$.

The function~$g$ is strictly convex with a stationary minimum at
\[
y^* = a_u + \frac{\log\left(\frac{\beta_u}{\gam_u}\log(1+\lam)\right)}{\log(1+\lam)},
\]
and if we set~$g(y^*)=1$ for strong local $(\beta_u,\gam_u)$-occupancy, and solve for~$\beta_u$ we obtain the definition given in the statement of the lemma. 
Then the function~$\beta_u+\gam_u D_u$ is strictly convex in~$\gam_u$, noting that~$\frac{\partial^2}{\partial x^2}x\cdot \exp(c/x) = \frac{c^2}{x^3}\exp(c /x)$ for any constant~$c$,
and the given~$\gam_u$ is the unique minimiser.
    One checks that, indeed, setting~$\beta_u$ and~$\gamma_u$ to the announced values, and writing~$X_u=D_u{(1+\lam)}^{a_u}\log(1+\lam)$, we have
    \[
        \beta_u = \frac{1+\lambda}{\lambda}\cdot\frac{X_u}{W(X_u)\cdot(1+W(X_u))},
    \]
and hence
    \[\beta_u+\gamma_u D_u=\frac{1+\lam}{\lam}\left(\frac{X_u}{W(X_u)(1+W(X_u))}+\frac{X_u}{1+W(X_u)}\right)
    =\frac{1+\lam}{\lam} \frac{X_u}{W(X_u)}\]
as announced.
Furthermore,
    \[ y^* = \frac{W(X_u)}{\log(1+\lambda)},\]
and hence indeed
	    \begin{align*}
	        g(y^*) &= \frac{X_u\cdot{(1+\lambda)}^{-W(X_u)/\log(1+\lambda)}}{W(X_u)(1+W(X_u))} + \frac{W(X_u)}{1+W(X_u)}\\
	               &= \frac{1}{1+W(X_u)}+\frac{W(X_u)}{1+W(X_u)}=1. 
	    \end{align*}
Since $g(y)\ge g(y^*)=1$ also when $y=0$, this proves
\ref{itm:mad,lo} for the empty graph as well.

For $p$-chromatic local $(\beta,\gam)$-occupancy, eventually $b_u>1$
uniformly in~$u$, so the empty graph, whose partition function is~$1$,
is covered by the first alternative in Lemma~\ref{lem:chromatic,threshold}.
Fix a nonempty graph $F\subset G[N(u)]$ with $z:=Z_F(\lam)>b_u$.
Lemma~\ref{lem:sparse,hcm} gives $\log z\ge c_uy$, and hence
\[
    y\le\frac{\log z}{c_u}.
\]
Consequently,
\[
    \frac{y\lam}{\log z}\le\frac{\lam}{c_u}.
\]
The exact form~\eqref{eq:uniformAlon,asymptotic} of Lemma~\ref{lem:uniformAlon}, together with~$\log(1/(1-p))=o(\log
b_u)$, therefore gives uniformly for~$u$ and~$z>b_u$
	\[
	Q_{F,p}(\lam)\ge(1-o(1))\frac{\log b_u}{\fn(\lam/c_u)}.
	\]
	Thus the hypothesis of Lemma~\ref{lem:chromatic,threshold} holds with
	\[
	    T_u=b_u
	    \qquad\text{and}\qquad
	    A_u=(1-o(1))\frac{\log b_u}{\fn(\lam/c_u)}.
	\]
	Since $(1+\lam)T_u/\lam=\beta_u$, that lemma gives $p$-chromatic local
	$(\beta,\gam)$-occupancy with
	\[
	    \gam_u=\frac{1}{pA_u}
	    =(1+o(1))\frac{\fn(\lam/c_u)}{p\log b_u},
	\]
and hence for every larger choice of~$\gam_u$.

For the final assertion,~$a_u\lam\to0$ gives~$\lam/c_u=1+o(1)$ and hence~$K_u=1+o(1)$.  The displayed choice
has~$b_u=e^{w_u}$ and satisfies the preceding hypotheses.  Since~$w_u^2e^{w_u}=d_uK_u/(p\ell)$, we have
\begin{align}
\frac{\lam\ell}{1+\lam}\beta_u+d_u\gam_u
&\sim\ell e^{w_u}+\frac{d_uK_u}{pw_u}
\sim\frac{d_u}{\log(d_u/\ell)},
\end{align}
by the standard asymptotics of~$W$.
\end{proof}

\noindent
Note the parameters~\eqref{eq:betagam,triangle} obtained in an earlier work~\cite{DJKP20} for triangle-free
graphs arise in the special case of statement~\ref{itm:mad,lo} of the above lemma when~$a_u=0$
and~$d_u=\Delta$ for all~$u$.

\subsection{An excluded cycle length}\label{sub:Cfree}
If a graph contains no subgraph isomorphic to the cycle~$C_k$, then the local path length in~$G$ is bounded. 
It follows from a theorem of Erd\H{o}s and Gallai~\cite[Theorem~(2.6)]{ErGa59} that we may use
Lemma~\ref{lem:mad} with~$a_u=k-3$ in particular, and then apply the main framework to prove the following
result.

\begin{theorem}\label{thm:Cfree}
    For any graph~$G$ of maximum degree~$\Delta$ that contains no subgraph isomorphic to~$C_k$, for some~$k\in\{3,\dotsc,\Delta+1\}$, the following statements hold.
\begin{enumerate}
\item\label{itm:Cfree,occfrac}
For any~$\lam>0$, the occupancy fraction of the hard-core model on~$G$ at fugacity~$\lam$ satisfies
\[
    \frac{1}{|V(G)|}\frac{\lam Z_G'(\lam)}{Z_G(\lam)} \ge \max_{0<x\le\lam}\left\{ \frac{x}{1+x} \frac{W(\Delta{(1+x)}^{k-3}\log(1+x))}{\Delta{(1+x)}^{k-3}\log(1+x)} \right\}.
\]
\item\label{itm:Cfree,chif}
For any $\eps>0$ there exists $\delta_0$ such that, writing
$d_u^+=\max\{d_u,\delta_0k\}$, there is a fractional colouring of~$G$
where each~$u\in V(G)$ is coloured with a subset of the interval
\[
\interval[co,scaled]{0}{(1+\eps)\frac{d_u^+}{\log(d_u^+/k)}}.
\]
    In particular, if~$\Delta/k\to\infty$, then the fractional chromatic number of~$G$ satisfies~$\chi_f(G)
    \le (1+o(1))\Delta/\log (\Delta/k)$ as~$\Delta\to\infty$.
\item\label{itm:Cfree,chic}
If~$k=\Delta^{o(1)}$, then the correspondence chromatic number of~$G$ satisfies~$\chi_c(G) \le (1+o(1))\Delta/\log \Delta$ as~$\Delta\to\infty$.
\end{enumerate}
\end{theorem}

\begin{proof}
For every vertex~$u\in V(G)$, no subgraph~$F$ of~$G[N(u)]$ contains a $(k-1)$-vertex path.
Hence a theorem of Erd\H{o}s and Gallai~\cite[Theorem~(2.6)]{ErGa59} implies that~$F$ has average degree at most~$k-3$. 
Having at our disposal the analysis of the hard-core model in graphs with bounded local $\mad$ given by Lemma~\ref{lem:mad} with $a_u=k-3$ for all $u$, the theorem follows easily from the main framework. 

Setting~$D_u=\Delta$ for all~$u$, statement~\ref{itm:Cfree,occfrac} then follows from Theorem~\ref{thm:main,occfrac} and the fact that the occupancy fraction of the hard-core model at fugacity~$\lam$ is strictly increasing in~$\lam$~\cite[Prop.~1]{DJPR18}. 

Statement~\ref{itm:Cfree,chif} follows from Theorem~\ref{thm:main,chif} and some asymptotic analysis. 
	With~$D_u=d_u^+$ in Lemma~\ref{lem:mad} we deduce that~$G$ admits a fractional colouring where each vertex~$u$ is coloured by a subset of the interval~$\interval[co]{0}{c_u}$ with
\begin{align}
c_u 
  &\coloneqq \beta_u + \gam_ud_u
    \le \beta_u+\gam_ud_u^+
    \\&= \frac{1+\lam}{\lam} \frac{d_u^+{(1+\lam)}^{k-3}\log(1+\lam)}{W(d_u^+{(1+\lam)}^{k-3}\log(1+\lam))}.
\end{align}
Take~$1/\lam= k\log\delta_0$.
This gives~$\lam=o(1)$ and~$k\lam=o(1)$ as~$\delta_0\to\infty$ since~$k\ge 3$. It follows from the asymptotic
properties of~$W$ (see Section~\ref{sub:Lambert}) that~$c_u \le (1+o(1))d_u^+ / \log (d_u^+/k)$
as~$\delta_0\to\infty$. 

For~\ref{itm:Cfree,chic}, set $\lam=1/(k\log\log\Delta)$ and
\[
    p=\frac{1}{1+\lam},
    \qquad
    \ell=80k\log\Delta\log\log\Delta,
    \qquad
    \iota=\sqrt{\frac{7\log\Delta}{\ell}}.
\]
Since~$k=\Delta^{o(1)}$, we have~$\iota=o(1)$ and~$\ell=\Delta^{o(1)}$.  Lemma~\ref{lem:mad}\ref{itm:mad,largeZ}, with~$a_u=k-3$, shows that every subgraph~$F$ of a neighbourhood on at least~$\ell/8$ vertices satisfies~$Z_F(\lam)\ge8\Delta^4$ for all sufficiently large~$\Delta$.

Moreover, $\Delta/\ell\to\infty$, $\lam\to0$, $p\to1$, $(k-3)\lam\to0$, and
\[
    \log\frac{1}{1-p}=\log(1+k\log\log\Delta)
    =o\bigl(\log(\Delta/\ell)\bigr).
\]
Apply Lemma~\ref{lem:mad}\ref{itm:mad,chromatic} uniformly with $d_u=\Delta$.  Its optimising choice gives chromatic local $(\beta,\gam)$-occupancy and
\[
    \frac{1}{1-\iota}\left(\frac{\lam\ell}{1+\lam}\beta+\gam\Delta\right)
    =(1+o(1))\frac{\Delta}{\log(\Delta/\ell)}
    =(1+o(1))\frac{\Delta}{\log\Delta}.
\]
Moreover, $\gam=(1+o(1))/\log(\Delta/\ell)$.  Before applying Theorem~\ref{thm:main,chic}, repeatedly delete vertices of degree below~$2\ell\log(\Delta/\ell)$.  The remaining graph satisfies~\eqref{eq:ysmall}, while the deleted vertices can be coloured greedily afterwards from any list of size~$(1+o(1))\Delta/\log\Delta$.  Theorem~\ref{thm:main,chic} now proves the asserted correspondence chromatic bound.
\end{proof}

\subsection{Bounded local triangle count}\label{sub:sparsenbhds}

The following result, where we suppose that each vertex of~$G$ is contained in at most~$t$ triangles, implies
and elaborates upon Theorem~\ref{thm:sparsenbhds,chi}. With~$t=\Delta^2/f$ we regain the form of
Theorem~\ref{thm:sparsenbhds,chi}. 

\begin{theorem}\label{thm:sparsenbhds}
For any graph~$G$ of maximum degree~$\Delta$ in which each vertex is contained in at most~$t\ge 1/2$ triangles, the following statements hold.
\begin{enumerate}
\item\label{itm:sparsenbhds,occfrac}
For any~$\lam>0$, the occupancy fraction of the hard-core model on~$G$ at fugacity~$\lam$ satisfies
\[
    \frac{1}{|V(G)|}\frac{\lam Z_G'(\lam)}{Z_G(\lam)} \ge \max_{0<x\le\lam}\left\{ \frac{x}{1+x} \frac{W(\Delta{(1+x)}^{\sqrt{2t}}\log(1+x))}{\Delta{(1+x)}^{\sqrt{2t}}\log(1+x)} \right\}.
\]
\item\label{itm:sparsenbhds,chif}
For any~$\eps>0$ there exists~$\delta_0$ such that, writing~$d_u^+=\max\{d_u,\delta_0\sqrt t\}$, there is a
fractional colouring of~$G$ such that each~$u\in V(G)$ is coloured with a subset of the interval
\[
\interval[co,scaled]{0}{(1+\eps)\frac{d_u^+}{\log(d_u^+/\sqrt t)}}.
\]
In particular, if~$\Delta/\sqrt t\to\infty$, then the fractional chromatic number of~$G$ satisfies~$\chi_f(G)
\le (1+o(1))\Delta/\log(\Delta/\sqrt t)$ as~$\Delta\to\infty$.
\item\label{itm:sparsenbhds,chic}
Write
\[
    h=1+\sqrt{2t}
    \qquad\text{and}\qquad
    R=\frac{\Delta}{h\log\Delta\log\log\Delta}.
\]
If~$R\to\infty$, then the correspondence chromatic number of~$G$
satisfies
\[
    \chi_c(G)\le(1+o(1))\frac{\Delta}{\log R}.
\]
In particular, if~$t=\Delta^{o(1)}$, then~$\chi_c(G) \le (1+o(1))\Delta/\log \Delta$.
\end{enumerate}
\end{theorem}

\noindent
The requirement that~$t$ be at least~$1/2$ is not restrictive: for any~$t\in\interval[co]{0}{1}$ each vertex
being contained in at most~$t$ triangles is equivalent to being triangle-free. 
This merely helps us avoid issues with small~$t$ when choosing parameters.
Before proving Theorem~\ref{thm:sparsenbhds}, we compare it with related earlier results. 
For convenience we restate a non-local Theorem~\ref{thm:sparsenbhds}\ref{itm:sparsenbhds,chic} with~$t=\Delta^2/f$.

\begin{corollary}\label{cor:sparsenbhds}
For any~$\eps>0$ there exists~$\Delta_0$ such that the following holds for all~$\Delta\ge\Delta_0$.
    For any graph~$G$ of maximum degree~$\Delta$ in which each vertex is contained in at most~$\Delta^2/f$ triangles, where
    ${(\log\Delta)}^{4(1+\eps)/\eps}\le f\le \Delta^2+1$, we have
    $\chi_c(G)\le (1+\eps)\Delta/\log\sqrt f$.
\end{corollary}

\begin{proof}
Set~$t=\Delta^2/f$ and use the
notation in Theorem~\ref{thm:sparsenbhds}\ref{itm:sparsenbhds,chic}.
Since~$f\le\Delta^2+1$,
\[
    R\ge
    \frac{(1-o(1))\sqrt f}
         {(1+\sqrt2)\log\Delta\log\log\Delta}.
\]
The lower bound on~$f$ therefore gives~$R\to\infty$ and
\begin{align*}
    \log R
    &\ge \log\sqrt f-\log\log\Delta-\log\log\log\Delta-O(1)\\
    &\ge \left(1-\frac{\eps}{2(1+\eps)}-o(1)\right)\log\sqrt f.
\end{align*}
Theorem~\ref{thm:sparsenbhds}\ref{itm:sparsenbhds,chic} now gives the
result, since
\[
    \left(1-\frac{\eps}{2(1+\eps)}\right)^{-1}
    =\frac{2(1+\eps)}{2+\eps}<1+\eps.\qedhere
\]
\end{proof}

\noindent
The exponent in the hypothesis is not optimised; it simply absorbs
the auxiliary factor~$\log\Delta\log\log\Delta$ in~$R$.  This
polylogarithmic change has no effect on the splitting argument below.

Improving on earlier results of Alon, Krivelevich and Sudakov~\cite{AKS99} and Vu~\cite{Vu02}, a statement similar to Corollary~\ref{cor:sparsenbhds} was proved by Achlioptas, Iliopoulos and Sinclair~\cite[Thm.~II.5]{AIS19} for the list chromatic number. They however required a much stronger lower bound on~$f$ of the form $f\ge \Delta^{\frac{2+2\eps}{1+2\eps}}{(\log\Delta)}^2$, this last expression being roughly~$\Delta^{2-\eps}$. 
They used this weaker statement together with a known reduction~\cite{AKS99} from the `small~$f$' case to the easier `large~$f$' case to obtain a quantitatively weaker bound than in Theorem~\ref{thm:sparsenbhds,chi} for chromatic number.
Armed with our stronger Corollary~\ref{cor:sparsenbhds}, we can perform this same reduction but without a noticeable degradation of the leading constant to obtain Theorem~\ref{thm:sparsenbhds,chi}. After next showing Theorem~\ref{thm:sparsenbhds}, we give the reduction in Subsection~\ref{sub:sparsenbhds,proof2}.
A large part of the proof of Theorem~\ref{thm:sparsenbhds} is omitted, being nearly identical to the
corresponding part in the proof of Theorem~\ref{thm:Cfree}.

\begin{proof}[Proof of Theorem~\ref{thm:sparsenbhds}]
Fix a vertex~$u\in V(G)$ and any subgraph~$F\subseteq G[N(u)]$ on~$y$ vertices.
By assumption,~$F$ contains at most~$t$ edges. It follows that the average degree of~$F$ is at most
\[
\min\left\{y-1,\frac{2t}{y}\right\} \le \sqrt{2t}.
\]
Indeed, the first bound is immediate and the second follows from the
handshaking lemma.  Their minimum is at most their geometric mean,
which is at most~$\sqrt{2t}$.
The first two assertions follow from Lemma~\ref{lem:mad} with
$a_u=\sqrt{2t}$ for all~$u$, exactly as in the proof of
Theorem~\ref{thm:Cfree}.

For~\ref{itm:sparsenbhds,chic}, write~$h=1+\sqrt{2t}$,
$L=\log\log\Delta$, and assume that~$R=\Delta/(hL\log\Delta)\to\infty$.
Set
\[
    \lam=\frac{1}{hL},\qquad
    p=1-\frac{1}{\log R},\qquad
    \ell=80hL\log\Delta,\qquad
    \iota=\sqrt{\frac{7\log\Delta}{\ell}}.
\]
These parameters are well defined for all sufficiently large~$\Delta$.
Moreover,~$\iota=o(1)$ and condition~\eqref{eq:elllarge} holds.
Lemma~\ref{lem:mad}\ref{itm:mad,largeZ} shows that, whenever~$F\subseteq G[N(u)]$ has~$y\ge\ell/8$ vertices,
\[
    \log Z_F(\lam)
    \ge\frac{y}{h}\log(1+h\lam)
    \ge10L\log\Delta\log(1+1/L)
    >\log(8\Delta^4),
\]
so condition~\eqref{eq:Zlarge} also holds.

We have~$\Delta/\ell=R/80\to\infty$, $\lam\to0$, $p\to1$,
\[
    \log\frac{1}{1-p}=\log\log R=o(\log R),
    \qquad
    \log(\Delta/\ell)\sim\log R,
    \qquad
    \sqrt{2t}\,\lam=o(1).
\]
Thus Lemma~\ref{lem:mad}\ref{itm:mad,chromatic}, applied uniformly
with~$d_u=\Delta$, gives $p$-chromatic local $(\beta,\gam)$-occupancy
and
\[
    \frac{1}{1-\iota}
    \left(\frac{\lam\ell}{1+\lam}\beta+\gam\Delta\right)
    =(1+o(1))\frac{\Delta}{\log R},
    \qquad
    \gam=(1+o(1))\frac{1}{p\log R}.
\]
Before applying Theorem~\ref{thm:main,chic}, repeatedly delete vertices
of degree below~$2\ell\log R$.  The remaining graph
satisfies~\eqref{eq:ysmall}.  Also~$2\ell\log R=o(\Delta/\log R)$, so the deleted vertices can be
coloured greedily afterwards from the same lists.
Theorem~\ref{thm:main,chic} completes the proof.
\end{proof}

\subsection{Proof of Theorem~\ref{thm:sparsenbhds,chi}}\label{sub:sparsenbhds,proof2}

Due to the condition on~$f$, Corollary~\ref{cor:sparsenbhds} does not directly imply
Theorem~\ref{thm:sparsenbhds,chi}, at least not via~\eqref{eq:params}. Instead we appeal to an iterative
splitting procedure that was used similarly before~\cite{AKS99,AIS19}. We include the details of the
derivation for completeness. This requires two results which are shown with the help of the Lov\'{a}sz local
lemma.

\begin{lemma}[\cite{AKS99}]\label{lem:AKS}
For any graph~$G$ of maximum degree~$\Delta\ge2$ in which the neighbourhood of every vertex spans at most~$s$ edges, 
there exists a partition~$V(G)=V_1\cup V_2$ such that~$V_i$ for~$i\in\{1,2\}$ induces a subgraph of maximum
degree at most~$\Delta/2 + 2\sqrt{\Delta\log\Delta}$ in which the neighbourhood of every vertex spans at most~$s/4 + 2\Delta^{3/2}\sqrt{\log\Delta}$ edges.
\end{lemma}

\begin{lemma}[\cite{AIS19}, cf.~also~\cite{AKS99}]\label{lem:AIS} 
    Given~$\Delta$ and~$f$ sufficiently large, define the sequences~${(\Delta_t)}_{t\ge 0}$ and~${(s_t)}_{t\ge 0}$ as follows: $\Delta_0=\Delta$, $s_0=\Delta^2/f$, and
\begin{align}
\Delta_{t+1} &= \Delta_t/2 + 2\sqrt{\Delta_t\log\Delta_t},&
s_{t+1}      &= s_t/4 + 2\Delta_t^{3/2}\sqrt{\log\Delta_t}.
\end{align}
    For any~$\delta\in\interval[oo]{0}{1/100}$ and~$\zeta>0$ such that~$\zeta(2+\delta)<1/10$, let~$j$ be the smallest nonnegative integer for which~$f > {\left((1+\delta)\Delta/2^j\right)}^{\zeta(2+\delta)}$. 
    Then~$\Delta_j \le(1+\delta)\Delta/2^j$ and~$s_j\le{\left((1+\delta)\Delta/2^j\right)}^2/f$.
\end{lemma}

\begin{proof}[Proof of Theorem~\ref{thm:sparsenbhds,chi}]
  Let~$\eps>0$. It suffices to prove that~$\chi(G)\le (1+\eps)\Delta/\log \sqrt{f}$ for~$f$ (and thus~$\Delta$) sufficiently large.
We may assume that~$f\le \Delta^{\eps^2(2+\eps^2)}$, otherwise we can apply Corollary~\ref{cor:sparsenbhds}. 
Without loss of generality, we may also assume that~$\eps\le 1/11$. 
    Let~$\delta = \zeta = \eps^2$ (which is indeed less than~$1/100$) so that~$\zeta(2+\delta)<1/10$ and we may apply Lemma~\ref{lem:AIS}. Let~$j=j(\Delta, f,\delta,\zeta)$ be the integer given therein and, starting with the trivial partition~$\{V(G)\}$, iterate the following procedure~$j$ times to form a partition of~$V(G)$. 

In one iteration of the procedure, for each part~$W$ of the current partition, do the following. 
If the induced subgraph~$G[W]$ has maximum degree at most~$1$ then do nothing, or else split~$W$ into two parts as given by Lemma~\ref{lem:AKS}. 

    The ultimate partition of~$V(G)$ yields at most~$2^j$ induced subgraphs of~$G$, and by Lemma~\ref{lem:AIS} each such subgraph~$H$ has maximum degree at most~$1$, or it has maximum degree at most~$\Delta_*\coloneqq (1+\delta)\Delta/2^j$ and the property that every neighbourhood of~$H$ spans at most~$\Delta_*^2/f$ edges. Observe that~$\Delta_*^{\zeta(2+\delta)}<f \le{(2\Delta_*)}^{\zeta(2+\delta)}$ due to the choice of~$j$ in Lemma~\ref{lem:AIS} and the fact that~$f\le \Delta^{\eps^2(2+\eps^2)} \le {((1+\delta)\Delta/2^0)}^{\zeta(2+\delta)}$ ---which implies that~$j\ge1$.  In particular,~$f\ge(\log\Delta_*)^{4(1+\zeta)/\zeta}$ for all sufficiently
large~$\Delta_*$.
Now either~$\chi(H)\le 2$ by~\eqref{eq:params,basic}, or~$\chi(H)\le (1+\zeta)\Delta_*/\log\sqrt{f}$ by
Corollary~\ref{cor:sparsenbhds}.
Therefore 
\[
\chi(G)\le 2^j\max\left\{2,\, (1+\zeta)\frac{\Delta_*}{\log\sqrt{f}}\right\} = \max\left\{2^{j+1},\, (1+\zeta)\frac{(1+\delta)\Delta}{\log\sqrt{f}}\right\}.
\]
We also have
\[
\frac{(1+\delta)\Delta}{\Delta/(4\log\sqrt{f})} =2(1+\delta)\log f < f^{\frac{1}{\zeta(2+\delta)}} \le 2\Delta_*=\frac{(1+\delta)\Delta}{2^{j-1}},
\]
where the first inequality follows from a choice of large enough $f$ and the second holds by the above range for~$f$ in terms of~$\Delta_*$.  This bounds the first argument in the maximisation as it yields
    that~$2^{j+1}\le \Delta/\log\sqrt{f}$.
    For the second argument, note that~$(1+\zeta)(1+\delta) = {(1+\eps^2)}^2 \le (1+\eps)$ since~$\eps\le1/11$.
\end{proof}

\subsection{Bounded local cycle count}\label{sub:boundedC}

In Subsections~\ref{sub:Cfree} and~\ref{sub:sparsenbhds}, we showed that our analysis of the hard-core model under bounded local~$\mad$ can be applied effectively to graphs with no~$C_k$ and graphs for which each vertex is contained in few triangles. We next note how to combine these two ideas, hinting at some further flexibility in our approach.

\begin{theorem}\label{thm:boundedC}
For any graph $G$ of maximum degree~$\Delta$ in which the subgraph induced by each neighbourhood contains at most~$t$ copies of~$P_{k-1}$, where~$t\ge1/2$ and~$k\ge3$,
the following statements hold.
\begin{enumerate}
\item\label{itm:boundedC,occfrac}
For any $\lam>0$, the occupancy fraction of the hard-core model on $G$ at fugacity $\lam$ satisfies
\[
    \frac{1}{|V(G)|}\frac{\lam Z_G'(\lam)}{Z_G(\lam)} \ge \max_{0<x\le\lam}\left\{ \frac{x}{1+x} \frac{W(\Delta{(1+x)}^{k-3+\sqrt{2t}}\log(1+x))}{\Delta{(1+x)}^{k-3+\sqrt{2t}}\log(1+x)} \right\}.
\]
\item\label{itm:boundedC,chif}
For any $\eps>0$ there exists $\delta_0$ such that, writing
$d_u^+=\max\{d_u,\delta_0(k+\sqrt t)\}$, there is a fractional colouring
of~$G$ such that each~$u\in V(G)$ is coloured with a subset of the interval
\[
\interval[co,scaled]{0}{(1+\eps)\frac{d_u^+}{\log(d_u^+/(k+\sqrt t))}}.
\]
	In particular, if $\Delta/(k+\sqrt{t})\to\infty$, then the fractional chromatic number of $G$ satisfies
	        $\chi_f(G) \le (1+o(1))\Delta/\log(\Delta/(k+\sqrt{t}))$ as $\Delta\to\infty$.
\item\label{itm:boundedC4,chic}
If $k+\sqrt{t}=\Delta^{o(1)}$, then the correspondence chromatic number of~$G$ satisfies $\chi_c(G) \le (1+o(1))\Delta/\log \Delta$ as~$\Delta\to\infty$.
\end{enumerate}
\end{theorem}

\begin{proof}[Proof sketch]
Fix a vertex $u\in V(G)$ and any subgraph~$F\subset G[N(u)]$ on~$y$ vertices. 
By assumption,~$F$ contains at most~$t$ copies of~$P_{k-1}$.
It follows that the average degree of~$F$ is at most
\[
\min\left\{ y-1,\,k-3+ \frac{2t}{y}\right\} \le k-3 + \sqrt{2t}.
\]
Indeed, the first bound is straightforward as there are at most~$y-1$ possible neighbours for any vertex in~$F$.
For the second, we only need to remove at most~$t$ edges from~$F$ to destroy all copies of~$P_{k-1}$, so that by the result of Erd\H{o}s and Gallai~\cite[Theorem~(2.6)]{ErGa59} the remaining graph has at most~$y(k-3)/2$ edges. Thus~$F$ has at most $y(k-3)/2+t$ edges, implying the second bound.
We consider the subcases $y\le \sqrt{2t}$ and $y>\sqrt{2t}$ to crudely upper bound the minimum, from which the statement follows.
The first two assertions follow from Lemma~\ref{lem:mad} with $a_u=k-3+\sqrt{2t}$ for all~$u$.  For~\ref{itm:boundedC4,chic}, the final part of the proof of Theorem~\ref{thm:Cfree} applies with $k$ replaced by $k+\sqrt{2t}$.
\end{proof}

\subsection{Neighbourhoods with an excluded bipartite subgraph}

We remark that by other classic extremal results, conclusions akin to those in Theorem~\ref{thm:Cfree} hold for a graph containing no subgraph of the form~$K_1 + T_{k-1}$, where~$T_{k-1}$ denotes some arbitrary $(k-1)$-vertex tree.
In fact, similar results hold (varying in the leading constants) for graphs containing no $K_1+H$ where $H$ is bipartite. 
Such observations were already made in~\cite{AKS99}, where one uses their analogue of Theorem~\ref{thm:sparsenbhds,chi} which has a larger, unspecified leading constant. 
Our framework permits the transfer of bounds on the extremal number of $H$ to chromatic number bounds in $(K_1+H)$-free graphs with better leading constants than via the reduction in~\cite{AKS99}. 
For example, in another `smooth' extension of Theorem~\ref{thm:molloy} one can also show with our framework that for fixed $s\ge t\ge 1$ and $G$ with no $K_1 + K_{s,t}$, we have the bound $\chi(G)\le (t+o(1))\Delta/\log\Delta$ as $\Delta\to\infty$. 
Of course one can also show a suite of local, list, and correspondence strengthenings, and even a `few copies' version along the lines of Theorem~\ref{thm:boundedC}. 
We do not give the details as the method is essentially the same as the other proofs we present in this section. 
The main idea is to use an extremal number result to bound $\mad(G[N(u)])$ for all $u\in V(G)$ (in the case mentioned above we use the well-known result on the Zarankiewicz Problem of Kővári, Sós, and Turán~\cite{KST54}), and then make suitable choices of $\ell$ and $\lam$. 
With the main work done by our framework and the details given already in this section, these remaining tasks are quite straightforward. 


\section{Bounded local Hall ratio}\label{sec:localhall}

The following result generalises Theorem~\ref{thm:localhall,chi}.
Recall that the Hall ratio of a graph~$G$ is~$\rho(G) = \max\{|V(H)|/\alpha(H)\mid H\subseteq G\}$.

\begin{theorem}\label{thm:localhall}
Let $K\colon\interval[co]{1}{\infty}\to\interval[co]{1}{\infty}$ be defined in terms of the lower branch of the Lambert-W function by $\fn(y) \coloneqq -W_{-1}(-1/(ey))$. Note that $\fn(y)\sim \log y$ as $y\to\infty$.

For any graph~$G$ of maximum degree~$\Delta$ in which the neighbourhood of every vertex $u\in V(G)$ induces a subgraph of Hall ratio at most $\rho_u\ge 1$, and with $\rho\coloneqq\max_{u\in V(G)}\rho_u$, the following statements hold. 
\begin{enumerate}
\item\label{itm:localhall,occfrac}
For any~$\lam>0$, setting $k(x)= \fn(\rho x/\log(1+x))$, the occupancy fraction of the hard-core model on~$G$ at fugacity~$\lam$ satisfies
\[
\frac{1}{|V(G)|}\frac{\lam Z_G'(\lam)}{Z_G(\lam)} \ge \max_{0<x\le\lam}\left\{ \frac{W(k(x)\Delta x/(1+x))}{k(x)\Delta} \right\}.
\]
\item\label{itm:localhall,chif}
For any~$\eps > 0$ there exists~$\delta_0$ such that, writing~$d_u^+=\max\{d_u,\delta_0\}$, there is a fractional colouring of~$G$
such that each~$u\in V(G)$ is coloured with a subset of the interval
\[
\interval[co,scaled]{0}{(1+\eps)\frac{\fn(\rho_u)d_u^+}{\log(\fn(\rho_u)d_u^+)}}.
\]
\item\label{itm:localhall,chic}
If $\rho=\Delta^{o(1)}$, then the correspondence chromatic number of $G$ satisfies $\chi_c(G)\le(1+o(1))\fn(\rho)\Delta/\log\Delta$ as $\Delta\to\infty$.
\end{enumerate}
\end{theorem}

\subsection{Local occupancy with bounded local Hall ratio}\label{sub:localhall,localocc}

We begin with the requisite local analysis of the hard-core model, which relies critically on Lemma~\ref{lem:generalizedShearer}.

\begin{lemma}\label{lem:localhall,localocc}
For any graph $G$ in which the neighbourhood of every vertex $u\in V(G)$ induces a subgraph of Hall ratio at most~$\rho_u\ge 1$, the following holds. 
\begin{enumerate}
\item\label{itm:localhall,hcm}
    For any $\lam>0$ and collection ${(d_u)}_{u\in V(G)}$ of positive reals, setting $k_u = K\big(\rho_u\lam/\log(1+\lam)\big)$, 
\begin{align*}
\beta_u &=\frac{1+\lam}{\lam}\cdot\frac{\exp(W(k_u d_u \lam/(1+\lam)))}{1+W(k_u d_u \lam/(1+\lam))}\quad\text{and}\\
\gam_u  &= \frac{k_u}{1+W(k_u d_u \lam/(1+\lam))},
\end{align*}
        we have strong local ${(\beta_u,\gam_u)}_u$-occupancy in the hard-core model on~$G$ at fugacity~$\lam$.
Moreover, with $k=\max_{u\in V(G)} k_u$ and~$\beta$ and~$\gam$ obtained by replacing~$k_u$ in the definitions of~$\beta_u$ and~$\gam_u$ by~$k$, the graph~$G$ has strong local $(\beta, \gam)$-occupancy.
\item\label{itm:localhall,largeZ}
For any $u\in V(G)$ and any subgraph~$F$ of~$G[N(u)]$ on~$y$ vertices,
\begin{align*}
\log Z_F(\lam) \ge 
 \frac{y}{\rho_u}\log(1+\lam).
\end{align*}
\item\label{itm:localhall,chromatic}
For this item, consider an asymptotic family of graphs and parameters satisfying the hypotheses below, and
interpret all asymptotic relations as holding uniformly over~$u$.
For~$\lam\in\interval[oc]{0}{1}$, $p\in\interval[oo]{0}{1}$, and any collection~$\beta\in(\RR^+)^{V(G)}$ of positive reals,
let~$b_u=\beta_u\lam/(1+\lam)$ and suppose that
\[
    b_u\to\infty,\qquad
    \log\frac{1}{1-p}=o(\log b_u).
\]
Then the hard-core model on~$G$ at fugacity~$\lam$ has $p$-chromatic
local $(\beta,\gam)$-occupancy for a choice of~$\gam$ satisfying
\[
\gam_u=(1+o(1))\frac{1}{p\log b_u}\fn\left(\frac{\rho_u\lam}{\log(1+\lam)}\right),
\]
and every componentwise larger choice of~$\gam$ also works.
\end{enumerate}
\end{lemma}

\begin{proof}
Fix an arbitrary vertex~$u\in V(G)$, and let~$F$ be any subgraph of~$G[N(u)]$ on~$y$ vertices. 
Since~$G[N(u)]$ has Hall ratio at most~$\rho_u$, the graph~$F$ contains an independent set of size at least~$y/\rho_u$. 
This immediately implies part~\ref{itm:localhall,largeZ}.
For the estimate from Lemma~\ref{lem:generalizedShearer}, assume
$y\ge1$; the empty case of the occupancy assertion is checked below.
We can now bound~$\lam Z'_F(\lam)/Z_F(\lam)$ from below with Lemma~\ref{lem:generalizedShearer}:
\[
\frac{\lam Z_F'(\lam)}{Z_F(\lam)} \ge 
\frac{\log Z_F(\lam)}{\fn(y\lam/\log Z_F(\lam))} \ge 
\frac{\log Z_F(\lam)}{\fn(\rho_u\lam/\log(1+\lam))} = 
\frac{\log Z_F(\lam)}{k_u},
\]
where the second inequality follows from applying part~\ref{itm:localhall,largeZ} in the denominator. 

For part~\ref{itm:localhall,hcm}, we start working with arbitrary
    positive reals~$\beta_u, \gam_u$, and show that choosing them as in the statement gives the desired properties.
We have 
\[
\beta_u\frac{\lam}{1+\lam} \frac{1}{Z_F(\lam)} + \gam_u\frac{\lam Z_F'(\lam)}{Z_F(\lam)} \ge 
\beta_u\frac{\lam}{1+\lam}e^{-\log Z_F(\lam)} + \gam_u\frac{\log Z_F(\lam)}{k_u},
\]
which we can minimise over all possible nonnegative values of $\log Z_F(\lam)$.
Let the function $g\colon\RR^+\to\RR^+$ be defined such that the right-hand side above is~$g(\log Z_F(\lam))$, and note that~$k_u$ is independent of~$F$.
It is easy to verify that $g''(z) = \beta_u \frac{\lam}{1+\lam}e^{-z} > 0$, and hence that~$g$ is strictly convex. 
Then we bound~$g(\log Z_F(\lam))$ from below by finding the unique stationary point of~$g(z)$, which must be a global minimum. 
This occurs at
\[
z^*=\log \left(\frac{\beta_u}{\gam_u}\frac{\lam}{1+\lam}k_u\right),
\]
giving 
\[
g\big(\log Z_F(\lam)\big) \ge g(z^*) = \frac{\gam_u}{k_u}\left(1+\log \left(\frac{\beta_u}{\gam_u}\frac{\lam}{1+\lam}k_u\right)\right).
\]
This is equal to~$1$ when~$\beta_u$ is given in terms of~$\gam_u$ by 
\begin{equation}\label{eq:localhall,beta}
\beta_u = \frac{\gam_u}{k_u}\frac{1+\lam}{\lam}e^{\frac{k_u}{\gam_u}-1},
\end{equation}
and then taking~$\gam_u$ as in the statement of the lemma, that is,
\[
\gam_u = \frac{k_u}{1+W(k_ud_u  \lam/(1+\lam))},
\]
means~$\beta_u$ as given by~\eqref{eq:localhall,beta} agrees with~$\beta_u$ as in the statement of the lemma. 
	    This completes the proof that~$G$ has strong ${(\beta_u, \gam_u)}_u$-local occupancy.
For $y=0$, the same conclusion follows directly from
$g(0)\ge g(z^*)=1$.

For uniform local occupancy parameters, it suffices to show that~$\beta_u$ and~$\gam_u$ increase with~$k_u$, and hence with~$\rho_u$, since increasing either parameter retains local occupancy. Since~$K$ is increasing, $k_u$ increases with~$\rho_u$. Set~$t=d_u\lam/(1+\lam)$. The dependence of~$\gam_u$ on~$k_u$ is increasing because
\[
\frac{\mathrm d}{\mathrm dx}\left(\frac{x}{1+W(xt)}\right)
=\frac{W(xt)^2+W(xt)+1}{(1+W(xt))^3}>0.
\]
Similarly, the nonconstant factor in~$\beta_u$ is increasing because
\[
\frac{\mathrm d}{\mathrm dx}\left(\frac{\exp(W(xt))}{1+W(xt)}\right)
=\frac{\exp(W(xt))W(xt)^2}{x(1+W(xt))^3}>0.
\]

We note in addition that one can show the choice of~$\gam_u$ (with~$\beta_u$ as in~\eqref{eq:localhall,beta}) minimises $\beta_u+\gam_u d_u$ with some calculus similar to the analysis of~$g$. 

For~\ref{itm:localhall,chromatic}, eventually~$b_u>1$ uniformly
in~$u$, so the empty graph is covered by the first alternative in
Lemma~\ref{lem:chromatic,threshold}.  Fix a nonempty subgraph~$F$
of~$G[N(u)]$ with~$z=Z_F(\lam)>b_u$.  The bound
in~\ref{itm:localhall,largeZ} gives
\[
    y\le\frac{\rho_u\log z}{\log(1+\lam)}.
\]
Thus
\[
\frac{y\lam}{\log z}\le\frac{\rho_u\lam}{\log(1+\lam)}.
\]
The exact form~\eqref{eq:uniformAlon,asymptotic} of Lemma~\ref{lem:uniformAlon} and the hypothesis on~$b_u$
now yield, uniformly over~$u$,
\[
Q_{F,p}(\lam)\ge(1-o(1))\frac{\log b_u}{\fn\left(\rho_u\lam/\log(1+\lam)\right)},
\]
	so the hypothesis of Lemma~\ref{lem:chromatic,threshold} holds with
	\[
	    T_u=b_u
	    \qquad\text{and}\qquad
	    A_u=(1-o(1))
	       \frac{\log b_u}
	            {\fn\left(\rho_u\lam/\log(1+\lam)\right)}.
\]
Since $(1+\lam)T_u/\lam=\beta_u$, that lemma gives the claimed
	$p$-chromatic local occupancy with
	\[
	    \gam_u=\frac{1}{pA_u}
    =(1+o(1))
      \frac{\fn\left(\rho_u\lam/\log(1+\lam)\right)}
           {p\log b_u},
\]
and hence for every larger choice of~$\gam_u$.
\end{proof}

\subsection{Proof of Theorem~\ref{thm:localhall}}\label{sub:localhall,proof}

\begin{proof}[Proof of Theorem~\ref{thm:localhall}]
We simply apply the main framework to derive the results. 
    Lemma~\ref{lem:localhall,localocc} gives for each collection~${(d_u)}_{u\in V(G)}$ a choice for every~$u\in V(G)$ of~$\beta_u$ and $\gam_u$ such that the hard-core model on $G$ at fugacity~$\lam$ has strong local ${(\beta_u, \gam_u)}_u$-local occupancy, and we compute
\[
\beta_u+\gam_u d_u = \frac{k_u d_u}{W(k_u d_u \lam/(1+\lam))}.
\]
We will choose~$d_u$ carefully to obtain each part of the theorem.

For~\ref{itm:localhall,occfrac} we are interested in local occupancy with uniform parameters but Lemma~\ref{lem:localhall,localocc} also gives suitable~$\beta$ and~$\gam$ for this.
Choosing~$d_u=\Delta$, applying Theorem~\ref{thm:main,occfrac}, and recalling that occupancy fraction is monotone increasing in~$\lam$ yields~\ref{itm:localhall,occfrac}.

	Statement~\ref{itm:localhall,chif} follows from Theorem~\ref{thm:main,chif}
	and some asymptotic analysis.  Apply
	Lemma~\ref{lem:localhall,localocc} with its parameter~$d_u$ set to
	$d_u^+=\max\{d_u,\delta_0\}$.  For any~$\lam>0$ we have
\[
\beta_u+\gam_ud_u \le \frac{d_u^+ k_u}{W(d_u^+ k_u \lam/(1+\lam))},
\]
	and, with $\lam=1/\log\delta_0$, for large enough~$\delta_0$ in terms of~$\eps$,
\[
\beta_u+\gam_ud_u \le (1+\eps)\frac{\fn(\rho_u)d_u^+}{\log(\fn(\rho_u)d_u^+)},
\]
as required for~\ref{itm:localhall,chif}. 
For the asymptotic properties of $K$ and $W$ necessary for the final step we refer to Subsection~\ref{sub:Lambert}.

For~\ref{itm:localhall,chic}, set $\lam=1/\log\log\Delta$,
$p=1/(1+\lam)$, and
\[
    \ell=80\rho\log\Delta\log\log\Delta,
    \qquad
    \iota=\sqrt{\frac{7\log\Delta}{\ell}},
    \qquad
    K=\fn\left(\frac{\rho\lam}{\log(1+\lam)}\right).
\]
As $\rho=\Delta^{o(1)}$, we have $\iota=o(1)$, $K=(1+o(1))\fn(\rho)$, and $\ell=\Delta^{o(1)}$.  Choose
\[
    w=2W\left(\frac12\sqrt{\frac{\Delta K}{\ell}}\right),
    \qquad
    \beta=\frac{1+\lam}{\lam}e^w,
    \qquad
    \gam=(1+o(1))\frac{(1+\lam)K}{w}.
\]
Here $w=(1+o(1))\log\Delta$, so
\[
    \frac{\beta\lam}{1+\lam}=e^w\to\infty
    \qquad\text{and}\qquad
    \log\frac{1}{1-p}=\log(1+\log\log\Delta)=o(w).
\]
Lemma~\ref{lem:localhall,localocc}\ref{itm:localhall,chromatic} gives chromatic local $(\beta,\gam)$-occupancy.  Its large-partition-function assertion gives~\eqref{eq:Zlarge}; and, after greedily deleting vertices of degree below $2\ell w/K$, it also gives~\eqref{eq:ysmall}.  The deleted vertices can be coloured afterwards from lists of size $(1+o(1))K\Delta/\log\Delta$.  Finally,
\[
\frac{1}{1-\iota}\left(\frac{\lam\ell}{1+\lam}\beta+\gam\Delta\right)
=(1+o(1))\frac{K\Delta}{w}
=(1+o(1))\frac{\fn(\rho)\Delta}{\log\Delta},
\]
so Theorem~\ref{thm:main,chic} proves the result.
\end{proof}


\section{Bounded clique number}\label{sec:clique}

The following result generalises Theorem~\ref{thm:clique,chi}.

\begin{theorem}\label{thm:clique}
    For any graph $G$ of maximum degree~$\Delta$ in which the largest clique containing any vertex~$u\in V(G)$ has size at most~$\omega_u\ge 3$, the following statements hold. 
    We write $\omega = \max_{u\in V(G)}\{\omega_u\}$ for an upper bound on the clique number of~$G$.
    \begin{enumerate}
    \item\label{itm:clique,occfrac}
    For any $\lam>0$, the occupancy fraction of the hard-core model on~$G$ at fugacity~$\lam$ satisfies as~$\Delta\to\infty$,
    \[
    \frac{1}{|V(G)|}\frac{\lam Z_G'(\lam)}{Z_G(\lam)} \ge (1-o(1))\max\left\{ \frac{\log \Delta}{(\omega-2)\Delta\log\log\Delta}, \frac{1}{2\Delta}\sqrt{\frac{\log \Delta}{\log(\omega-1)}} \right\}.
    \]
    \item\label{itm:clique,chif}
    For any~$\eps > 0$ there exists~$\delta_0$ such that, writing~$d_u^+=\max\{d_u,\delta_0\}$, there is a fractional colouring of~$G$
    such that each~$u\in V(G)$ is coloured with a subset of the smaller of the following two intervals:
    \begin{align}
    &\interval[co,scaled]{0}{(1+\eps)(\omega_u -2) \frac{d_u^+ \log\log d_u^+}{\log d_u^+}},
    \\
    &\interval[co,scaled]{0}{(1+\eps)2d_u^+\sqrt{\frac{\log(\omega_u-1)}{\log d_u^+}}}.
    \end{align}
    In particular, the fractional chromatic number of~$G$ satisfies 
    \[
      \chi_f(G)\le(1+o(1))\min\left\{(\omega -2) \frac{\Delta \log\log \Delta}{\log \Delta} ,\, 2\Delta\sqrt{\frac{\log(\omega-1)}{\log\Delta}} \right\}
    \] 
    as $\Delta\to\infty$.
    \item\label{itm:clique,chic}
    For all $\eps>0$ there exist $\delta_0$ and~$\Delta_0$ such that the following holds for all~$\Delta\ge\Delta_0$ and~$\delta=\delta_0 k$ where
    \begin{align}
      k &= \min\left\{
          {\left(e^2\log(8\Delta^4)\right)}^{\omega-1},\,
        \exp\left(\sqrt{4\log(\omega-1)\cdot(1+\eps)\log(8\Delta^4)}\right)
      \right\}.
      \end{align}
    Write $\overline d_u=\max\{d_u,\delta\}$.  If $\sH=(L,H)$ is a cover of $G$ such that for each $u\in V(G)$, the size of~$L(u)$ is at least the smaller of the two expressions below,
    \begin{align}
      &(1+\eps)(\omega_u -2)
          \frac{\overline d_u\log\log(\overline d_u/k)}
               {\log(\overline d_u/k)},
      \\
      &(1+\eps)2\overline d_u
          \sqrt{\frac{\log(\omega_u-1)}
                       {\log(\overline d_u/k)}},
    \end{align}
    then $G$ is $\sH$-colourable.
    
    In particular, the correspondence chromatic number of~$G$ satisfies 
    \[
      \chi_c(G)\le(1+o(1))\min\left\{(\omega -2) \frac{\Delta \log\log \Delta}{\log \Delta} ,\, 5\Delta\sqrt{\frac{\log(\omega-1)}{\log\Delta}} \right\}
    \] 
    as $\Delta\to\infty$.
    \end{enumerate}
\end{theorem}

For convenience, we restate the independence number bound that either of parts~\ref{itm:clique,occfrac}
and~\ref{itm:clique,chif} above immediately yields.

\begin{corollary}\label{cor:clique,alpha}
For any graph~$G$ of clique number~$\omega\ge3$ and maximum degree~$\Delta$, the independence number of~$G$ satisfies as~$\Delta\to\infty$
\begin{align}
\alpha(G) \ge (1-o(1))|V(G)|\max\left\{
\frac{\log\Delta}{(\omega-2)\Delta\log \log \Delta},
\frac{1}{2\Delta}\sqrt{\frac{\log\Delta}{\log(\omega-1)}}
\right\}.
\end{align}
\end{corollary}

\noindent
As such the simpler parts of our framework lead to improved (and explicit) leading asymptotic constants in the work of Shearer~\cite{She95} and of Bansal, Gupta, and Guruganesh~\cite[Thm~1.2]{BGG18}. In the former case (which is the more useful bound when~$\omega$ is fixed), this therefore constitutes the best progress towards an earlier-mentioned conjecture of Ajtai, Erd\H{o}s, Koml\'{o}s and Szemer\'{e}di~\cite{AEKS81}. We remind the reader that those authors would have had little interest in an increased leading constant, but rather in the removal of the stubborn~$\log\log\Delta$ factor, so as to more cleanly generalise Theorem~\ref{thm:shearer}. (See `Note added' for an update on this conjecture.)


\subsection{The hard-core model with bounded clique number}\label{sub:clique,hcm}

Here is a mild generalisation of~\cite[Lem.~4.4]{BKNP22}, which is a
consequence of the Erd\H{o}s--Szekeres recurrence for Ramsey
numbers~\cite{ErSz35}.

\begin{lemma}\label{lem:clique,logZ}
For any graph~$F$ on~$y$ vertices that contains no clique of size~$\omega\ge2$, any positive real~$\lam$, and any positive integer~$\alpha$ such that $y\ge R(\omega,\alpha)$,
\begin{align}\label{eq:logZ,smallomega}
\log Z_F(\lam) & \ge \alpha\left(\log(y\lam)-(\omega-1)\log\left(e\cdot\frac{\alpha-1}{\omega-1}+e\right)\right).
\intertext{If $\alpha\ge2$, then also}
\log Z_F(\lam) & \ge \alpha\left(\log(y\lam)-(\alpha-1)\log\left(e\cdot\frac{\omega-1}{\alpha-1}+e\right)\right).\label{eq:logZ,largeomega}
\end{align}
\end{lemma}

\begin{proof}
Let $R$ be the Ramsey number $R(\omega,\alpha)$. By the assumption on~$F$, every subset of~$V(F)$ of size~$R$ has an independent set of size $\alpha$. Note that every independent set of size $\alpha$ is contained in at most $\binom{y-\alpha}{R-\alpha}$ subsets of~$V(F)$ of size~$R$. Thus there are at least
\begin{align}
    \binom{y}{R}\left/\binom{y-\alpha}{R-\alpha}\right. \ge{\left(\frac{y}{R}\right)}^{\alpha} 
    = \exp(\alpha(\log y-\log R))
\end{align}
independent sets of size $\alpha$. By the result of Erd\H{o}s and
    Szekeres~\cite{ErSz35}, we have $R\le
    \binom{\alpha+\omega-2}{\omega-1}=\binom{\alpha+\omega-2}{\alpha-1}$. A
    standard estimate on the binomial coefficients gives
\begin{align}
    \log R &\le \log {\left(e\cdot\frac{\alpha+\omega-2}{\omega-1}\right)}^{\omega-1}
  =  (\omega-1) \log \left(e\cdot\frac{\alpha-1}{\omega-1}+e\right).
\intertext{If $\alpha\ge2$, the other representation of the binomial coefficient also gives}
    \log R &\le \log {\left(e\cdot\frac{\alpha+\omega-2}{\alpha-1}\right)}^{\alpha-1} 
  =  (\alpha-1) \log \left(e\cdot\frac{\omega-1}{\alpha-1}+e\right).
\end{align}
The announced inequalities then follow from the fact that
\begin{align}
Z_F(\lam) \ge |\{I\in \cI(F) \mid |I|=\alpha\}| \cdot \lam^\alpha
\end{align}
together with the monotonicity of~$\log$.
\end{proof}

\begin{lemma}\label{lem:clique,hcm}
For any graph~$F$ on~$y$ vertices that contains no clique of size~$\omega\ge3$ and any positive real~$\lam$, the following holds, where we write $z=\log Z_F(\lam)$.
\begin{enumerate}
\item\label{itm:smallclique,hcm}
Writing~$\widehat\lam=\min\{\lam,1\}$, we have
\[
    z \ge {(\omega-1)}^2\left(\frac{{(y\widehat\lam)}^{1/(\omega-1)}}{e^2}-1\right),
\]
and, for every fixed~$\lam>0$, as $y\to\infty$ we have
\begin{align*}
\frac{\lam Z_F'(\lam)}{Z_F(\lam)} \ge \frac{1-o(1)}{\omega-2}\frac{z}{\log z}.
\shortintertext{Moreover,}
Q_F(\lam) \ge \frac{1-o(1)}{\omega-2}\frac{z}{\log z}.
\end{align*}
\item\label{itm:largeclique,hcm}
    If $\lam>0$ is fixed, then as $y\to\infty$ and $\omega\to\infty$ we have
\begin{align*}
z 
    & \ge \left(\frac{1}{4}-o(1)\right)\frac{{(\log (y\lam))}^2}{\log(\omega-1)},\quad \text{and} \\
\frac{\lam Z_F'(\lam)}{Z_F(\lam)} 
& \ge \left(\frac12-o(1)\right)\sqrt{\frac{z}{\log(\omega-1)}}.
\end{align*}
Moreover,
\[
Q_F(\lam)
\ge \left(\frac12-o(1)\right)\sqrt{\frac{z}{\log(\omega-1)}}.
\]
\end{enumerate}
\end{lemma}

\begin{proof}
For~\ref{itm:smallclique,hcm}, write $\widehat\lam=\min\{\lam,1\}$ and set
\[
    \alpha -1 = \left\lfloor \frac{\omega-1}{e^2}{(y\widehat\lam)}^{1/(\omega-1)} - (\omega-1) \right\rfloor,
\]
which can be assumed to be positive as otherwise the desired lower bound on~$z$ is trivially true.
Writing~$R=R(\omega,\alpha)$, the Erd\H{o}s--Szekeres bound gives that
\[
R\le\binom{\alpha+\omega-2}{\omega-1}
\le\left(e\frac{\alpha+\omega-2}{\omega-1}\right)^{\omega-1}
\le\frac{y\widehat\lam}{e^{\omega-1}}\le y.
\]
We may therefore apply~\eqref{eq:logZ,smallomega} at fugacity~$\widehat\lam$. Since
\[
(\omega-1)\log\left(e\cdot\frac{\alpha-1}{\omega-1}+e\right) \le \log(y\widehat\lam) - (\omega-1),
\]
we deduce that
\[
    z\ge\log Z_F(\widehat\lam)
    \ge (\omega-1)\alpha
    \ge {(\omega-1)}^2\left(\frac{{(y\widehat\lam)}^{1/(\omega-1)}}{e^2}-1\right),
\]
as desired.
Rearranging this gives
\[
    y\lam \le \frac{\lam}{\widehat\lam}
    e^{2(\omega-1)}{\left(\frac{z}{{(\omega-1)}^2}+1\right)}^{\omega-1}
    \le \frac{\lam}{\widehat\lam}e^{2(\omega-1)}{(z+1)}^{\omega-1}.
\]

For the lower bound on $\lam Z_F'(\lam)/Z_F(\lam)$ we apply Lemma~\ref{lem:generalizedShearer}, which means we want to bound the term~$K(y\lam / z)$. Since~$K$ is increasing we may use the upper bound on~$y$ we just established, and we also use the facts that as~$y\lam\to\infty$ we have $z\to\infty$, and that $K(x)\sim\log(x)$ as~$x\to\infty$. 
Then 
\begin{align}
    K(y\lam / z) \le K\left(\frac{\lam}{\widehat\lam}
        e^{2(\omega-1)}\frac{{(z+1)}^{\omega-1}}{z}\right)
= (1+o(1))(\omega-2)\log z ,
\end{align}
so by Lemma~\ref{lem:generalizedShearer} we have, as $y\lam\to\infty$, 
\[
\frac{ \lam Z_F'(\lam)}{Z_F(\lam)} \ge \frac{1-o(1)}{\omega-2}\frac{z}{\log z},
\]
as required.  Lemma~\ref{lem:uniformAlon}, applied with
$p=1/(1+\lam)$ to the same upper bound on $\fn(y\lam/z)$, gives the
asserted bound on $Q_F(\lam)$.

    For~\ref{itm:largeclique,hcm}, we may replace~$\omega$ by~$\min\{\omega,y+1\}$: this parameter still tends
    to infinity, and the
    resulting bounds only strengthen those claimed.  Thus assume~$\omega\le y+1$. 
    Write~$\widehat\lam=\min\{\lam,1\}$.  The assumptions imply
    \[
        (\omega-1)\widehat\lam>1
        \quad\text{and}\quad
        \log(y\widehat\lam)=(1-o(1))\log(y\lam).
    \]
Then set
\[
    \alpha -1 = \left\lfloor  \frac{\log(y\widehat\lam)}{2\log((\omega-1)\widehat\lam)}\right\rfloor,
\]
which is nonnegative.  If~$\alpha=1$, then~$\log(y\lam)\le(2+o(1))\log(\omega-1)$, so the desired
partition-function bound follows directly from~$Z_F(\lam)\ge1+y\lam$.
Otherwise~$\alpha\ge2$, and from this choice we have
\begin{align}
(\alpha-1)\log&\left(e\cdot\frac{\omega-1}{\alpha-1}+e\right) 
\\&\le \frac{\log(y\widehat\lam)}{2\log((\omega-1)\widehat\lam)}\log\left(2e(\omega-1)\frac{\log((\omega-1)\widehat\lam)}{\log(y\widehat\lam)} + e\right)
\\&\le \frac{\log(y\widehat\lam)}{2\log((\omega-1)\widehat\lam)}\log\left(2e(\omega-1) + e\right).
\end{align}
    As~$\omega\to\infty$, and using the fact that the fixed~$\lam$ satisfies~$\lam = {(\omega-1)}^{o(1)}$, we
    have
\[
(\alpha-1)\log\left(e\cdot\frac{\omega-1}{\alpha-1}+e\right) \le \left(\frac{1}{2}+o(1)\right)\log(y\widehat\lam).
\]
The same estimate and the Erd\H{o}s--Szekeres bound give
\[
R(\omega,\alpha)\le {(y\widehat\lam)}^{1/2+o(1)}\le y.
\]
Thus we may apply~\eqref{eq:logZ,largeomega} at fugacity~$\widehat\lam$.  Using also the monotonicity of the
partition function, we have
\[
    z\ge\log Z_F(\widehat\lam)
    \ge \alpha \left(\frac{1}{2}-o(1)\right) \log (y\widehat\lam)
    \ge \left(\frac{1}{4}-o(1)\right)\frac{{(\log (y\lam))}^2}{\log(\omega-1)}, 
\]
Thus the asserted partition-function bound holds in either case, and
rearranging it gives
\[
\log(y\lam) \le (2+o(1))\sqrt{z\log(\omega-1)}.
\]

Like before we bound~$K(y\lam/z)$ from above and apply Lemma~\ref{lem:generalizedShearer} to obtain the lower
bound on~$\lam Z_F'(\lam)/Z_F(\lam)$. 
As~$y\lam\to\infty$ and hence~$z\to\infty$, we deduce from the properties of~$K$ and the above bound on~$y\lam$ that
\[
K(y\lam /z) \le (2+o(1))\sqrt{z\log(\omega-1)},
\]
which gives the expectation bound.  Lemma~\ref{lem:uniformAlon},
applied with~$p=1/(1+\lam)$ and this same estimate, gives the asserted
bound on~$Q_F(\lam)$.
\end{proof}

\subsection{Local occupancy with bounded clique number}\label{sub:clique,localocc}

The next result establishes the local occupancy required to prove Theorem~\ref{thm:clique}, which involves details for the optimisation tasks necessary to apply Lemma~\ref{lem:clique,hcm} and our framework. 
It is the main optimisation needed for the leading constants in Corollary~\ref{cor:clique,alpha}.
It is worth remarking here that Lemma~\ref{lem:clique,hcm} already yields via~\cite[Thm~1.13]{BKNP22} the conclusion of Theorem~\ref{thm:clique,chi} but with worse leading asymptotic constants.
We would not hold it against a reader uninterested in particular constants who might be tempted to skip over much of the remainder of this section. 

\begin{lemma}\label{lem:clique,localocc}
For any~$\lam,\xi>0$, there exists~$d_0>0$ such that the following holds.
For any graph~$G$ in which the maximum size of a clique containing a vertex~$u\in V(G)$ is at most~$\omega_u \ge 3$, 
    and any collection ${(d_u)}_{u\in V(G)}$ of positive reals satisfying
    $d_u\ge\max\{d_0,\omega_u-1\}$ for all~$u$,
there is a choice of parameters that satisfies
\[
    \beta_u + \gam_ud_u \le {(1+\xi)}^2\min\left\{(\omega_u -2) \frac{d_u\log\log d_u}{\log d_u},\, 2d_u\sqrt{\frac{\log(\omega_u-1)}{\log d_u}} \right\}
\]
for all $u\in V(G)$
    and strong local ${(\beta_u,\gam_u)}_u$-occupancy in the hard-core model on~$G$ at fugacity~$\lam$.
\end{lemma}
  
  \begin{proof}
      Let~$\zeta\in\interval[oo]{0}{1}$ be a constant to be specified later. 
  Fix an arbitrary vertex~$u\in V(G)$, and let~$F$ be any subgraph of~$G[N(u)]$ on~$y$ vertices. 
  As before, we write~$z=\log Z_F(\lam)$ for brevity.

  For strong local occupancy we must give a choice of~$\beta_u$ and~$\gam_u$ that is independent of~$F$ and such that
  \begin{equation}\label{eq:clique:localocc}
  \beta_u \frac{\lam}{1+\lam}e^{-z} + \gam_u \frac{\lam Z_F'(\lam)}{Z_F(\lam)} \ge 1.
  \end{equation}
  Writing 
  \begin{align}
  g_1(z) &= \beta_u \frac{\lam}{1+\lam} e^{-z} + \gam_u\; \frac{1-\zeta}{\omega_u-2}\;\frac{z}{\log z},\\
  g_2(z) &= \beta_u \frac{\lam}{1+\lam} e^{-z} + \gam_u \frac{1-\zeta}{2\sqrt{\log(\omega_u-1)}}\sqrt z,
  \end{align}
  Lemma~\ref{lem:clique,hcm} allows us to choose constants~$y_0=y_0(\zeta,\lam)$
  and~$\omega_0=\omega_0(\zeta,\lam)$ such that,
  for~$y>y_0$, it suffices for the first parameter above below to have~$g_1(z)\ge1$, and,
  provided~$\omega_u\ge\omega_0$, it suffices for the
  second parameter choice to have~$g_2(z)\ge1$.  Thus the corresponding
  condition is, respectively,
  \begin{equation}\label{eq:clique:betasmall}
      g_1(z)\ge1
      \quad\text{or, if $\omega_u\ge\omega_0$,}\quad
      g_2(z)\ge1.
  \end{equation}
  In either case we also require
  \begin{align}\label{eq:clique:betalarge}
      \beta_u \ge \frac{1+\lam}{\lam}{(1+\lam)}^{y_0}.
  \end{align}
      Indeed, on the one hand,~\eqref{eq:clique:betalarge} implies~\eqref{eq:clique:localocc} in the
      case~$y\le y_0$, since~$e^z\le {(1+\lam)}^y\le {(1+\lam)}^{y_0}$. On the other
      hand,~\eqref{eq:clique:betasmall} and Lemma~\ref{lem:clique,hcm} together
      yield~\eqref{eq:clique:localocc} if~$y> y_0$, for~$y_0$ large enough in terms of~$\lam$ and~$\zeta$. 

      Increase~$y_0$ if necessary so that~$\log(1+y_0\lam)>e$.
      If~$y_0<y$, then~$z\ge\log(1+y_0\lam)$, so to
      establish~\eqref{eq:clique:betasmall} it suffices to investigate~$g_1$ and~$g_2$
      on~$I=\interval[co]{\log(1+y_0\lam)}{\infty}$.
  
  Let~$z^*\ge e$ be given by the equation
  \begin{align}\label{eq:z1star}
    d_u \frac{\lam}{1+\lam} e^{-z^*} = \frac{1-\zeta}{\omega_u-2} \frac{z^*}{\log z^*}.
  \end{align}
  There is a unique solution in~$\interval[co]{e}{\infty}$ because both sides are monotone and continuous functions of~$z^*$,
  the left-hand side is greater than the right-hand side at~$z^*=e$ and, as~$z^*$ tends to infinity,
  the left-hand side tends to~$0$ while the right-hand side tends to infinity.
  Writing 
  \[
    \tau_u = d_u\frac{\lam}{1+\lam}\frac{\omega_u-2}{1-\zeta},
  \]
  an asymptotic analysis of~\eqref{eq:z1star} gives
  \[
  z^* = \log \tau_u - \log\log \tau_u + \log\log\log \tau_u + o(1),
  \]
  as $d_u$, and hence $\tau_u$, tends to infinity. 
  
  Let $\beta_u$ and~$\gam_u$ be given by finding the values of~$\beta_u$ and~$\gam_u$ that solve the equations $g_1'(z^*)=0$ and $g_1(z^*)=1$, and then multiplying each by the factor $1/(1-\zeta)> 1$, so as to ensure~\eqref{eq:clique:betasmall}.
  Some elementary calculus checks that
  \begin{align}
      \gam_u  &= \frac{\omega_u -2}{{(1-\zeta)}^2} \frac{{(\log z^*)}^2}{(1+z^*)\log z^*-1},\\
    \beta_u &= \frac{1}{1-\zeta}\frac{1+\lam}{\lam}e^{z^*} \frac{\log z^* -1}{(1+z^*)\log z^*-1},
    \intertext{and hence}
      \beta_u+\gamma_ud_u&=\frac{1}{1-\zeta}\frac{1+\lambda}{\lambda}e^{z^*}.
  \end{align}
  As~$d_u$ (and hence~$z^*$) tends to infinity,
  \[
      \beta_u + \gam_u d_u \sim \frac{d_u(\omega_u -2)}{{(1-\zeta)}^2} \frac{\log\log (d_u(\omega_u-2))}{\log (d_u(\omega_u-2))}.
  \]
  Since~$\beta_u\to\infty$ as~$d_u\to\infty$, for~$d_0$ large enough in terms of~$\lam$ and~$\zeta$ (and hence~$y_0$) we will have~\eqref{eq:clique:betalarge}.
  Furthermore, choosing~$\zeta$ small enough in terms of~$\xi$ we have, for large enough~$d_0$,
  \[
    \beta_u + \gam_ud_u \le (1+\xi)d_u(\omega_u -2) \frac{\log\log (d_u(\omega_u-2))}{\log (d_u(\omega_u-2))}.
  \]
	  We must now justify that this parameter choice controls~$g_1$ on
	  all of~$I$.  The function
	  \[
	      z\mapsto e^z\frac{\mathrm d}{\mathrm dz}\frac{z}{\log z} = \frac{e^z}{\log z}\left(1-\frac{1}{\log z}\right)
	  \]
      is strictly increasing for~$z>e$ (as the product of two strictly increasing and positive functions),
      so~$g_1'$ has at most one zero in~$\interval[oo]{e}{\infty}$.  Thus~$z^*$ is the unique global minimum
      of~$g_1'$ when~$z^*\in I$, and otherwise the minimum on~$I$ is attained at its left endpoint. By
      construction~$g_1(z^*)=1/(1-\zeta)>1$, while, for~$d_0$ large enough in terms of~$\lam$ and~$\xi$ (and
      hence~$y_0$),
	  \[
	    g_1(\log(1+y_0\lam))
	    > \beta_u\frac{\lam}{1+\lam}\frac{1}{1+y_0\lam}>1.
	  \]
	  Hence~$g_1(z)>1$ throughout~$I$.
  
  For the second choice assume that~$\omega_u\ge\omega_0$.  For~$g_2$ we
  give a similar argument, and we redefine~$z^*$, $\beta_u$ and~$\gam_u$
  rather than introduce additional notation.
  Let~$z^*\ge0$ be given by the equation
  \[
    d_u \frac{\lam}{1+\lam}e^{-z^*} = \frac{1-\zeta}{2\sqrt{\log(\omega_u-1)}}\sqrt{z^*}.
  \]
  There is a unique solution in~$\interval[co]{0}{\infty}$ because both sides are monotone and continuous
  functions of~$z^*$, the left-hand side is greater than the right-hand side at~$z^*=0$ and, as~$z^*$  tends
  to infinity, the left-hand side tends to~$0$ while the right-hand side tends to infinity.
  The solution satisfies
  \[
      z^* = \frac{1}{2}W\left(\frac{8d_u^2\lam^2\log(\omega_u-1)}{{(1-\zeta)}^2{(1+\lam)}^2}\right)
    \sim \log d_u
	  \]
	  as~$d_u\to\infty$, where we used the fact that~$\omega_u-1\le d_u$.
	  Increasing~$d_0$ if necessary, we assume that~$z^*>e$.
	  
      Let~$\beta_u$ and~$\gam_u$ be given by finding the values of~$\beta_u$ and~$\gam_u$ that solve the
      equations~$g_2'(z^*)=0$ and~$g_2(z^*)=1$, and then multiplying each by the factor~$1/(1-\zeta)> 1$, so
      as to ensure~\eqref{eq:clique:betasmall}. Then
  \begin{align}
    \beta_u\frac{\lam}{1+\lam}e^{-z^*} &= \gam_u\frac{1-\zeta}{4\sqrt{\log(\omega_u-1)}\sqrt{z^*}},\\
    \gam_u &= \frac{4\sqrt{\log(\omega_u-1)}}{(1-\zeta)^2}\frac{\sqrt{z^*}}{2z^*+1},
    \intertext{and hence, as~$d_u\to\infty$ we have}
\beta_u + \gam_u d_u &\sim \frac{2d_u}{{(1-\zeta)}^2} \sqrt{\frac{\log(\omega_u-1)}{\log d_u}}.
  \end{align}
  Choosing~$\zeta$ small enough in terms of~$\xi$, and for a large enough constant~$d_0$ in terms of~$\lam$ and~$\xi$, we have~\eqref{eq:clique:betalarge} and
  \[
    \beta_u+\gam_ud_u \le (1+\xi)2d_u\sqrt{\frac{\log(\omega_u-1)}{\log d_u}}.
  \]
	  Similarly,~$z\mapsto e^z/(2\sqrt z)$ is strictly increasing for~$z>1/2$,
    so~$g_2'$ has at most one zero in~$\interval[oo]{1/2}{\infty}$.  Thus the minimum of~$g_2'$ on~$I$ is
	  attained at~$z^*$ if~$z^*\in I$, and otherwise at the left
	  endpoint of~$I$.  By construction~$g_2(z^*)=1/(1-\zeta)>1$, and
	  \[
	    g_2(\log(1+y_0\lam))
	    > \beta_u\frac{\lam}{1+\lam}\frac{1}{1+y_0\lam}>1
	  \]
	  for~$d_0$ large enough.  Hence~$g_2(z)>1$ throughout~$I$.
  
  The final step is to obtain the simplified asymptotic forms given in the lemma. 
  The first choice of~$(\beta_u,\gam_u)$ satisfies strong local
  occupancy for every~$\omega_u\ge3$, while the second does so whenever
  $\omega_u\ge\omega_0$.  After increasing~$d_0$, the first expression
  below is smaller whenever $3\le\omega_u<\omega_0$.  We may therefore
  choose the better admissible pair, obtaining
  \begin{align}
  \beta_u\negthinspace +\negthinspace \gam_ud_u\negthinspace \le\negthinspace (1+\xi)\min\left\{\negthickspace
    (\omega_u -2) \frac{d_u \log\log (d_u(\omega_u-2))}{\log (d_u(\omega_u-2))},
    2d_u\sqrt{\frac{\log(\omega_u-1)}{\log d_u}}
  \right\}.
  \end{align}
  Since $x\mapsto\log\log x/\log x$ is decreasing for large~$x$,
  increasing~$d_0$ if necessary now gives
  \[
      \beta_u + \gam_ud_u \le {(1+\xi)}^2\min\Bigg\{(\omega_u -2) \frac{d_u \log\log d_u}{\log d_u}, 2d_u\sqrt{\frac{\log(\omega_u-1)}{\log d_u}} \Bigg\}. \qedhere
  \]
  \end{proof}

\begin{lemma}\label{lem:clique,chromaticlocalocc}
  For every~$\xi>0$ there is~$\lam_0\in\interval[oc]{0}{1}$ such that, for
  every~$\lam\in\interval[oc]{0}{\lam_0}$ and~$M>0$, there exists $d_0>0$ with the following property. Under the
  hypotheses of Lemma~\ref{lem:clique,localocc}, with~$d_u\ge d_0$ for every~$u$, there are
  parameters~$\beta_u,\gam_u>0$ such that the hard-core model has chromatic local
  ${(\beta_u,\gam_u)}_u$-occupancy,
\[
    \beta_u+\gam_ud_u
    \le {(1+\xi)}^2\min\left\{
        (\omega_u-2)\frac{d_u\log\log d_u}{\log d_u},
        2d_u\sqrt{\frac{\log(\omega_u-1)}{\log d_u}}
    \right\},
\]
and~$d_u\gam_u^2\ge M$ for every~$u$.
\end{lemma}

\begin{proof}
  Choose~$\zeta>0$ sufficiently small in terms of~$\xi$, and then choose~$\lam_0>0$ sufficiently small.  For
  $\lam\interval[oc]{0}{\lam_0}$, take~$p=1/(1+\lam)$.  For every~$u$, write
\[
    t_u=\log d_u-3\log\log d_u
    \qquad\text{and}\qquad
    T_u=e^{t_u}=\frac{d_u}{{(\log d_u)}^3}.
\]
We take~$d_0$ large enough below that~$t_u>e$ for every~$u$.

Let~$F\subseteq G[N(u)]$ have~$y$ vertices and write~$z=\log Z_F(\lam)$.  The estimates in
Lemma~\ref{lem:clique,hcm}, together with
Lemma~\ref{lem:uniformAlon}, imply uniformly as~$z\to\infty$ that
\begin{equation}\label{eq:clique,threshold,small}
    Q_F(\lam)\ge
    \frac{1-o(1)}{\omega_u-2}\frac{z}{\log z}.
\end{equation}
When $\omega_u\to\infty$, they also imply
\begin{equation}\label{eq:clique,threshold,large}
    Q_F(\lam)\ge
    \left(\frac12-o(1)\right)
      \sqrt{\frac{z}{\log(\omega_u-1)}}.
\end{equation}
For completeness, the uniformity in the first estimate follows directly
from the exact bound in Lemma~\ref{lem:clique,hcm}: uniformly for~$\omega_u\ge3$,
\[
\fn(y\lam/z)
\le\fn\left(e^{2(\omega_u-1)}
             \frac{{(z+1)}^{\omega_u-1}}{z}\right)
\le(1+o(1))(\omega_u-2)\log z.
\]
For the second estimate, write
$r=\min\{\omega_u-1,y\}$ and $L=\log(y\widehat\lam)$, where
$\widehat\lam=\min\{\lam,1\}$.  In the $\alpha\ge2$ case of the proof of
Lemma~\ref{lem:clique,hcm}\ref{itm:largeclique,hcm}, the relevant error
is controlled explicitly by
\[
    \frac{\log(2er+e)}{\log(r\widehat\lam)}
    =1+O_\lam\left(\frac{1}{\log r}\right).
\]
Together with
$\log(y\widehat\lam)=\log(y\lam)+O_\lam(1)$, the same estimate gives
\[
    \log R(r+1,\alpha)
    \le\left(\frac12+o(1)\right)L,
\]
and hence $R(r+1,\alpha)\le(y\widehat\lam)^{1/2+o(1)}\le y$ uniformly
for large~$y,r$.  We may therefore apply
\eqref{eq:logZ,largeomega}, which gives the claimed partition-function
estimate with an error tending to zero uniformly whenever
$y,r\to\infty$; the $\alpha=1$ case follows from
$Z_F(\lam)\ge1+y\lam$.  Since $z\le y\lam$, while
$z,\omega_u\to\infty$ implies $y,r\to\infty$, the ensuing estimate for
$\fn(y\lam/z)$ and Lemma~\ref{lem:uniformAlon} give
\eqref{eq:clique,threshold,large} uniformly.  Taking
$r=\min\{\omega_u-1,y\}$ only strengthens the displayed bound when
$\omega_u>y$.

Define
\[
    B_{1,u}=(\omega_u-2)
       \frac{d_u\log\log d_u}{\log d_u},
    \qquad
    B_{2,u}=2d_u
       \sqrt{\frac{\log(\omega_u-1)}{\log d_u}}.
\]
If $B_{1,u}\le B_{2,u}$, set
\[
    a_u=\frac{1-\zeta}{\omega_u-2}\frac{t_u}{\log t_u}.
\]
For large enough~$d_0$, \eqref{eq:clique,threshold,small} and the
monotonicity of $z/\log z$ show that
\[
    Z_F(\lam)>T_u\quad\Longrightarrow\quad Q_F(\lam)\ge a_u.
\]
If $B_{2,u}<B_{1,u}$, then
\[
    \frac{\omega_u-2}{\sqrt{\log(\omega_u-1)}}
    >\frac{2\sqrt{\log d_u}}{\log\log d_u},
\]
so $\omega_u\to\infty$ uniformly as $d_u\to\infty$.  We may therefore
use~\eqref{eq:clique,threshold,large} and set
\[
    a_u=\frac{1-\zeta}{2}
       \sqrt{\frac{t_u}{\log(\omega_u-1)}}.
\]
Again, after increasing~$d_0$,
$Z_F(\lam)>T_u$ implies $Q_F(\lam)\ge a_u$.

Lemma~\ref{lem:chromatic,threshold} now gives chromatic local occupancy
with
\[
    \beta_u=\frac{1+\lam}{\lam}T_u
    \qquad\text{and}\qquad
    \gam_u=\frac{1+\lam}{a_u}.
\]
According to the two cases above,
\begin{align*}
    \gam_ud_u
    &=\frac{1+\lam}{1-\zeta}(\omega_u-2)d_u
      \frac{\log t_u}{t_u},
    \qquad\text{or}\\
    \gam_ud_u
    &=\frac{2(1+\lam)}{1-\zeta}d_u
      \sqrt{\frac{\log(\omega_u-1)}{t_u}}.
\end{align*}
Since $t_u\sim\log d_u$ and
$\log t_u\sim\log\log d_u$, choosing $\zeta$ and $\lam_0$ sufficiently
small and then $d_0$ sufficiently large gives
\[
    \beta_u+\gam_ud_u
    \le {(1+\xi)}^2\min\{B_{1,u},B_{2,u}\}.
\]
Indeed, since~$\omega_u\ge3$, both~$B_{1,u}$ and~$B_{2,u}$ are at least~$d_u\log\log d_u/\log d_u$ for all
sufficiently large~$d_u$, whereas~$\beta_u=O_\lam(d_u/{(\log d_u)}^3)$.  Thus~$\beta_u$ is uniformly
lower order than the displayed minimum.  Finally, in the two cases
respectively,
\[
    d_u\gam_u^2
    \gg d_u\left(\frac{\log\log d_u}{\log d_u}\right)^2
    \qquad\text{or}\qquad
    d_u\gam_u^2\gg\frac{d_u}{\log d_u}.
\]
Both tend to infinity uniformly, so increasing~$d_0$ once more gives~$d_u\gam_u^2\ge M$.
\end{proof}

\subsection{Proof of Theorem~\ref{thm:clique}}\label{sub:clique,proof}

\begin{proof}
For parts~\ref{itm:clique,occfrac} and~\ref{itm:clique,chif}, we may
replace~$\omega_u$ by~$\widetilde\omega_u=\max\{3,\min\{\omega_u,d_u+1\}\}$.
This remains a valid upper bound on the size of a clique containing~$u$
and only strengthens the asserted bounds, so we may assume that~$\omega_u-1\le\max\{2,d_u\}$.
For~\ref{itm:clique,occfrac}, note that for fixed~$\lam$ we may apply Lemma~\ref{lem:clique,localocc} with~$d_u = \Delta$ for each~$u$, and~$\xi$ an arbitrary positive constant. 
Since~$\xi$ is arbitrary we obtain the desired asymptotic form as $\Delta\to\infty$. 
Applying Theorem~\ref{thm:main,occfrac} and recalling that occupancy fraction is monotone increasing in~$\lam$ yields~\ref{itm:clique,occfrac}.

Similarly,~\ref{itm:clique,chif} arises from Theorem~\ref{thm:main,chif} and an application of
Lemma~\ref{lem:clique,localocc} with an arbitrary fixed~$\lam>0$, a choice of~$\xi>0$ such that~${(1+\xi)}^2 <
1+\eps$, and the lemma's parameter~$d_u$ set to~$d_u^+ = \max\{d_u, \delta_0\}$ for each~$u$,
where~$\delta_0\ge2$ is at least the~$d_0$ provided by the lemma. 

The remainder of the proof is devoted to proving~\ref{itm:clique,chic}.
We first seek to apply Theorem~\ref{thm:main,chic} to obtain the asserted correspondence colouring.
Choose~$\xi>0$ such that~${(1+\xi)}^4<1+\eps$, and then choose~$\lam\in\interval[oc]{0}{1}$ sufficiently small to apply
Lemma~\ref{lem:clique,chromaticlocalocc}.
Set~$p=1/(1+\lam)$ and define
\begin{align*}
    L_\Delta&=\log(8\Delta^4),&
    A_\Delta&=\frac{1}{\lam}{(e^2L_\Delta)}^{\omega-1},\\
    B_\Delta&=\exp\left(
       \sqrt{4(1+\eps)\log(\omega-1)L_\Delta}\right),
\end{align*}
and choose
\[
    \ell=8\min\{A_\Delta,B_\Delta\}.
\]
Since~$\omega\ge 3$, this choice satisfies~$\ell=\omega(\log\Delta)$ as~$\Delta\to\infty$, and thus that~$\ell > 7\log\Delta$ for large enough~$\Delta_0$. 

We verify condition~\eqref{eq:Zlarge} directly.  Let~$F\subseteq G[N(u)]$ have~$y\ge\ell/8$ vertices.  The graph~$F$
contains no clique of size~$\omega$.  If~$A_\Delta\le B_\Delta$, then~$y\ge A_\Delta$, so Lemma~\ref{lem:clique,hcm}\ref{itm:smallclique,hcm}
gives
\[
    \log Z_F(\lam)
    \ge {(\omega-1)}^2
       \left(\frac{{(y\lam)}^{1/(\omega-1)}}{e^2}-1\right)
    \ge {(\omega-1)}^2(L_\Delta-1)
    \ge L_\Delta
\]
for all sufficiently large~$\Delta$.

Suppose instead that~$B_\Delta<A_\Delta$.  This forces~$\omega\to\infty$, since for bounded~$\omega$ the quantity~$A_\Delta$
is polynomial in~$L_\Delta$, whereas~$B_\Delta$ is exponential in~$\sqrt{L_\Delta}$.  If a graph~$F$ as above
exists, then~$B_\Delta\le y\le\Delta$.  The inequality~$B_\Delta\le\Delta$ implies
\[
    \log(\omega-1)
    \le\frac{{(\log\Delta)}^2}{4(1+\eps)L_\Delta},
\]
and hence~$B_\Delta>\omega-1$ for all sufficiently large~$\Delta$.
In particular,~$\omega\le y$.  Since~$\lam$ is fixed,
Lemma~\ref{lem:clique,hcm}\ref{itm:largeclique,hcm} applies and gives
\begin{align*}
    \log Z_F(\lam)
    &\ge\left(\frac14-o(1)\right)
       \frac{{(\log(y\lam))}^2}{\log(\omega-1)}\\
    &\ge(1-o(1))(1+\eps)L_\Delta
    \ge L_\Delta.
\end{align*}
Here the second inequality uses
\[
    \log(y\lam)\ge\log(B_\Delta\lam)
    =\sqrt{4(1+\eps)\log(\omega-1)L_\Delta}+O_\lam(1).
\]
Thus~$Z_F(\lam)\ge8\Delta^4$ in both cases, as required.
Set
\[
    \eta=\sqrt{\frac{7\log\Delta}{\ell}},
    \qquad
    s=\frac{\lam}{1+\lam}\frac{\ell}{1-\eta},
    \qquad
  D_u=\frac{\max\{d_u,\delta\}}{s}=\frac{\overline d_u}{s}.
\]
  Thus~$\eta=o(1)$, and the definition of~$\eta$ gives condition~\eqref{eq:elllarge} (with~$\iota = \eta$).
  First repeatedly set aside any vertex~$u$ for which~$|L(u)|$ is larger than its current degree; such
  vertices can be coloured greedily in reverse order.  In the graph that remains, the list-size hypothesis
  and~$s=\Theta_\lam(k)$ imply, after increasing~$\delta_0$ and~$\Delta_0$ if necessary, that~$D_u\ge\omega_u-1$. 
  Indeed, if either of the two displayed list bounds in~\ref{itm:clique,chic} is at most~$d_u$, then
  respectively
\[
  \omega_u-2
  \le \frac{\log(\max\{d_u,\delta\}/k)}
  {\log\log(\max\{d_u,\delta\}/k)}
  \quad\text{or}\quad
  \log(\omega_u-1)
  \le\frac14\log(\max\{d_u,\delta\}/k),
\]
and either inequality gives the claim.

Choose~$\delta_0$ sufficiently large~that~$D_u$ is at least the constant~$d_0$ in
Lemma~\ref{lem:clique,chromaticlocalocc}, where in that lemma we take
\[
    M=\frac{2(1+\lam)^2}{\lam}.
\]
This yields parameters~$\beta_u,\gam_u$ giving chromatic local
occupancy.  Define
\[
    b_u=\frac{s}{1-\eta}\left(\beta_u+\gam_uD_u\right).
\]
Since~$s=\Theta(k)$ and~$\eta=o(1)$, Lemma~\ref{lem:clique,chromaticlocalocc}
gives, uniformly in~$u$,
\begin{align*}
    b_u\le (1+o(1)){(1+\xi)}^2\max\{d_u,\delta\}\cdot{}
    \min\Bigg\{&
      (\omega_u-2)\frac{\log\log D_u}{\log D_u},\\
      &2\sqrt{\frac{\log(\omega_u-1)}{\log D_u}}
    \Bigg\}.
\end{align*}
Moreover,
\[
    \log D_u
    =\log\left(\frac{\max\{d_u,\delta\}}{k}\right)+O(1).
\]
Consequently, for large enough~$\Delta_0$, the list-size hypothesis in
\ref{itm:clique,chic} gives~$|L(u)|\ge b_u$.

It remains to verify condition~\eqref{eq:ysmall} (with~$\iota=\eta$).  Let
\[
    R=\frac{2}{\eta^2p}\log(16\Delta^3).
\]
Repeatedly delete any vertex~$u$ whose current degree is less than~$R/\gam_u$. 
Since~$\eta^2\ell=7\log\Delta$, for all sufficiently
large~$\Delta$ we have~$R/s\le M$.  The additional conclusion~$D_u\gam_u^2\ge M$ of
Lemma~\ref{lem:clique,chromaticlocalocc} therefore gives
\[
    \frac{R}{\gam_u}\le sD_u\gam_u\le b_u\le |L(u)|.
\]
Thus the deleted vertices can be coloured greedily in reverse order.
Every vertex in the graph that remains satisfies~\eqref{eq:ysmall};
chromatic local occupancy and~\eqref{eq:Zlarge} are inherited by
subgraphs.  Finally, for its current degree $d'_u\le d_u$,
\[
 \frac{1}{1-\eta}\left(
       \frac{\lam\ell}{1+\lam}\beta_u+\gam_ud'_u\right)
 \le b_u\le |L(u)|.
\]
Theorem~\ref{thm:main,chic} now supplies a correspondence colouring of
the remaining graph, and hence of~$G$.

To complete the proof of~\ref{itm:clique,chic}, we next derive the non-local corollary (which implies
Theorem~\ref{thm:clique,chi}), where we replace~$d_u$ with~$\Delta$ and~$\omega_u$ with~$\omega$.  For this
assertion we may take~$\eps\in\interval[oo]{0}{0.01}$.
If~$\omega-1\ge\Delta^{1/24}$, then, provided~$\Delta$ is large enough, we have by~\eqref{eq:params}
\[
  \chi_c(G) \le \Delta +1
   \le (\Delta+1)\sqrt{\frac{24\log(\omega-1)}{\log\Delta}}
   < 5\Delta\sqrt{\frac{\log(\omega-1)}{\log\Delta}}.
\]
The first expression in the claimed minimum is also larger than~$\Delta+1$, so this proves the result in this
case.  We may therefore assume that~$\omega-1<\Delta^{1/24}$.
In particular,
\[
  \log k \le \sqrt{(1+2\eps)16/24}\log\Delta<\log\Delta
\]
for all sufficiently large~$\Delta$, so the following quantities are
well defined.  Let
\[
  b_1 = (\omega -2) \frac{\Delta \log\log(\Delta/k)}{\log(\Delta/k)}
  \quad\text{and}\quad
  b_2 = 2\Delta\sqrt{\frac{\log(\omega-1)}{\log (\Delta/k)}},
\]
and note that as $\Delta\to\infty$ we have 
\[
  \log k = \min\left\{O(\omega\log\log\Delta) ,\, \sqrt{4\log(\omega-1)\cdot(1+\eps)\log(8\Delta^4)}\right\}.
\]
If~$b_1 < b_2$, then
\[
  \frac{\omega-2}{\sqrt{\log(\omega-1)}} < \frac{2\sqrt{\log(\Delta/k)}}{\log\log(\Delta/k)} < \frac{2\sqrt{\log\Delta}}{\log\log\Delta},
\]
where the last inequality holds because~$\Delta\ge e^{e^2}\cdot k$ for large enough~$\Delta$.
This means that~$\omega = O(\log\Delta/\log\log\Delta)$ as~$\Delta\to\infty$
and hence~$\log k = o(\log\Delta)$.  It then follows that
\[
  b_1 \sim (\omega -2) \frac{\Delta \log\log\Delta}{\log\Delta},
\]
as required. 
If~$b_2\le b_1$, we use the definition of~$k$ to deduce that
\[
  b_2 \le \frac{2}{\sqrt{1-\sqrt{(1+2\eps)16/24}}}\Delta\sqrt{\frac{\log(\omega-1)}{\log \Delta}} < 5\Delta\sqrt{\frac{\log(\omega-1)}{\log\Delta}},
\]
as required. \qedhere
\end{proof}

The reader will notice at the end of the proof that an improvement in the constant below~$5$ is possible, but
not below~$4$ without a better analysis or some improvement in Theorem~\ref{thm:main,chic} in the
condition~\eqref{eq:Zlarge}.

\section{Concluding remarks}\label{sec:conc}

\subsection{Ramsey numbers and graph colouring}

Ever since the seminal work of Johansson~\cite{Joh96}, researchers have intuitively felt that finding asymptotic bounds on the (list) chromatic number of triangle-free graphs is closely tied to estimation of the off-diagonal Ramsey number~$R(3,k)$. 
In particular it is believed that the bottleneck in bounding the chromatic number from above is essentially in bounding the independence number from below, and this may extend to~$R(4,k)$ and other off-diagonal cases. 
Our work shows this in a concrete sense, and for more general classes of sparse graphs. 
Our results are most interesting in the sparsest settings, where it seems that the best independence number bounds come from a suitable local understanding of the hard-core model.

This general belief relating (list) chromatic numbers and independence numbers is essentially valid when we instead consider the binomial random graph, cf.~\cite{KaMc15survey}; this has affinity to the problem settings considered here.
To elaborate on this further, we present a general family of strengthened Ramsey numbers, followed by a simple observation.
For positive integers~$k$ and~$\ell$, define the \emph{chromatic Ramsey number~$R_\chi(k,\ell)$} as the least~$n$ such that for any $K_k$-free graph~$G$ on~$n$ vertices, the chromatic number of~$G$ satisfies $\chi(G) < n/\ell$. Note the asymmetry in the definition, for $k\ne \ell$ it may be the case that $R_\chi(k,\ell)\ne R_\chi(\ell,k)$. With~\eqref{eq:params} in mind, one can analogously define the \emph{fractional chromatic Ramsey number~$R_{\chi_f}(k,\ell)$} and the \emph{list chromatic Ramsey number~$R_{\chi_\ell}(k,\ell)$}. By~\eqref{eq:params}, we have $R(k,\ell) \le R_{\chi_f}(k,\ell) \le R_\chi(k,\ell) \le R_{\chi_\ell}(k,\ell)$ always. The following shows how all of three of these parameters are well defined and, at least in the symmetric case, obey a similar asymptotic upper bound as do the classical Ramsey numbers~\cite{ErSz35}.

\begin{proposition}\label{prop:erdosszekeres,kahn}
If~$R(k,k) \le (c+o(1))^k$ as~$k\to\infty$ for some constant~$\sqrt{2}\le c\le 4$, then~$R_{\chi_\ell}(k,k)
\le (c+o(1))^k$ as $k\to\infty$.
\end{proposition}

\begin{proof}
Set~$q=k+\lceil\sqrt{k}\rceil$, $r=R(q,q)$, and~$n=\lceil k^2r\rceil$.
The hypothesis gives
\[
    \log n\le 2\log k+q(\log c+o(1))=k\log c+o(k),
\]
so~$n\le(c+o(1))^k$.
If~$G$ is a $K_k$-free graph on~$n$ vertices, then every set of at least~$r$ vertices contains an independent set of size~$q$.

Let~$L$ be an $m$-list-assignment of~$G$, where
\[
    m=\left\lceil\frac{n}{q}\right\rceil+r<\frac{n}{k}
\]
for all sufficiently large~$k$, where the inequality follows from
$n(1/k-1/q)>r+1$.
As long as some colour~$x$ appears on the lists of at least~$r$ uncoloured vertices, colour an independent set of~$q$ such vertices with~$x$, and remove~$x$ from all remaining lists.
There are at most~$\lfloor n/q\rfloor$ such steps.
Afterwards every remaining list has size at least~$r$, while every colour appears on fewer than~$r$ remaining lists.
Indeed, the union of any nonempty set of fewer than~$r$ remaining lists has size at least~$r$, while an incidence count shows that the union of any set~$S$ of at least~$r$ remaining lists has size greater than~$|S|$.
Hall's theorem therefore gives distinct colours to all remaining vertices.
Thus $\chi_\ell(G)\le m<n/k$, as required.
\end{proof}

\noindent
An interesting but perhaps difficult question is the following: is it true that $R_{\chi_\ell}(k,k)-R(k,k) \to\infty$ as~$k\to\infty$? What about the harder problem with $\chi$ in place of $\chi_\ell$?
In another direction, the triangle-free analogue of Proposition~\ref{prop:erdosszekeres,kahn} remains an intriguing challenge.

\begin{conjecture}[\cite{CJKP20}]\label{conj:CJKP}
$R_{\chi_\ell}(3,k) \le (1+o(1))k^2/\log k$ as $k\to\infty$.
\end{conjecture}

\noindent
By~\eqref{eq:params}, this would imply Shearer's bound on the off-diagonal Ramsey numbers~\cite{She83}.
The asymptotic order has since been matched, and the fractional chromatic version verified, although the conjectured leading constant in the list version remains open. We have a strong feeling that hard-core methods may be of use here, but it needs another idea or two.
(See `Note added' for an update on these remarks.)

\smallskip
By contrast, it is worth mentioning that the arguments in Subsection~\ref{sub:ideas} are sharp, in the sense that the conclusion about occupancy fraction in Theorem~\ref{thm:main,occfrac} cannot be asymptotically improved in general. Specifically, the triangle-free occupancy fraction bound in~\cite{DJPR18} is asymptotically extremal due to the random regular graphs. The authors of~\cite{DJPR18} naturally raised the question of whether maximum independent set size can be guaranteed to be large in comparison to average independent set size, especially for triangle-free graphs, and conjectured the following.
\begin{conjecture}[\cite{DJPR18}]\label{conj:DJPR18}
For any triangle-free graph~$G$ of minimum degree~$\delta$,
\[
\alpha(G) \ge (2-o(1))Z'_G(1)/Z_G(1)
\]
as $\delta\to\infty$.
\end{conjecture}
\noindent
Any partial progress, i.e.~with an asymptotic leading constant strictly greater than~$1$, would be a major advance in quantitative Ramsey theory as it would constitute an improvement over Shearer's bound~\cite{She83}. (See `Note added' for an update on this conjecture.)

\subsection{Local colouring of triangle-free graphs}\label{sub:local}

In Section~\ref{sec:mad} we proved local colouring results for graphs with no~$C_k$ and graphs in which neighbourhoods contain few edges. 
In Section~\ref{sec:localhall} we did the same for graphs with bounded local Hall ratio. 
Each of these settings generalises the condition of containing no triangles.
Specialising the local maximum-average-degree analysis to triangle-free
graphs and combining it with Theorem~\ref{thm:main,chic} gives the
following stronger form of Theorem~\ref{thm:molloy,local}.

\begin{corollary}\label{cor:molloy,local}
For any $\eps>0$ there exist $\delta_0$ and $\Delta_0$ such that the following holds for all $\Delta\ge\Delta_0$. For any triangle-free graph~$G$ of maximum degree~$\Delta$, write $d_u^+=\max\{d_u,\delta_0\log\Delta\}$.
If $\sH=(L,H)$ is a cover of~$G$ such that
\[
|L(u)| \ge (1+\eps)\frac{d_u^+}{\log(d_u^+/\log \Delta)}
\]
for each~$u\in V(G)$, then~$G$ is $\sH$-colourable.
\end{corollary}

\begin{proof}
Let $r>1$ be a sufficiently large constant in terms of~$\eps$, and set
\[
    \lam=\frac1r,\qquad p=1-\frac1r,\qquad
    \iota=\frac1{\sqrt r},\qquad
    \ell=40r\log(16\Delta^3),\qquad
    \delta_0=e^r.
\]
Write $D_u=d_u^+$.  Since $G$ is triangle-free,
$\mad(G[N(u)])=0$ for every~$u$.  Apply
Lemma~\ref{lem:mad}\ref{itm:mad,chromatic} with $a_u=0$ and with its
auxiliary degree parameter set to~$D_u$.
Indeed, uniformly in~$u$ as~$r\to\infty$,
\[
    \frac{D_u}{\ell}\ge\frac{e^r}{160r}\to\infty,
    \qquad
    \log\frac1{1-p}=\log r=o\left(\log\frac{D_u}{\ell}\right).
\]
The final assertion of that lemma therefore gives $p$-chromatic local
$(\beta,\gam)$-occupancy with parameters such that, for all sufficiently
large~$r$ and then~$\Delta$,
\begin{equation}\label{eq:molloy,local,parameters}
    \frac1{1-\iota}\left(
        \frac{\lam\ell}{{1+\lam}}\beta_u+\gam_uD_u
    \right)
    \le(1+\eps)\frac{D_u}{\log(D_u/\log\Delta)}.
\end{equation}
	Here we used
    \begin{equation}\label{eq-logduol}
	    \log(D_u/\ell)
	    =\left(1-O\left(\frac{\log r}{r}\right)\right)
	      \log(D_u/\log\Delta)
        \end{equation}
	uniformly in~$u$ for all sufficiently large~$\Delta$.

Set
\[
    R=\frac{2}{\iota^2p}\log(16\Delta^3).
\]
The explicit parameters in Lemma~\ref{lem:mad} also give~$D_u\gam_u^2\ge R$ after increasing~$r$ and~$\Delta$
if necessary. For completeness, uniformly in~$u$ the ratio of the left-hand side to
the right-hand side is bounded below by~$\Omega(e^r/r^3)$, noting that~$K_u\to 1$ and~$\omega_u /r \le 1$
as~$r\to\infty$ by~\eqref{eq-logduol}.  Repeatedly remove any vertex~$u$ whose current degree is less
than~$R/\gam_u$.  By~\eqref{eq:molloy,local,parameters},
\[
    \frac{R}{\gam_u}\le D_u\gam_u<|L(u)|,
\]
so every removed vertex can be coloured greedily in reverse order.

It remains to colour the graph left by this procedure.  Every remaining
vertex satisfies condition~\eqref{eq:ysmall}.  Condition~\eqref{eq:elllarge}
follows directly from the definition of~$\ell$.  Moreover, if
$F\subseteq G[N(u)]$ has $y\ge\ell/8$ vertices, then $F$ is edgeless and
\[
    \log Z_F(\lam)
    =y\log(1+\lam)
    \ge5r\log(16\Delta^3)\log(1+1/r)
    \ge\log(8\Delta^4)
\]
for all sufficiently large~$\Delta$, proving~\eqref{eq:Zlarge}.
The chromatic local occupancy condition is inherited by the remaining
induced subgraph, and its current degrees are at most~$D_u$.
Thus~\eqref{eq:molloy,local,parameters} and
Theorem~\ref{thm:main,chic} give an $\sH$-colouring of that subgraph.
Greedily restoring the removed vertices completes the colouring of~$G$.
\end{proof}

\noindent
Next we indicate how this result matches or improves upon previous work about local list/correspondence colourings of triangle-free graphs~\cite{DJKP20,BKNP22}. 

In particular, for some (small) fixed~$\eps'>0$, and a target list size bound of
\[
(1+\eps')\frac{d_u}{\log d_u},
\]
taking~$\eps''$ arbitrarily close to but less than~$\eps'$, and then~$\eps>0$ arbitrarily small so that $(1+\eps)(1+\eps'')< 1+\eps'$, Corollary~\ref{cor:molloy,local} implies that it suffices to maintain that
\[
\log\left(\frac{d_u}{\log\Delta}\right) \ge \frac{1}{1+\eps''}\log d_u
\]
whenever $d_u\ge \delta$, which is implied by 
$\delta \ge {(\log\Delta)}^{(1+\eps'')/\eps''}$.
When $\eps'$ is small, this roughly matches~\cite[Thm.~1]{DJKP20} which required $\delta\ge{(192\log\Delta)}^{2/\eps'}$ for the same bound.
When $\eps'>1$, the bound on $\delta$ can be taken as $\delta \ge{(\log\Delta)}^2$ and this improves upon the asymptotic leading factor in~\cite[Thm.~1.12]{BKNP22}, which was~$4\log 2$ with the same condition on~$\delta$.

We have just illustrated how the form of Corollary~\ref{cor:molloy,local} allows some limited trade-off between lowering the leading constant in the bound and lowering the minimum list size condition. In fact, Theorem~\ref{thm:main,chic} allows for an analogous trade-off more generally in locally sparse graph classes.

One may naturally wonder though if such a minimum list size condition is truly required. Or perhaps it is a mere technicality that is by-product to our methods?
On the contrary, as already noted~\cite{DJKP20}, \emph{some} such assumption is strictly necessarily for local list colouring results, even within the class of bipartite graphs. The following is an elementary inductive construction.

\begin{proposition}[\cite{DJKP20}]\label{prop:DJKP18p}
For any $\delta$, there is a bipartite graph of minimum degree $\delta$ and
maximum degree $\exp^{\delta-1}(\delta)$ (so a tower of exponentials of height
$\delta-1$) that is not $L$-colourable for some list assignment $L \colon V(G)\to
2^{\ZZ^+}$ satisfying \(|L(u)| \ge d_u/\log d_u\) for all~$u\in V(G)$.
\end{proposition}

\noindent
An interesting challenge is to find as a function of~$\Delta$ the best sufficient minimum list size condition for local list/correspondence colouring of triangle-free graphs with a target list size of the form, say,
\[
2\frac{d_u}{\log d_u}
\]
for every vertex~$u$.
By the remarks after Corollary~\ref{cor:molloy,local}, it is at most around ${(\log\Delta)}^2$, while (a light adaptation of) the construction of Proposition~\ref{prop:DJKP18p} shows that it is at least some inverse of a tall tower of exponentials.

By contrast, for all of our local colouring results involving~$\chi_f$, we need no analogous assumption, although we
wonder if this is compensated by the alternative assumption that the colour lists are nested.
Relatedly, we would be curious to know if the minimum list size condition in Theorem~\ref{thm:molloy,local}, say, could be avoided or significantly reduced.


\subsection{Degree bounds}\label{sub:degree}

We humbly point out the limitation in our framework
that we only treat graphs of given maximum degree. Shearer~\cite{She83} (improving on~\cite{AKS81}) proved a stronger form of Theorem~\ref{thm:shearer} with degeneracy~$\delta^*$ in the place of~$\Delta$ (cf.~\eqref{eq:params}).
Unfortunately, it is impossible to hope for a similar strengthening for example in Theorem~\ref{thm:molloy}, due to a well-known construction~\cite{AKS99}. On the other hand, walking back a little along~\eqref{eq:params}, it might yet be possible for the occupancy fraction bound in Theorem~\ref{thm:main,occfrac} or the fractional chromatic number bound in Theorem~\ref{thm:main,chif} to hold with degeneracy instead of maximum degree. If so, in the latter case it would confirm the conjectures of Harris~\cite{Har19} and of Esperet, Thomass\'{e} and the second author~\cite{EKT19}. 
(See `Note added' for an update on these remarks.)

\subsection{Algorithms}

As we alluded to in earlier sections, it is natural to wonder about the computational effectiveness of the framework we developed: does it lead to efficient
randomised algorithms to colour sparse graphs?  This topic touches upon algorithmic aspects of the Lovász
local lemma, but also raises the question of how to efficiently sample according to the hard-core
distribution at some given fugacity.  As it turns out, the sufficient conditions provided throughout this
work to colour sparse graphs do not automatically yield randomised polynomial-time algorithms for
correspondence colouring. But the situation is better as soon as one restricts to list colouring. In a
companion work~\cite{DKPS20}, we provide a study of these issues and consider, as an application, graphs with
few short cycles. For any fixed integer~$k\ge3$ and fixed~$0<\rho<\varepsilon$, we give a randomised polynomial-time algorithm for list
colouring graphs of maximum degree~$\Delta$ in which each vertex belongs to at most~$t$ copies of a
$k$-cycle, where~$1/2\le t\le \Delta^{2\rho/(1+2\rho)}/{(\log\Delta)}^2$. To this end, each
list is required to have size at least~$(1+\varepsilon)\Delta/\log(\Delta/\sqrt{t})$, building on a
result of Achlioptas, Iliopoulos and Sinclair~\cite{AIS19}, who
obtained this statement restricted to~$k=3$.
When~$1/2\le t<1$, integrality of the cycle count makes the graph~$C_k$-free, and the same bound follows from the cycle-free case with any fixed~$\varepsilon'<\varepsilon$ once~$\Delta$ is sufficiently large.
We point out that our result is tight up to the algorithmic
barrier for colouring random graphs. On the other hand, it seems more difficult to obtain similar
statements for correspondence colouring. 
The example of complete bipartite graphs given at the end of Subsection~\ref{sub:structure} already shows that list and correspondence chromatic numbers can be very different, but it would be rather interesting to either quantify some computational differences between list and correspondence colouring, or devise algorithmic techniques that show the absence of such differences. 

\subsection{Continuous models}

Our work here continues a line of research started in~\cite{DJPR17,DJPR18} based on applying some analysis of the hard-core model to extremal combinatorics. 
It is important to note that a continuous analogue of the model, known as the \emph{hard-sphere model} is also well-studied. 
Continuing the theme from~\cite{DJPR17,DJPR18} in a different direction, Jenssen, Joos, and Perkins~\cite{JJP18,JJP19} analysed the occupancy fraction of both the hard sphere model and the hard cap model using ideas that resemble the local occupancy and local sparsity concepts we discuss here.
There are significant issues with bringing the more complex parts of our framework (that deal with colourings) to bear on these models and their underlying infinite geometric graphs, as independent sets in e.g.\ the hard sphere model have measure zero in Euclidean space. 
It would be interesting to apply our framework to related geometric models where these issues are not present.

\section*{Acknowledgements}

We would like to thank Alistair Sinclair, Fotis Iliopoulos, and Charlie Carlson for insightful discussions. 
This work was completed while ED was affiliated with University of Colorado Boulder, RJK with Radboud University, FP with Radboud University and Universit\'e de Lorraine.
We are grateful to the anonymous referees for their careful reading and helpful comments. 
Importantly, one of the referees pressed for a clarification on the negative correlation assumed by the original concentration argument in the proof of Theorem~\ref{thm:main,chic}.
From this, we discovered a counterexample to the requisite negative correlation as described above at the end of Subsection~\ref{sub:tfree}. This led us to devise the concentration argument in terms of chromatic local occupancy,
which is a natural evolution of concepts employed earlier~\cite{Mol19,Ber19,BKNP22} that is faithful to our
original approach.

\section*{Note added}

A version of this work was posted to a public preprint repository during a momentous period, in March~2020.
In the years since then, there have been many interesting related developments.
Here we mention what we consider to be the most relevant of these, as of November~2026.
\begin{itemize}
\item Hurley and Pirot~\cite{HuPi23} improved upon Theorem~\ref{thm:sparsenbhds,chi}, with an improvement of the factor~$\log\sqrt{f}$ in the denominator to~$\log\min\{f,\Delta\}$.
\item Campos, Griffiths, Morris, and Sahasrabudhe~\cite{CGMS23+} improved the exponent in the upper bounds for diagonal (and near-diagonal) Ramsey numbers. This was later improved in~\cite{GNNW24+}. Any further refinements influence the hypothesis in our Proposition~\ref{prop:erdosszekeres,kahn}. 
\item
With respect to Conjecture~\ref{conj:CJKP}, Illingworth and the first author~\cite{DaIl22} confirmed it up to the leading multiplicative constant, while Martinsson and Steiner~\cite{MaSt24+} recently confirmed the fractional version.
\item
With respect to Harris's problem on the Hall ratio of triangle-free graphs~\cite{Har19}, Steiner~\cite{Ste25} reports that Davies, Hatzel and Yepremyan constructed graphs with Hall ratio at most~$3-\delta$, for some~$\delta>0$, and arbitrarily large fractional chromatic number. Since any graph containing a triangle has Hall ratio at least~$3$, these graphs are triangle-free.
\item
With respect to the remarks in Subsection~\ref{sub:degree}, Martinsson and Steiner~\cite{MaSt24+} proved Harris's conjecture on the fractional chromatic number of triangle-free graphs of given degeneracy~\cite{Har19}, which implies the conjecture in~\cite{EKT19}.
Moreover, Buys, van den Heuvel and the second author~\cite{BHK25+} proved a sharp bound on the independence polynomial of triangle-free graphs of given average degree, which lends more support to the suggestion that some results within our framework might still hold under more relaxed degree conditions.
They also conjectured a corresponding extension of Shearer's bound to the occupancy fraction; the first author, Sandhu and Tan~\cite{DST25+} subsequently made progress on this conjecture for a restricted range of small fugacities.
\item
Campos, Jenssen, Michelen, and Sahasrabudhe~\cite{CJMS25+} improved the leading constant in the lower bound for $R(3,k)$ to~$1/3$.
This was subsequently improved to~$1/2$ by Hefty, Horn, King, and Pfender~\cite{HHKP25+}.
\item
A counterexample to Conjecture~\ref{conj:DJPR18} was found by an automated LLM-based mathematical-discovery system~\cite{ZZATW26}.
\item
On 6 October 2026, OpenAI made public a batch of claimed solutions to hundreds of major open problems in mathematics, one of which the Ajtai--Erd\H{o}s--Koml\'{o}s--Szemer\'{e}di conjecture~\cite{Ope26+}, that were obtained through the use of an OpenAI-internal system.
The method does not establish a local occupancy property which we missed and it remains open whether that could be shown.
The second author~\cite{Kan26+} used commercially available systems to convert the OpenAI output to a cleaner one with a stronger bound, in a more general form. There he posed a conjecture to strengthen the Ajtai--Erd\H{o}s--Koml\'{o}s--Szemer\'{e}di conjecture, that for any graph~$G$ on $n$ vertices of clique number $\omega\ge3$ and average degree~$d\ge 2$, the independence number of~$G$ satisfies
\[
\alpha(G) \ge \Omega\left( \frac{n}{d} \max\left\{1,\frac{\log d}{\log(e+\omega/\log d)}\right\} \right).
\]
\end{itemize}

\section*{Tool and computational resource disclosure}
All of the main mathematical contributions in this work are our own, unassisted by any generative AI tools, including the repair for Theorem~\ref{thm:main,chic}.
In the latter stages of finalising the revisions, during the summer, the second author unilaterally deployed a proprietary closed-weight generative AI system to assist in the tedious optimisation repairs necessary to establish chromatic local occupancy at values that preserved our headline claims.

\bibliographystyle{habbrv}
\bibliography{hcm_link}

\end{document}